%% file: The_Prasad_conjecture_for_GSp_4_period.tex
\documentclass{amsart}
\usepackage{amssymb
,amsthm
,amsmath,enumerate
,amscd
,mathtools,
hyperref,esvect,
}
\usepackage[shortlabels]{enumitem}
\usepackage[all]{xy}
\usepackage[top=30truemm,bottom=30truemm,left=25truemm,right=25truemm]{geometry}

\newcommand{\A}{\mathbb{A}}

\newcommand{\Ind}{\mathrm{Ind}}
\newcommand{\Res}{\mathrm{Res}}

\newcommand{\Hom}{\mathrm{Hom}}
\newcommand{\Gal}{\mathrm{Gal}}

\newcommand{\Sp}{\mathrm{Sp}}
\newcommand{\SL}{\mathrm{SL}}
\newcommand{\GL}{\mathrm{GL}}
\newcommand{\Mp}{\mathrm{Mp}}
\newcommand{\Oo}{\mathrm{O}}
\newcommand{\SO}{\mathrm{SO}}

\newcommand{\GSO}{\mathrm{GSO}}
\newcommand{\GSp}{\mathrm{GSp}}
\newcommand{\PGSp}{\mathrm{PGSp}}
\newcommand{\GO}{\mathrm{GO} }
\newcommand{\GU}{\mathrm{GU}}

\newcommand{\Ext}{\mathrm{Ext}}

\newcommand{\PGU}{\mathrm{PGU}}
\newcommand{\PGL}{\mathrm{PGL}}

\newcommand{\resp}{resp.~}

\renewcommand{\1}{{\bf 1}}

\newtheorem{theorem}{Theorem}[section]
\newtheorem{proposition}[theorem]{Proposition}
\newtheorem{corollary}[theorem]{Corollary}

\newtheorem{thm}{Theorem}[subsection]
\newtheorem{lem}[thm]{Lemma}
\newtheorem{prop}[thm]{Proposition}
\newtheorem{coro}[thm]{Corollary}

\theoremstyle{remark}
\newtheorem{remark}[theorem]{Remark}

\newtheorem{rem}[thm]{Remark}

\theoremstyle{definition}
\newtheorem{defn}[thm]{Definition}

\numberwithin{equation}{section}

\makeatletter
\def\iddots{\mathinner{\mkern1mu\raise\p@
	\hbox{.}\mkern2mu\raise4\p@\hbox{.}\mkern2mu
	\raise7\p@\vbox{\kern7\p@\hbox{.}}\mkern1mu}}
\def\adots{\mathinner{\mkern2mu\raise\p@\hbox{.}
 \mkern2mu\raise4\p@\hbox{.}\mkern1mu
 \raise7\p@\vbox{\kern7\p@\hbox{.}}\mkern1mu}}
\makeatother

\allowdisplaybreaks

\begin{document}
	\title{The Prasad conjectures for $\GSp_4$ and $\PGSp_4$}
	\author{Hengfei Lu}
	\date{}
	\address{Department of Mathematics, Weizmann Institute of Science, 234 Herzl St. P.O.B. 26, Rehovot 7610001, Israel}
	\email{hengfei.lu@weizmann.ac.il}
	


\begin{abstract}
In this paper, we use the theta correspondence between $\GSp_4(E)$ and $\GO(V)$ to study the $\GSp_4$-distinction problems over  a quadratic extension $E/F$ of nonarchimedean local fields of characteristic $0$.  With a similar strategy, we investigate the distinction problem for the pair $(\GSp_4(E),\GSp_{1,1}(F) )$, where $\GSp_{1,1}$ is the unique inner form of $\GSp_4$ defined over $F$. Then we verify the Prasad conjecture for a discrete series representation $\tau$ of $\PGSp_4(E)$. 
\end{abstract}

\maketitle

\tableofcontents

\section{Introduction}
Let $F$ be a finite field extension over $\mathbb{Q}_p$
and $E$ be  a quadratic extension over $F$ with associated  Galois group
$\Gal(E/F)=\{1,\sigma\}$ and associated quadratic character $\omega_{E/F}$ of $F^\times$. Let $W_F$ be the Weil group of $F$ and $WD_F$ be the Weil-Deligne group. Then $\omega_{E/F}$ is a quadratic character of $W_F$ with kernel $W_E.$
 Let $\mathbf{G}$ be a connected reductive group defined over $F$ and $\mathbf{G}(F)$ (resp. $\mathbf{G}(E)$) be the $F$-rational (resp. $E$-rational) points. 
Given a smooth representation $\tau$ of $\mathbf{G}(E)$ and a character $\chi$ of $\mathbf{G}(F)$, we say that $\tau$ is $(\mathbf{G}(F),\chi)$-distinguished or has a nonzero $(\mathbf{G}(F),\chi)$-period if
\[\Hom_{\mathbf{G}(F)}(\tau,\chi)\neq0. \]
If $\chi$ is the trivial character, then $\tau$ is called $\mathbf{G}(F)$-distinguished.
There exists a rich literature, such as \cite{beuzart2017distinguished,flicker1991ondist,gan22arithmeticity,hengfei2016new,matringe2009distinction,prasad2015arelative}, trying to classify all $\mathbf{G}(F)$-distinguished representations  of $\mathbf{G}(E)$. The method often used to study the distinction problems  is the relative trace formula, such as \cite{beuzart2017distinguished,flicker1994quaternionic}, which is  powerful especially for the global period problems. This paper focuses on the local period problems for $\mathbf{G}=\GSp_4,\PGSp_4$ and their inner forms. The main tool in this paper is the local theta correspondence appearing in \cite{gan2011theta,Kudla1992,yamana2011deg}.
\par
Let $V$ be the unique non-split quaternion algebra $D_E$ with quadratic form $N_{D_E}$ over $E,$
or  the split  $6$-dimensional  quadratic space $\mathbb{H}_E^3$ over $E.$
Then 
\[\GSO(V)\cong \begin{cases}\GSO_{4,0}(E)=D_E^\times(E)\times D_E^\times(E)/\{(t,t^{-1})\},&\mbox{  if  }V=D_E,\\
\GSO_{3,3}(E)=\GL_4(E)\times E^\times/\{(t^{-1},t^2) \},&\mbox{  if  }V=\mathbb{H}_E^3,
\end{cases} \]
and any irreducible representation of $\GSO(V)$ must be of the form
\begin{itemize}
	\item $\pi_1\boxtimes\pi_2$ with $\omega_{\pi_1}=\omega_{\pi_2}$ if $V=D_E;$
	\item $\Pi\boxtimes\mu$ with $\omega_\Pi=\mu^2$ if $V=\mathbb{H}_E^3.$
\end{itemize}
Here for each $i$, $\pi_i$  is an irreducible representation of $D_E^\times(E)$.

In \cite{gan2011theta}, Gan-Takeda have proved that
any irreducible representation $\tau$ of
$\GSp_4(E)$ falls into one of the following two disjoint families of representations:
\begin{itemize}
	\item $\tau=\theta(\pi_1\boxtimes\pi_2)$ with $\omega_{\pi_1}=\omega_{\pi_2};$
	\item $\tau=\theta(\Pi\boxtimes\mu)$ with $\mu=\omega_\tau$
	and $\omega_\Pi=\mu^2.$
\end{itemize}

In order to use the see-saw identities, we need to study the big theta lift to $\GO(V)$ of a generic representation $\tau$ of $\GSp_4(E)$. In fact, we have studied the general (almost equal rank) case for the irreducibility of big theta lifts to $\GO_{n+1,n+1}(F)$ of a generic representation  of $\GSp_{2n}(F)$ in \S\ref{subsect:big}.
After computing the big theta lifts following \cite{gan2014formal,gan2011theta},
we use the local theta correspondences between $\GSp_4(E)$ and $\GSO(V)$ and the see-saw identities to
discuss $\GSp_4$-period problems, by transferring the period problem for $\GSp_4$ to various analogous period problems for $\GL_2, \GL_4$ and their various forms (not necessarily inner). Then we obtain the following results:
\begin{theorem}[Theorem \ref{localgspperiod}]\label{maingsp(4)theorem}
	Suppose that $\tau$ is an irreducible representation of $\GSp_4(E)$ with a central character $\omega_\tau$ and $\omega_\tau|_{F^\times}=\mathbf{1}$. 
	\begin{enumerate}[(i)]
		\item If $\tau=\theta(\Sigma)$ is an irreducible representation of $\GSp_4(E),$ where $\Sigma$ is an irreducible representation of $\GO_{4,0}(E),$ 
		then the representation $\tau$ is not $\GSp_4(F)$-distinguished.
		\item If $\tau=\theta(\pi_1\boxtimes\pi_2),$ where  $\pi_1\boxtimes\pi_2$ is a generic representation of $\GSO_{2,2}(E),$
		then
		\[\dim \Hom_{\GSp_4(F)}(\tau,\mathbb{C})=\begin{cases}
		2,&\mbox{ if }\pi_i\ncong\pi_0 \mbox{  are both }\GL_2(F)\mbox{-distinguished};\\
		1,&\mbox{ if }\pi_1\ncong\pi_2\mbox{ but }\pi_1^\sigma\cong\pi_2^\vee;\\
		1,&\mbox{ if }\pi_1\cong\pi_2 \mbox{ is }\GL_2(F)\mbox{-distinguished but not }(\GL_2(F),\omega_{E/F})\mbox{-distinguished};\\
		1,&\mbox{ if }\pi_2 \mbox{ is }\GL_2(F)\mbox{-distinguished and }\pi_1\cong\pi_0;
		\\
		0,&\mbox{ the other cases.}
		\end{cases}\]
		Here 
		$\pi_0=\pi(\chi_1,\chi_2)$  with $\chi_1\neq\chi_2,\chi_1|_{F^\times}=\chi_2|_{F^\times}=\mathbf{1}$ is a principal series representation of $\GL_2(F)$.\\
		Note that these conditions are mutually exclusive.
		\item Assume that $\tau$ is not in case (i) or (ii) and that $\tau=\theta(\Pi\boxtimes\chi)$ is generic, where $\Pi\boxtimes\chi$ is a representation of $\GSO_{3,3}(E)$. 		Then
	\[\dim \Hom_{\GSp_4(F) }(\tau,\mathbb{C})=\begin{cases}
	1,&\mbox{ if }\Pi\mbox{ is }\GL_4(F)\mbox{-distinguished;}\\
	0,&\mbox{ otherwise.}
	\end{cases} \]
	\end{enumerate}
\end{theorem}
Then we can verify the Prasad conjecture for $\GSp_4$ in \S\ref{subsect:GSp(4)conj}. More precisely,
 let $G_0$ be a quasi-split group defined over $F$ such that
\[{}^LG_0=\GSp_4(\mathbb{C})\rtimes\Gal(E/F), \]
where the nontrivial element $\sigma\in\Gal(E/F)$ acts on $\GSp_4(\mathbb{C})$ by
\[\sigma(g)=\text{sim}(g)^{-1}\cdot g. \]
Here sim$:\GSp_4(\mathbb{C})\longrightarrow\mathbb{C}^\times$ is the similitude character.
\begin{theorem}[The Prasad conjecture for $\GSp_4$]\label{thm1.2}
	Let $\tau$ be an irreducible smooth representation of $\GSp_4(E)$ with enhanced Langlands parameter $(\phi_{\tau},\lambda_\tau)$. Assume that the $L$-packet $\Pi_{\phi_{\tau}}$ is generic. Then
	\[\dim\Hom_{\GSp_4(F)}(\tau,\omega_{E/F})=\begin{cases}
	|F(\phi_{\tau})|,&\mbox{ if }\tau\mbox{ is generic},\\
	0,&\mbox{ otherwise},
	\end{cases}  \]  
	where the finite set $F(\phi_\tau )$ is given by
	$$F(\phi_{\tau})=\{\tilde{\phi}:WD_F\longrightarrow {}^LG_0|~\tilde{\phi}|_{WD_E}=\phi_{\tau}  \}$$
	and $|F(\phi_{\tau})|$ denotes the cardinality of the set $F(\phi_{\tau})$.
\end{theorem}
We have proved  analogous results for the inner form.
Let $ D$ be the $4$-dimensional quaternion division algebra of $F$. In a similar way,
we study the period problem for the inner form $\GU_2(D)=\GSp_{1,1}$, i.e. try to figure out the multiplicity
\[\dim \Hom_{\GSp_{1,1}(F)}(\tau,\mathbb{C}). \]
We will not state the results of the inner form case in the introduction; the precise results can be found in Theorem \ref{innerformperiod}.
\par
Combining Theorem \ref{maingsp(4)theorem} and its analog for inner forms, we can verify the  conjecture of Dipendra Prasad \cite[Conjecture 2]{prasad2015arelative}  for $\PGSp_4$.
Given a quasi-split reductive group $\mathbf{G}$ defined over  $F$ and a quadratic extension $E/F,$ assuming the Langlands-Vogan conjectures for $\mathbf{G}$, Dipendra Prasad \cite{prasad2015arelative} uses the information from the Galois side to give a formula for the individual multiplicity
\[\dim\Hom_{G_\alpha(F)}(\tau,\chi_\mathbf{G} ), \]
where
\begin{itemize}
	\item  $\tau$ is an irreducible discrete series representation of $\mathbf{G}(E);$
\item  $\chi_\mathbf{G}$ is a quadratic character of $\mathbf{G}(F)$ depending on $\mathbf{G}$ and $E;$ 
\item $G_\alpha$ is any pure inner form of $\mathbf{G}$ defined over $F$ satisfying $G_\alpha(E)=\mathbf{G}(E).$
\end{itemize}
\par
In \S \ref{secpgsp}, we will focus on the case $\mathbf{G}=\PGSp_4$. Then $\chi_\mathbf{G}=\omega_{E/F}$ and $H^1(F,\mathbf{G})=\{\PGSp_4,\PGU_2(D) \}.$
More precisely,
 given a discrete series representation $\tau$ of $\PGSp_4(E)$ with the enhanced L-parameter 
$(\phi_\tau,\lambda)$ (called the Langlands-Vogan parameter), where $\lambda$ is a character of the component group $\pi_0(Z(\phi_\tau)),$ set $$F(\phi_{\tau})=\{\tilde{\phi}:WD_F\longrightarrow \Sp_4(\mathbb{C})|~\tilde{\phi}|_{WD_E}=\phi_\tau \} .$$
Up to the twisting by the quadratic character $\omega_{E/F},$ there are several orbits in $F(\phi_\tau),$ denoted by $\sqcup_{i=1}^r\mathcal{O}(\tilde{\phi}_i).$ Each orbit $\mathcal{O}(\tilde{\phi}_i)$ corresponds to a unique subset $\mathcal{C}_i$ of $H^1(W_F,\mathbf{G})$. (See \S \ref{subsect:prasad}  for more details.)
\begin{theorem}\label{prasadfordisc}
	Let the notation be as above. Given a discrete series representation $\tau$ of $\PGSp_4(E),$ we have
	\begin{equation}\label{equfordist}
		\dim\Hom_{G_\alpha(F)}(\tau,\omega_{E/F})=\sum_{i=1}^rm(\lambda,\tilde{\phi}_i)\mathbf{1}_{\mathcal{C}_i}(G_\alpha)/d_0(\tilde{\phi}_i),
	\end{equation}	where
	\begin{itemize}
		\item $\mathbf{1}_{\mathcal{C}_i}$ is the characteristic function of the set $\mathcal{C}_i;$
		\item  $m(\lambda,\tilde{\phi})$ is the multiplicity for the trivial representation contained in the restricted representation $\lambda|_{\pi_0(Z(\tilde{\phi}))}$;
		\item  $d_0(\tilde{\phi})=|Coker\{\pi_0(Z(\tilde{\phi}))\rightarrow\pi_0(Z(\phi_\tau))^{\Gal(E/F)} \}|$, where $|-|$ denotes its cardinality.
	\end{itemize}
\end{theorem}
\begin{remark} We would like to highlight the fact that
	Theorem \ref{prasadfordisc} provides the first example in the Prasad conjecture that the square-integrable representation $\tau$ may be nongeneric and so $\tau$ is not $\PGSp_4(F)$-distinguished (see Theorem \ref{innerformperiod}) but $\tau$ contains a nonzero period for the pure inner form $\PGSp_{1,1}(F)$. It is different from the case $\mathbf{G}=\PGL_2$ that if a representation $\pi$ of $\PGL_2(E)$ is $\mathrm{PD}^\times(F)$-distinguished, then $\pi$ must be $\PGL_2(F)$-distinguished. (See Lemma \ref{GL:period}.)
\end{remark}
\par
In fact, we have shown that the equality \eqref{equfordist} holds for almost all generic representations in \S\ref{secpgsp}, except that the Langlands parameter
	 $\phi_\tau=2\chi_F|_{W_E}\oplus\phi_2$ with $\phi_2$ conjugate-symplectic and $\chi_F^2=\omega_{E/F}.$
However,  there is a weaker version of the Prasad conjecture which determines the sum of $\dim\Hom_{G_\alpha(F)}(\tau,\chi_\mathbf{G})$
as $G_\alpha$ runs over all pure inner forms of $\mathbf{G}$ satisfying $G_\alpha(E)=\mathbf{G}(E)$. It involves the degree of the base change map $$\Phi:\Hom(WD_F,\Sp_4(\mathbb{C}))\longrightarrow\Hom(WD_E,\Sp_4(\mathbb{C}))$$ for the exception case, i.e., the following identity
\begin{equation}\label{equaforsum}
\dim\Hom_{\PGSp_4(F)}(\tau,\omega_{E/F})+\dim\Hom_{\PGSp_{1,1}(F) }(\tau,\omega_{E/F})=\sum_{\tilde{\phi}\in F(\phi_\tau) }m(\lambda,\tilde{\phi})\frac{\deg\Phi(\tilde{\phi})}{d_0(\tilde{\phi})}\end{equation}
when the $L$-parcket $\Pi_{\phi_{\tau}}$ is generic, 
which is the original identity formulated by Dipendra Prasad in \cite{prasad2015arelative}. 
\par
There is a brief introduction to the proof of Theorem \ref{prasadfordisc}. After introducing the local theta correspondence between quaternionic unitary groups following \cite{yamana2011deg}, we use the Morita equivalence $\GU_2(R)=\GSp_{1,1}(E)\cong\GSp_4(E),$  where $R\cong Mat_{2,2}(E)$ is the split quaternion algebra over $E,$
to embed the group $\GSp_{1,1}(F)$ into
$\GSp_4(E).$ Then one can use the see-saw identity to transfer the inner form $\GSp_{1,1}$-period problem to $\GO^\ast_{3,0}$ or $\GO_{1,1}^\ast$ side,  which are closely related to  $\GL_n$-period problems. But we need to be very careful when we use the see-saw identity for a pair of quaternionic unitary groups. (See Remark \ref{counterforseesaw}.) Once the see-saw identity for the quaternionic unitary groups has been set up, the rest of the proof for the inner form case is similar to the  case for $\GSp_4$-period. Then we obtain the results for the distinction problems for the automorphic side. For the Galois side, i.e., the right hand side of \eqref{equaforsum}, it will be checked  case by case in \S\ref{secpgsp}.
\begin{remark}  
	 Rapha\"{e}l Beuzart-Plessis \cite[Theorem 1]{beuzart2017distinguished} used the local trace formula to deal with the distinction problems for the Galois pair $(G'(E),G'(F))$ for the stable square-integrable representations, where $G'$ is an inner form of $\mathbf{G}$ defined over $F$, which generalizes \cite[Theorem C]{prasad1992gl(2)}.
\end{remark}
The paper is organized as follows.
In \S $2$, we set up the notation about the local theta correspondence. In \S\ref{sect:generic}, we will study the irreducibility for the big theta lift of a generic representation in the almost equal rank case, which generalizes the results of Gan-Ichino \cite[Proposition C.4]{gan2014formal} for the tempered representations.
In \S\ref{sect:GSp(4)},   we will study the distinction problems for $\GSp_4$ over a quadratic  extension $E/F.$ The proof of Theorem \ref{maingsp(4)theorem} will be given in \S\ref{subsect:proofof1.1}. The analogous results for the inner form $\GSp_{1,1}$ will be given in \S\ref{sect:GSp(1,1)}. In \S\ref{subsect:prasad}, we will introduce the Prasad conjecture for a reductive quasi-split group $\mathbf{G}$ defined over $F$. Then we will verify the Prasad conjecture for $\GSp_4$ in \S\ref{subsect:GSp(4)conj}.
Finally, the proof of Theorem \ref{prasadfordisc} will be given in \S\ref{secpgsp}.
\subsection*{Acknowledgments} This paper contains part of the author's Ph.D. thesis \cite{hengfei2017}.
He is grateful to   Wee Teck Gan for his guidance and numerous discussions when he was studying in Singapore. He would like to thank  Dipendra Prasad for proposing this conjecture and fruitful discussions. He wants to thank Dmitry Gourevitch and Lei Zhang for useful comments as well.

\input{localtheta}

\section{The irreducibility of the big theta lift}\label{sect:generic}
Let $\tau$ be an irreducible representation of $\Sp_{2n}(F)$.
Thanks to  \cite[Proposition C.4]{gan2014formal}, Gan-Ichino
show that the big theta lift $\Theta_{2n+2}^+(\tau)$ to $\Oo_{n+1,n+1}(F)$ (called the almost equal rank case) is irreducible if $\tau$ is tempered. We will use the generalized standard module  \cite[Theorem 3.2]{heiermann2016vogan} to study the case when $\Pi_{\phi_{\tau}}$ is generic. (See Theorem \ref{genericthetalift}.)

In \S\ref{subsect:big}, we mainly study the big theta lift to the split group $\Oo_{n+1,n+1}(F)$ from a representation $\tau$ of $\Sp_{2n}(F)$ when the associated $L$-packet $\Pi_{\phi_\tau}$ is generic. Then  we will focus on the computation for $n=2$. 
\subsection{Notation}
Let us introduce the notation used in this section.
\begin{itemize}
	\item $|-|_F$ (resp. $|-|_E$) is an absolute value defined on $F$ (resp. $E$).
	\item $P_{\vv{n}}$ (resp. $Q_{\vv{n}}$) is a parabolic subgroup of $\Sp_{2n}$ (resp. $\Oo_{n+1,n+1}$) defined over $F$.
	\item $\phi_{\tau}$ is the Langlands parameter or $L$-parameter of $\tau$ and $\phi^\vee_\tau$ is the dual parameter of $\phi_\tau$.
	\item $\tau^\vee$ is the contragredient representation of $\tau$.
	\item $\Pi_{\phi_{\tau}}$ is the $L$-packet containing $\tau$.
	\item $\mathcal{W}_r$ is the symplectic vector space over $E$ of dimension $2r$.
	\item $Z$ is a line in $\mathcal{W}_2$ and $Y$ is a maximal isotropic subspace in $\mathcal{W}_2$.
	\item $Q(Z)$ (resp. $P(Y)$) is the Klingen (resp. Siegel) parabolic subgroup of $\GSp_4(E)=\GSp(\mathcal{W}_2)$.
	\item $B$ (resp. $B_0$) is the Borel subgroup of $\GSp_4(E)$ (resp. $\GL_4(E)$).
	\item $P$ is the parabolic subgroup of $\GL_4(E)$ with Levi component $\GL_2(E)\times\GL_2(E)$. 
	\item $\Theta^+_{2r}(\tau)$ (resp. $\Theta_{6}(\tau)$) is the big theta lift to $\GO_{r,r}(E)$ (resp. $\GSO_{3,3}(E)$) of $\tau$ of $\GSp_4(E)$.
	\item $\theta_6^+(\tau)$ (resp. $\theta_6(\tau)$) is the small theta lift to $\GO_{3,3}(E)$ (resp. $\GSO_{3,3}(E)$) of $\tau$ of $\GSp_4(E)$.
\end{itemize}

\subsection{The standard module conjecture} Let $\mathbf{G}$ be a quasi-split reductive group defined over $F$. Fix a Borel subgroup $\mathbf{B}=\mathbf{T}\mathbf{U}$ of $\mathbf{G}$.  Let $\pi$ be an irreducible smooth representation of $\mathbf{G}(F)$. If there exists a nondegenerate character $\psi_U$ of $\mathbf{U}(F)$ such that $\Hom_{\mathbf{U}(F)}(\pi,\psi_U)\neq0$, then we say $\pi$ is $\psi_U$-generic or generic. If the $L$-packet $\Pi_{\phi_{\pi}}$ contains a generic representation, then
we call $\Pi_{\phi_{\pi}}$ a generic $L$-packet. Let $\mathbf{P}=\mathbf{M}\mathbf{N}$ be a standard parabolic subgroup of $\mathbf{G}$.
Suppose that there exists a generic tempered representation $\rho$ of $\mathbf{M}(F)$ such that
$\pi$ is isomorphic to the Langlands quotient $J(\rho,\chi)$, where $\chi$ is a character of $\mathbf{M}(F)$ and lies in the positive Weyl chamber with respect to $\mathbf{P}(F)$. (See \cite[Page 777]{Heiermann2013standard} for more details.)
\begin{theorem}[the standard module conjecture]\label{stand}
	If $\pi=J(\rho,\chi)$ is a generic representation of $\mathbf{G}(F)$, then $\Ind_{\mathbf{P}(F)}^{\mathbf{G}(F)}(\rho\otimes\chi )$(normalized induction) is irreducible.
	Moreover, for any irreducible representation $\rho'$ of $\mathbf{M}(F)$ lying inside the $L$-packet $\Pi_{\phi_\rho}$, $\Ind_{\mathbf{P}(F)}^{\mathbf{G}(F)}(\rho'\otimes\chi )$ is irreducible.
\end{theorem}
Heiermann-Mui\'{c} \cite{Heiermann2013standard} proved the standard module conjecture. Later Volker Heiermann proved its generalized version in \cite[Theorem 3.2]{heiermann2016vogan}, i.e., the "moreover" part of Theorem \ref{stand}. The following subsection will focus on the cases $\mathbf{G}=\Sp_{2n}$ and $\mathbf{G}=\Oo_{n+1,n+1}$.

\subsection{The theta lift from $\Sp_{2n}(F)$ to $\Oo_{n+1,n+1}(F)$}\label{subsect:big}
Suppose that $\tau$ is a generic irreducible admissible representation of $\Sp_{2n}(F)$. Assume that there exists a parabolic subgroup $P_{\vv{n}}=M_{\vv{n}}N_{\vv{n}}$ of $\Sp_{2n}$ and an irreducible  representation $\pi_1\otimes\pi_2\otimes\cdots\otimes\pi_r\otimes\tau_0$ of $M_{\vv{n}}(F)\cong \GL_{n_1}(F)\times\cdots\times\GL_{n_r}(F)\times\Sp_{2n_0}(F) $  $(n_1+n_2+\cdots+n_r+n_0=n)$ such that $\tau$ is the unique irreducible quotient of the standard module
\begin{equation}\label{standard}
\Ind_{P_{\vv{n}}(F)}^{\Sp_{2n}(F)}(\pi_1|-|_F^{s_1}\otimes\cdots\otimes \pi_r|-|_F^{s_r}\otimes\tau_0)(\mbox{ normalized induction}), \end{equation}
where $s_1> s_2> \cdots > s_r>0$, $n\geq n_0$, each $\pi_i$ is a tempered representation of $\GL_{n_i}(F)$ and $\tau_0$ is a tempered representation of $\Sp_{2n_0}(F)$.
Moreover, the Langlands parameter $\phi_{\tau}:WD_F\longrightarrow\SO_{2n+1}(\mathbb{C})$ is given by
\[\phi_{\tau}=\phi_{\pi_1}|-|_F^{s_1}\oplus\cdots\oplus\phi_{\pi_r}|-|_F^{s_r}\oplus\phi_{\tau_0}\oplus\phi_{\pi_r}^\vee|-|_F^{-s_r}\oplus\cdots\oplus \phi_{\pi_1}^\vee|-|_F^{-s_1} \]
where each $\phi_{\pi_i}$ is the Langlands parameter of $\pi_i$ and $\phi_{\tau_0}$ is the Langlands parameter of $\tau_0$. Here we identify the characters of $F^\times$ and the characters of the Weil group $W_F$ by the local class field theory.
 Due to Theorem \ref{stand}, the generic representation $\tau$ is isomorphic to the standard module, i.e. the standard module is irreducible. Thanks to \cite[Proposition C.4]{gan2014formal}, the small theta lift $\theta_{2n+2}^+(\tau)$ is the unique irreducible quotient of the standard module
\begin{equation}\label{st:O}
\Ind_{Q_{\vv{n}}(F)}^{\Oo_{n+1,n+1}(F)}(\pi_1|-|_F^{s_1}\otimes\cdots\otimes\pi_r|-|_F^{s_r}\otimes\Theta^+_{2n_0+2}(\tau_0) ), \end{equation}
where $Q_{\vv{n}}(F)$ is the parabolic subgroup of $\Oo_{n+1,n+1}(F)$ with Levi component $\GL_{n_1}(F)\times\cdots\times\GL_{n_r}(F)\times\Oo_{n_0+1,n_0+1}(F)$. We will show that \eqref{st:O} equals to $\theta_{2n+2}^+(\tau)$ under certain conditions.

\begin{theorem}
	Let $P_{\vv{n}}$(resp. $Q_{\vv{n}}$) be a parabolic subgroup of $\Sp_{2n}$(resp. $\Oo_{n+1,n+1}$) defined as above. If $\tau$ is generic and so $\tau$ is isomorphic to the standard module (\ref{standard}), and the standard $L$-function of $\tau$ is regular at $s=1$, then $\Theta_{2n+2}^+(\tau)$ is irreducible.\label{genericthetalift}
\end{theorem}
There is another key input in the proof of Theorem \ref{genericthetalift}.
\begin{theorem} Let $\mathbf{G}$ be $\Sp_{2n}$ or $\SO_{n+1,n+1}$. Let $\pi$ be an irreducible representation of $\mathbf{G}(F)$.
	The $L$-packet $\Pi_{\phi_{\pi}}$ is generic if and only if the adjoint $L$-function $L(s,\phi_{\pi},Ad)$
	is regular at $s=1$.
\end{theorem}
\begin{proof}
	See \cite[Theorem 1.2]{liu2011sp} and \cite[Theorem 1.5]{liu2014SO(2n)}.
\end{proof}
\begin{proof}[Proof of Theorem \ref{genericthetalift}]
	We will show that $\Theta_{2n+2}^+(\tau)|_{\SO_{n+1,n+1}(F)}$ is irreducible. If $n=n_0$, then it follows from \cite[Proposition C.4]{gan2014formal}. Assume that $s_1>0$. Then there exists a surjection
	\[\xymatrix{\Ind_{Q_{\vv{n}} (F)} ^{\Oo_{n+1,n+1}(F)}(\pi_1|-|_F^{s_1}\otimes\cdots\otimes\pi_r|-|_F^{s_r}\otimes\Theta_{2n_0+2}^+(\tau_0) )\ar@{->>}[r]&\Theta^+_{2n+2}(\tau). }\]
	Due to \cite[Proposition C.4]{gan2014formal}, if $\tau_0$ is tempered, then
	$\Theta^+_{2n_0+2}(\tau_0)$ is irreducible and generic. Moreover, if $$\phi_{\tau_0}:WD_F\longrightarrow \SO_{2n_0+1}(\mathbb{C}) $$ is the Langlands parameter of $\tau_0$, then $\phi_{\theta_{2n_0+2}^+(\tau_0)}=\phi_{\tau_0}\oplus\mathbb{C}$.
	Assume that $\phi_{\tau}=\phi_0\oplus \phi_{\tau_0}\oplus\phi_0^\vee$ with $\phi_{\tau_0}$ tempered and $\phi_0=\bigoplus_i \phi_{\pi_i}|-|^{s_i}$. Then $\phi_{\theta_{2n+2}^+(\tau)}=\phi_0\oplus (\phi_{\tau_0}\oplus\mathbb{C})\oplus \phi_0^\vee$ due to \cite[Proposition C.4]{gan2014formal}.
	Observe that
	\[L(s,Ad_{\SO_{2n+2}}\circ \phi_{\theta_{2n+2}^+(\tau)})=  L(s,Ad_{\SO_{2n+1}}\circ\phi_{\tau})\cdot L(s,\phi_{\tau},Std) \]
	where $L(s,\phi_{\tau},Std)$ is the standard $L$-function of $\tau$. By  \cite[Theorem 1.2]{liu2011sp} and the assumption that $\tau$ is generic, we obtain that $L(s,Ad_{\SO_{2n+1}}\circ\phi_{\tau})$ is regular at $s=1$. So $L(s, Ad_{\SO_{2n+2}}\circ \phi_{\theta_{2n+2}^+(\tau)} )$ is regular at $s=1$.
	Thanks to \cite[Theorem 1.5]{liu2014SO(2n)}, the $L$-packet $\Pi_{\phi_{ \theta_{2n+2}^+(\tau)}}$ is generic. By the generalization of the standard module conjecture  \cite[Theorem 3.2]{heiermann2016vogan} that the standard module with a generic quotient is irreducible,
	$$\theta_{2n+2}^+(\tau)=\Theta_{2n+2}^+(\tau)=\Ind_{Q_{\vv{n}}(F)}^{\Oo_{n+1,n+1}(F)}(\pi_1|-|_F^{s_1}\otimes\cdots\otimes\pi_r|-|_F^{s_r}\otimes\Theta_{2n_0+2}^+(\tau_0) )$$ i.e., $\Theta_{2n+2}^+(\tau)$ is irreducible.
\end{proof}
\begin{remark}
	Similarly, if $\Sigma$ is a generic representation of $\Oo_{n,n}(F)$ and $L(s,\Sigma,Std)$ is regular at $s=1$, then the big theta lift $\Theta_{n}(\Sigma)$ to $\Sp_{2n}(F)$ is irreducible. However, if $\tau$ is a generic  representation of $\Sp_{2n}(F)$ and $L(s,\tau,Std)$ is regular at $s=1$, the big theta lift to nonsplit group $\Oo(V_F)$ may be reducible when $V_F$ is a $(2n+2)$-dimensional quadratic space over $F$ with nontrivial discriminant. (See \cite[Proposition 3.8(iii)]{hengfei2016new}.)
\end{remark}
\begin{remark}
	In fact, there exists an isomorphism between the characters $\lambda_{\theta_{2n+2}^+(\tau)}\cong\lambda_{\theta_{2n_0+2}^+(\tau_0)} $, where the latter one is given by Atobe-Gan \cite[Theorem 4.3]{atobe2016local} in term of the character $\lambda_{\tau_0}$, conjectured by  Prasad in \cite{prasad1993local}.
\end{remark}

\begin{corollary} Let $\Pi_{\phi_{\tau}}$ be the $L$-packet of $\Sp_{2n}(F)$ containing $\tau$. Suppose that $\Pi_{\phi_{\tau}}$ is generic.
	If the standard $L$-function $L(s,\phi_\tau,Std)$  is a factor of the adjoint $L$-function $L(s,Ad\circ\phi_{\tau} )$, then the big theta lift $\Theta_{2n+2}^+(\tau)$ to $\Oo_{n+1,n+1}(F)$ is irreducible for any $\tau\in\Pi_{\phi_{\tau}}$.
\end{corollary}
For the rest of this section, we will compute the big theta lifts between $\GSp_4(E)$ and $\GO(V)$ explicitly when $\dim_E V=4$ or $6$.
	\subsection{Representations of $\GO(V)$}
	Let $\pi_i$ be an irreducible representations of $\GL_2(E)$ with central character $\omega_{\pi_i}$ and $\omega_{\pi_1}=\omega_{\pi_2}$.
	 Then  $\pi_1\boxtimes\pi_2$ is an irreducible representation of the similitude group $$\GSO_{2,2}(E)\cong\frac{ \GL_2(E)\times \GL_2(E)}{\{(t,t^{-1}) \}}.$$ If $\pi_1\neq\pi_2,$ then $\Sigma=\Ind_{\GSO_{2,2}(E)}^{\GO_{2,2}(E)}(\pi_1\boxtimes\pi_2)$ is an irreducible smooth representation of $\GO_{2,2}(E)$ and $\Sigma\cong\Sigma\otimes\nu,$ where $\nu|_{\Oo_{2,2}(E)}=\det.$ If $\pi_1=\pi_2,$
then there are two extensions $(\pi_1\boxtimes\pi_1)^\pm$ and only one of them participates in the theta lift between $\GSp_4(E)$ and $\GO_{2,2}(E),$ denoted by $(\pi_1\boxtimes\pi_1)^+=\Sigma^+.$ Moreover, we have $(\pi_1\boxtimes\pi_1)^+\otimes\nu\cong (\pi_1\boxtimes\pi_1)^-.$ (See \cite[\S6]{gan2011theta}.)

Any irreducible representation of $$\GSO_{3,3}(E)=\frac{\GL_4(E)\times\GL_1(E)}{\{(t,t^{-2}):t\in E^\times \} }$$ is of the form
\[\Pi\boxtimes \chi \]
where $\Pi$ is a representation of $\GL_4(E)$ with central character $\omega_\Pi$, $\chi$ is a character of $E^\times$ and $\chi^2=\omega_\Pi$.
\subsection{Representations of $\GSp_4(E)$}
Assume that ${\tau}=\theta({\Sigma})$ is a representation of $\GSp_4(E)$ and ${\Sigma}$ is a representation of $\GSO_{2,2}(E)$. Then ${\tau}$ is  generic if and only if ${\Sigma}$ is generic. We follow the notation in \cite{gan2011theta} to describe the non-discrete series representations of $\GSp_4(E)$. Thanks to \cite[Proposition 5.3]{gan2011theta}, the non-discrete series representations of $\GSp_4(E)$ fall into the following three families:
\begin{itemize}
	\item ${\tau}\hookrightarrow I_{Q(Z)}(\chi|-|_E^{-s},\pi)$ with $\chi$ a unitary character, $s\geq0$ and $\pi$ a discrete series representation of $\GL_2(E)$ up to twist;
	\item ${\tau}\hookrightarrow I_{P(Y)}(\pi|-|_E^{-s},\chi)$ with $\chi$ an arbitrary character, $s\geq0$ and
	$\pi$ a unitary discrete series representation of $\GL(Y)$;
	\item ${\tau}\hookrightarrow I_{B}(\chi_1|-|_E^{-s_1},\chi_2|-|_E^{-s_2};\chi)$ where $\chi_1,\chi_2$ are unitary and $s_1\geq s_2\geq0$. 
\end{itemize}
Note that if ${\tau}$ itself is generic and non-tempered, then those embeddings are in fact isomorphisms due to the standard module conjecture for $\GSp_4$, except
\[\tau\hookrightarrow I_{Q(Z)}(\mathbf{1},\pi). \]
 For instance, $\tau=J_{P(Y)}(\pi|-|_E^s,\chi)$ with $s\geq0$. If $\tau$ is generic, then $I_{P(Y)}(\pi|-|_E^{s},\chi)$ is irreducible and so 
$$\tau=I_{P(Y)}(\pi|-|_E^{s},\chi)\cong I_{P(Y)}(\pi^\vee|-|_E^{-s},\chi\omega_{\pi}|-|_E^{2s}) $$
 with $s\geq0$. (See \cite[Lemma 5.2]{gan2011theta}.)

 If the big theta lift $\Theta_6^+({\tau})$ to $\GO_{3,3}(E)$ of $\tau$ is irreducible, then the restricted representation $\Theta_6^+({\tau})|_{\GSO_{3,3}(E)}$ is irreducible due to \cite[\S5, Page 282]{prasad1993local}. We will use $\Theta_6({\tau})$ to denote the big theta lift to $\GSO_{3,3}(E)$ of ${\tau}$.
\begin{proposition}\label{big:GSp4GO6}
	Let ${\tau}$ be a generic representation of $\GSp_4(E)$. Then the big theta lift $\Theta_6({\tau})$ to $\GSO_{3,3}(E)$ of $\tau$ is an irreducible representation unless ${\tau}=I_{Q(Z)}(|-|_E,\pi)$ with $\pi$ essentially square-integrable. If $\tau=I_{Q(Z)}(|-|_E,\pi)$, then 
$\Theta_6(\tau)=I_P(\pi|-|_E,\pi)\boxtimes\omega_{\pi}|-|_E$ is reducible.
\label{bigtheta}
\end{proposition}
\begin{proof}
	If ${\tau}$ is a tempered representation, then $\Theta_6^+({\tau})$ is irreducible due to \cite[Proposition C.4]{gan2014formal}(which holds even for $p=2$ since the Howe duality conjecture holds) and so $\Theta_6({\tau})$ is irreducible. Assume that the generic representation ${\tau}$ is not essentially tempered. There are $4$ cases:
	\begin{itemize}
		\item If $\tau=I_{B}(\chi_1,\chi_2;\chi)$ is irreducible, then none of the characters $\chi_1,\chi_2,\chi_1/\chi_2,\chi_1\chi_2$ is $|-|_E^{\pm1}$ and so $I_{B_0}(\mathbf{1},\chi_2,\chi_1,\chi_1\chi_2)$ has a generic quotient where $B_0$ is a Borel subgroup of $\GL_{4}(E)$. Thus $\Theta_6(\tau)=I_{B_0}(\mathbf{1},\chi_2,\chi_1,\chi_1\chi_2)\cdot\chi\boxtimes\chi^2\chi_1\chi_2$ is irreducible due to the standard module conjecture for $\GL_4$.
		\item If $\tau=I_{P(Y)}(\pi,\chi)$, then $\Theta_6(\tau)$ is a quotient of 
		\[I_{Q}(\mathbf{1},\pi,\omega_\pi )\cdot\chi\boxtimes \chi^2\omega_\pi \]
		where $Q$ is  a parabolic subgroup of $\GL_4(E)$ with Levi subgroup $\GL_1(E)\times\GL_2(E)\times\GL_1(E)$. Since the standard $L$-function $L(s,\tau,Std)$ is a factor of $L(s,Ad\circ \phi_{\tau} )$, $L(s,\tau,Std)$ is regular at $s=1$. Then $I_Q(\mathbf{1},\pi,\omega_\pi)$ is irreducible and so $\Theta_6(\tau)=I_{Q}(\mathbf{1},\pi,\omega_\pi )\cdot\chi\boxtimes \chi^2\omega_\pi $ is irreducible.
		\item If $\tau=I_{Q(Z)}(\chi,\pi)$ with $\chi\neq\mathbf{1}$, then there is an epimorphism
		\[\xymatrix{ I_{P}(\pi\cdot\chi,\pi)\boxtimes\omega_\pi\chi \ar@{->>}[r]& \Theta_6(\tau)} \]
		of $\GSO_{3,3}(E)$-representations, where  $P$ is a parabolic subgroup of $\GL_4(E)$ with Levi subgroup $\GL_2(E)\times\GL_2(E)$. Gan-Takeda \cite[Proposition 13.2]{gan2011theta}  have proved that $I_P(\pi\cdot\chi,\pi)$ is irreducible if $I_{Q(Z)}(\chi,\pi)$ is irreducible and $\chi\neq|-|_E$. If $\chi=|-|_E$ and $\pi$ is essentially square-integrable, then $\Theta_6(\tau)=I_P(\pi\cdot\chi,\pi)\boxtimes\omega_{\pi}\chi$ and $\theta_6(\tau)=J_P(\pi\cdot\chi,\pi)\boxtimes\omega_{\pi}\chi$ is the Langlands quotient.
		\item If $\tau\hookrightarrow I_{Q(Z)}(\mathbf{1},\pi)$, then $\Theta_6(\tau)$ is either zero or $I_P(\pi,\pi)\boxtimes\omega_{\pi}$, where $P$ is a parabolic subgroup of $\GL_4(E)$ with Levi subgroup $\GL_2(E)\times\GL_2(E)$. In fact, $\Theta_6(\tau)=0$ only when $\tau$ is a nongeneric constituent representation of $I_{Q(Z)}(\mathbf{1},\pi)$.
	\end{itemize}
	This finishes the proof of Proposition \ref{bigtheta}. 
\end{proof}
\begin{remark}
	Similarly one can prove that if $\Sigma$ is a generic representation of $\GSO_{2,2}(E)$ and $L(s,\Sigma,Std)$ is regular at $s=1$, then the big theta lift $\Theta_2(\Sigma)$ to $\GSp_4(E)$ is an irreducible representation. 
\end{remark}
Let us turn the table around. The rest in this subsection focuses on the computation of local theta lifts to $\GO_{2,2}(E)$ from $\GSp_4(E)$.
\begin{proposition} Let $\tau$ be a generic representation of $\GSp_4(E)$. Assume that $\theta_4^+(\tau)\neq0$.
	\begin{enumerate}[(i)]
		\item
		If $\tau=I_{Q(Z)}(\mathbf{1},\pi(\mu_1,\mu_2))$, then
		the big theta lift
		$\Theta_4^+(\tau)$ to $\GO_{2,2}(E)$ of $\tau$  is  $\Ext^1_{\GO_{2,2}(E)}(\Sigma^+,\Sigma^-)$, 
		where $\Sigma^\pm$ are two distinct extensions of $\pi(\mu_1,\mu_2)\boxtimes\pi(\mu_1,\mu_2)$ from $\GSO_{2,2}(E)$ to $\GO_{2,2}(E)$;
		\item 	If   $\tau\neq I_{Q(Z)}(\mathbf{1},\pi(\mu_1,\mu_2))$, then 
		$\Theta_4^+(\tau)$ is an irreducible  representation of $\GO_{2,2}(E)$. 
	\end{enumerate}\label{GSp:GO(4)}
\end{proposition}
\begin{proof}
	\begin{enumerate}[(i)]
		\item If $\tau=I_{Q(Z)}(\mathbf{1},\pi(\mu_1,\mu_2))$, then the small theta lift $\theta_4^{+}(\tau)=\Sigma^+$ by the Howe duality, where $\Sigma^+$ is the extension to $\GO_{2,2}(E)$ of $\pi(\mu_1,\mu_2)\boxtimes\pi(\mu_1,\mu_2)$. Let $\psi_U$ be a non-degenerate character of the standard unipotent subgroup $U$ of $\GO_{2,2}(E)$. Then 
		\begin{equation}\label{whittakerdimension}
		\dim\Hom_U(\Theta_4^+(\tau),\psi_U)=\dim\Hom_{H(\mathcal{W}_1)\rtimes \Sp(\mathcal{W}_1)}(\tau,\omega_\psi)=2  \end{equation}
		where $\mathcal{W}_2=Z\oplus \mathcal{W}_1\oplus Z^\ast$, $H(\mathcal{W}_1)$ is the Heisenberg group of $\mathcal{W}_1$ equipped with the Weil representation $\omega_\psi$ and $\tau$ is the representation of $\GSp(\mathcal{W}_2)$. Thus the big theta lift $\Theta_4^+(\tau)$ to $\GO_{2,2}(E)$ is reducible. There is a short exact sequence of $\GO_{2,2}(E)$-representations
		\begin{equation}\label{bigthetaexact}
		\xymatrix{\Sigma^-\oplus\Sigma^+\ar[r]&\Theta_4^{+}(\tau)\ar[r]&\Sigma^+\ar[r]&0.}
		\end{equation}
		However, we can not determine $\Theta_4^+(\tau)$ at this moment.
		Note that
		\[\dim\Ext^1_{\GSO_{2,2}(E)}(\pi(\mu_1,\mu_2)\boxtimes\pi(\mu_1,\mu_2),\pi(\mu_1,\mu_2)\boxtimes\pi(\mu_1,\mu_2))=1\]
		due to \cite[Theorem 1]{dipendra2012extensions}. Here $\Ext^1$ is the extension functor defined on the category of all smooth representations with a fixed central character.
		Then $\dim\Ext^1_{\GO_{2,2}(E)}(\Sigma^+,\Sigma^-\oplus\Sigma^+)=1 $
		by Frobenius Reciprocity, which implies that either $\Ext^1_{\GO_{2,2}(E)}(\Sigma^+,\Sigma^-)$ or $\Ext^1_{\GO_{2,2}(E)}(\Sigma^+,\Sigma^+)$ is zero. Assume that $B$ is the Borel subgroup of $\GSO_{2,2}(E)$. Set $\tilde{B}=B\rtimes\mu_2$ to be a subgroup of $\GO_{2,2}(E)$ and $\tilde{B}\cap\GSO_{2,2}(E)=B$. Since $$\pi(\mu_1,\mu_2)\boxtimes\pi(\mu_1,\mu_2)=\Ind_B^{\GSO_{2,2}(E)}\chi~~(\mbox{normalized induction}), $$
		there are two extensions $\chi^\pm$  to $\tilde{B}$ of $\chi$ of $B$. We may assume without loss of generality that $\Sigma^+=\Ind_{\tilde{B}}^{\GO_{2,2}(E)}\chi^+$ and $\Sigma^-=\Ind_{\tilde{B}}^{\GO_{2,2}(E)}\chi^-$. Note that $\Ext^1_{\tilde{B}}(\chi^+,\chi^-)\neq0$. Then there is a short exact sequence of $\GO_{2,2}(E)$-representations
		\[\xymatrix{0\ar[r]&\Sigma^-\ar[r]&\Ind_{\tilde{B}}^{\GO_{2,2}(E)}(\Ext^1_{\tilde{B}}(\chi^+,\chi^-))\ar[r]&\Sigma^+\ar[r]&0, } \]
		which is not split. Hence $\Ext^1_{\GO_{2,2}(E)}(\Sigma^+,\Sigma^-)\neq0$. Together with
		\eqref{whittakerdimension} and \eqref{bigthetaexact}, one can obtain the desired equality	$\Theta_4^+(\tau)=\Ext^1_{\GO(2,2)(E)}(\Sigma^+,\Sigma^-)$.
		\item If $\tau$ is a (essentially) discrete series representation, then it follows from \cite[Proposition 5.4]{atobe2016local}. 
		\begin{itemize}
			\item If $\tau=I_{Q(Z)}(\mu_0,\pi(\mu_1,\mu_2))$ with $\mu_0\neq\mathbf{1}$, then there is only one orbit in the double coset  $Q(Z)\backslash \GSp_4(E)/ H(\mathcal{W}_1)\rtimes\Sp(\mathcal{W}_1)$  that contributes to the multiplicity
			\[\dim\Hom_{H(W_1)\rtimes\Sp(\mathcal{W}_1)}(\tau,\omega_\psi)  \]
			and so $\Theta_4^+(\tau)$ is irreducible.
			\item If $\tau\subset I_{Q(Z)}(\mathbf{1},\pi)$ with $\pi$ square-integrable, then $\tau$ is tempered. Due to \cite[Proposition 5.5]{atobe2016local}, $\Theta_4^+(\tau)$ is tempered. Note that $\theta_4^+(\tau)$ is a discrete series representation which is projective in the category of the tempered representations. Thus $\Theta_4^+(\tau)=\theta_4^+(\tau)$ is irreducible. Otherwise, it will contradict the Howe duality conjecture (see Theorem 2.1).
			\item If $\tau=I_{P(Y)}(\pi,\chi)$, then $\dim\Hom_U(\Theta_4^+(\tau),\psi_U) =1$ and so $\Theta^+_4(\tau)$ is irreducible.
		\end{itemize}
	\end{enumerate}
This finishes the proof of Proposition \ref{GSp:GO(4)}.
\end{proof}

\input{corestriction}

\input{PGSp}

\input{gu2D}

\section{The Prasad conjecture for $\GSp_4$}
\subsection{The Prasad conjecture}
\label{subsect:prasad}
In this subsection, we give a brief introduction to the Prasad conjecture, i.e. \cite[Conjecture 2]{prasad2015arelative}. One may refer to \cite[\S 16]{prasad2015arelative} for more details.
\par
Let $\mathbf{G}$ be a quasi-split reductive group defined over a local field $F$ with characteristic zero. Let $W_F$ be the Weil group of $F$ and $WD_F$ be the Weil-Deligne group of $F$. Let $E$
be a quadratic extension over $F$. Dipendra Prasad introduces
a quadratic character $\chi_\mathbf{G}$ in \cite[\S 10]{prasad2015arelative} and  another quasi-split reductive group $G^{op}$ defined over $F$ in \cite[\S 9]{prasad2015arelative}. Then there is a relation between the fibers of the base change map 
\[\Phi:\Hom(WD_F,{}^LG^{op})\longrightarrow\Hom(WD_E,{}^LG^{op})\]
 from the Galois side and the $\chi_\mathbf{G}$-distinction problems for $\mathbf{G}(E)/\mathbf{G}(F)$ from the automorphic side.
\par
More precisely, assume the Langlands-Vogan conjecture in \cite{vogan1993local}. Given an irreducible representation $\pi$ of $\mathbf{G}(E)$ with an enhanced L-parameter $(\phi_\pi,\lambda)$, where $\lambda$ is a character of the component group $\pi_0(Z(\phi_{\pi}))$ and  the $L$-packet $\Pi_{\phi_{\pi}}$ is generic, we have
\[\sum_{\alpha}\dim\Hom_{G_\alpha(F)}(\pi,\chi_\mathbf{G})=\sum_{i}m(\lambda,\tilde{\phi}_i)\deg\Phi(\tilde{\phi}_i)/d_0(\tilde{\phi}_i) \]
where
\begin{itemize}
	\item $\alpha\in H^1(W_F,\mathbf{G})$ runs over all pure inner forms of $\mathbf{G}$ satisfying $G_\alpha(E)=\mathbf{G}(E);$
	\item $\tilde{\phi}_i\in \Hom(WD_F,{}^LG^{op})$ runs over all parameters of ${}^LG^{op}$ satisfying $\tilde{\phi}_i|_{WD_E}=\phi_\pi;$
	\item $m(\lambda,\tilde{\phi})=\dim\Hom_{\pi_0(Z(\tilde{\phi}))}(\mathbf{1},\lambda) $ is the multiplicity of the trivial representation contained in the restricted representation $\lambda|_{\pi_0(Z(\tilde{\phi}) )}$ ;
	\item $d_0(\tilde{\phi})=|Coker\{\pi_0(Z(\tilde{\phi}))\longrightarrow\pi_0(Z(\phi_\pi))^{\Gal(E/F)} \}|.$
\end{itemize}
\begin{rem}
	If $H^1(F,\mathbf{G})$ is trivial such as $\mathbf{G}=\GSp_{2n},$ then the automorphic side contains only one term. The Prasad conjecture gives the precise formula for the multiplicity
	\[\dim\Hom_{\mathbf{G}(F)}(\pi,\chi_\mathbf{G} ).  \]
\end{rem}
\begin{rem}
	There exists a counterexample even for $\GL_2$ when $\Pi_{\phi_{\pi}}$ is not generic. Let $\mathbf{G}=\GL_2$ and $\pi=\mathbf{1}$ be the trivial representation. Then the automorphic side is zero however the Galois side is nonzero.
\end{rem}
\begin{rem}
	If $\tilde{\phi}$ comes from a square-integrable representation, then $\deg\Phi(\tilde{\phi})=1.$ 
\end{rem}
If $\pi$ is square-integrable, then we have a refined version, i.e. the formula for each dimension
\[\dim\Hom_{G_\alpha(F)}(\pi,\chi_\mathbf{G}). \] 
\par
Let $Z(\hat{G}^{op})$ be the center of the dual group $\hat{G}^{op}$.
There is a perfect pairing 
\[H^1(\Gal(E/F), Z(\hat{G}^{op}))\times H^1(\Gal(E/F),\mathbf{G}(E))\longrightarrow\mathbb{Q}/\mathbb{Z} \]
when Dipendra Prasad studies the character twists in \cite[\S 13]{prasad2015arelative}.
Set $\Omega_\mathbf{G}(E)=H^1(\Gal(E/F),Z(\hat{G}^{op})).$ Given a parameter $\tilde{\phi}\in H^1(W_F,\hat{G}^{op}),$ we consider the stabilizer $\Omega_\mathbf{G}(\tilde{\phi},E)\subset \Omega_\mathbf{G}(E)$ under the pairing
\[H^1(W_F,Z(\hat{G}^{op}))\times H^1(W_F,\hat{G}^{op})\longrightarrow H^1(W_F,\hat{G}^{op}). \] Set
$$A_\mathbf{G}(\tilde{\phi})\subset H^1(\Gal(E/F),\mathbf{G}(E))\cong\Omega_\mathbf{G}(E)^\vee $$ to be the annihilator of the stabilizer $\Omega_\mathbf{G}(\tilde{\phi},E).$ Then there is another perfect pairing
\[\Omega_\mathbf{G}(E)/\Omega_\mathbf{G}(\tilde{\phi},E)\times A_\mathbf{G}(\tilde{\phi})\longrightarrow\mathbb{Q}/\mathbb{Z}, \]
meaning that in the orbit $\Omega_\mathbf{G}(E)/\Omega_\mathbf{G}(\tilde{\phi},E)$ of character twists of $\tilde{\phi}$ (which go to a particular parameter under the basechange to $E$) there are exactly as many parameters as there are certain pure inner forms of $\mathbf{G}$ over $F$ which trivialize after basechange to $E.$
\par
Consider
 \[F(\phi_\pi)=\{\tilde{\phi}:WD_F\longrightarrow {}^LG^{op}|~\tilde{\phi}|_{WD_E}=\phi_\pi \}=\sqcup_{i=1}^r\mathcal{O}(\tilde{\phi}_i). \]
Each orbit $\mathcal{O}(\tilde{\phi}_i)$ of $\Omega_\mathbf{G}(E)$-action on $F(\phi_\pi)$ is associated to a coset  $\mathcal{C}_i$ of $A_\mathbf{G}(\tilde{\phi}_i,E)$ in $H^1(\Gal(E/F),\mathbf{G}(E))$ defining a set of certain pure inner forms $G_\alpha$ of $\mathbf{G}$ over $F$ such that $G_\alpha(E)=\mathbf{G}(E).$ Then
\[\dim \Hom_{G_{\alpha(F)}}(\pi,\omega_\mathbf{G})=\sum_{i=1}^r m(\lambda,\tilde{\phi}_i)\cdot 1_{\mathcal{C}_i}(G_\alpha)/ d_0(\tilde{\phi}_i) ,\]
where
\begin{itemize}
	\item $1_{\mathcal{C}_i}$ is the characteristic function of the coset $\mathcal{C}_i;$
	\item  $m(\lambda,\tilde{\phi} )$ is the multiplicity  for the trivial representation contained in the restricted representation  $\lambda|_ {\pi_0(Z(\tilde{\varphi}))},$ which may be zero; 
	\item  $d_0(\tilde{\phi}) =
	|Coker\{\pi_0(Z(\tilde{\phi})) \longrightarrow\pi_0(Z(\phi_\pi))^{\Gal(E/F)}\} |
	$.
\end{itemize}

\subsection{The Prasad conjecture for $\GL_2$}
Before we give the proof of Theorem \ref{thm1.2}, let us recall the Prasad conjecture for $\mathbf{G}=\GL_2=\GSp_2.$
Set $\mathbf{G}=\GL_2.$ Then $\chi_\mathbf{G}=\omega_{E/F}$ and $G^{op}=\mathrm{U(2,E/F)}$ is the quasi-split unitary group, where $E$ is a quadratic field extension over a p-adic field $F$. Denote $$^LG^{op}=\GL_2(\mathbb{C})\rtimes<\sigma>,$$ where $\sigma$-action on $\GL_2(\mathbb{C})$ is given by
\[\sigma(g)=\omega_0 {(g^t)}^{-1}\omega_0^{-1}=g\cdot\det(g)^{-1},\]$\omega_0=\begin{pmatrix}
&1\\-1
\end{pmatrix} \mbox{ and }g\in \GL_2(\mathbb{C}) $, $g^t$ denotes its transpose matrix.
Given an irreducible representation $\pi$ of
$\GL_2(E)$ with $\phi=\phi_\pi$ irreducible (for simplicity),  there is no other pure inner form for $\GL_2$. Then 
\[\dim \Hom_{\GL_2(F)}(\pi,\omega_{E/F})=|F(\phi)|, \]
where $F(\phi)=\{\tilde{\phi}:WD_F\longrightarrow {}^LG^{op}|~\tilde{\phi}|_{WD_E}=\phi \}$ and $|F(\phi)|$ denotes its cardinality.
\begin{prop} \label{conj:GL(2)}
	The following statements are equivalent:
	\begin{enumerate}[(i)]
		\item $\dim \Hom_{\GL_2(F)}(\pi,\omega_{E/F})=1;$
		\item the Langlands parameter $\phi$ is conjugate-symplectic;
		\item there is only one extension $\tilde{\phi}\in F(\phi).$ 
	\end{enumerate}
\end{prop}
\begin{proof}
	We only prove the direction (ii)$\Rightarrow$(iii) and the rest follows from Flicker's results \cite{flicker1991ondist}.
	If $\phi$ is conjugate-symplectic, then $$\phi^s=\phi^\vee=\phi(\det\phi)^{-1},$$
	where $s\in W_F\setminus W_E $ is fixed. There exists $A\in \GL_2(\mathbb{C})$ such that
	\[\phi(sts^{-1})=\phi^s(t)= A\cdot\phi(t)\det(\phi(t))^{-1}\cdot A^{-1}\]
	  for all $t\in WD_E$. 
	Pick $a\in\mathbb{C}^\times$ such that $a^2\cdot\det A=1,$ so that $aA\in \SL_2(\mathbb{C}).$ Set 
	\[\tilde{\phi}(s)=aA\cdot \sigma\]
	  and  $\tilde{\phi}(t)=\phi(t)$   for $t\in WD_E$.
	Then $$\tilde{\phi}(sts^{-1})=\tilde{\phi}(s)\cdot \phi(t)\cdot\tilde{\phi}(s)^{-1}$$ and $\tilde{\phi}(s^2)=\phi(s^2)=(\tilde{\phi}(s))^2$ due to the sign of $\phi.$ Therefore $\tilde{\phi}\in F(\sigma).$ If there are two extensions $\tilde{\phi}_i$ with $A_i\in \SL_2(\mathbb{C})$ such that $\tilde{\phi}_i|_{WD_E}=\phi,$ then $A_1A_2^{-1}\in Z(\phi)\cong\mathbb{C}^\times$ by Schur's lemma, so that $\phi_1=\phi_2.$ 
\end{proof}
\begin{rem}
	This method will appear again when we study the Prasad conjecture for $\mathbf{G}=\GSp_4$ in \S\ref{7.4.1}  The key idea is to choose a proper element $A$ such that the lift $\tilde{\phi}$ satisfies $\tilde{\phi}(s)=A\cdot\sigma$ and $\tilde{\phi}|_{WD_E}=\phi.$
\end{rem}

\subsection{The Prasad conjecture for  $\GSp_4$}\label{subsect:GSp(4)conj}
The aim of this subsection is to verify the Prasad conjecture for $\GSp_4$.
Now we consider the generic representation $\tau=\theta(\Pi\boxtimes\chi)$ of $\GSp_4(E)$, with $\phi_\Pi$ conjugate-symplectic and
$\chi|_{F^\times}=1.$  Note that the Langlands parameter $\phi_\Pi=i\circ\phi_\tau,$
where \[i:\GSp_4(\mathbb{C})\rightarrow\GL_4(\mathbb{C}) \]
is the embedding between $L$-groups.
 Moreover, $\chi$ is  the similitude character of $\phi_\tau.$ If $\phi_{\Pi}$ is conjugate-symplectic (resp. conjugate-orthogonal), we say that $\phi_\tau$ is conjugate-symplectic (resp. conjugate-orthogonal).
 
 \begin{lem}
 	Assume that $\tau=\theta(\Pi\boxtimes\chi)$ is a generic representation of $\GSp_4(E)$ and $\omega_\tau|_{F^\times}=\mathbf{1}$. Then $\tau$ is $(\GSp_4(F),\omega_{E/F})$-distinguished if and only if $\phi_{\Pi}$ is conjugate-symplectic.
 \end{lem}
\begin{proof}
	Due to Theorem \ref{localgspperiod}, the following are equivalent:
	\begin{itemize}
		\item $\tau$ is $\GSp_4(F)$-distinguished;
		\item $\Pi$ is $\GL_4(F)$-distinguished;
		\item $\phi_{\Pi}$ is conjugate-orthogonal.
	\end{itemize}
Fix a character $\chi_E$ of $E^\times$ such that $\chi_E|_{F^\times}=\omega_{E/F}$. Then
$\tau$ is $(\GSp_4(F),\omega_{E/F})$-distinguished if and only if
$\tau\otimes\chi_E\circ\lambda_W$ is $\GSp_4(F)$-distinguished, which is equivalent to that $\phi_\Pi\otimes\chi_E$ is conjugate-orthogonal. Note that $\chi_E^{-1}$ is conjugate-symplectic. Hence $\tau$ is $(\GSp_4(F),\omega_{E/F})$-distinguished if and only if $\phi_{\Pi}$ is conjugate-symplectic.
\end{proof}

Recall that if $\mathbf{G}=\GSp_{2n}$, then $\chi_\mathbf{G}=\omega_{E/F}$ and $$G^{op}(F)=\{g\in \GSp_{2n}(E)|\sigma(g)=\theta(g)\} $$
where $\theta(g)=\lambda_W (g)^{-1}g $ is the involution. Note that the $\sigma$-actions on $\GSp_4(E)$ and $\GSp_4(\mathbb{C})$ are totally different. (Hope that it does not confuse the reader.)
Observe that $H^1(\Gal(E/F),Z(\hat{G}^{op})^{W_E})=1,$ which corresponds to the fact that the pure inner form of $\GSp_{2n}$ is trivial. 

According to Theorem \ref{localgspperiod}, we will divide the proof of Theorem \ref{thm1.2} into four parts:
\begin{itemize}
	\item $\phi_\tau $ is irreducible;
	\item $\phi_{\tau}=\rho\oplus\rho\nu$ with $\nu\neq\mathbf{1}$;
	\item the endoscopic case $\phi_\tau=\phi_{\pi_1}\oplus\phi_{\pi_2} $ and $\tau$ is generic;
	\item$ \phi_\tau=\phi_{\pi_1}\oplus\phi_{\pi_2} $ and $\tau$ is nongeneric.
\end{itemize}
 See \S\ref{7.4.1}--\S\ref{7.4.4}.

\subsubsection{The irreducible $L$-parameter $\phi_\tau $}\label{7.4.1}
Given a conjugate-symplectic $L$-parameter $\phi=\phi_\tau,$ which is irreducible, we want to extend $\phi$ to  $$\tilde{\phi}:WD_F\longrightarrow {}^LG_0= \GSp_4(\mathbb{C})\rtimes<\sigma>,$$ where $\sigma$ acts on $\GSp_4(\mathbb{C})$ by
\[\sigma(g)=g\cdot\text{sim}(g)^{-1}. \]

Let $s\in W_F\setminus W_E.$ The parameter $\phi$ is conjugate-symplectic, so that $\phi^\vee=\phi^s$
and $\phi^\vee=\phi\chi^{-1}$. Hence there exists an element $A\in \GSp_4(\mathbb{C})$ such that
\begin{equation}\label{conjdual}
\phi(sts^{-1})=\phi^s(t)=A\cdot \phi(t)\chi^{-1}(t) \cdot A^{-1}    \end{equation}
for  all  $t\in WD_E$.
Pick $a\in\mathbb{C}^\times$ such that $a^2=\text{sim} (A)^{-1}$. Then
$aA\in \Sp_4(\mathbb{C}).$ Set
\[\tilde{\phi}(s)=aA\cdot\sigma\mbox{  and }\tilde{\phi}(t)=\phi(t)\] for $t\in WD_E$. Then 
$\phi(sts^{-1})=A\phi(t)\chi^{-1}(t)A^{-1}=\tilde{\phi}(s)\cdot\phi(t)\cdot\tilde{\phi}(s)^{-1}$. Moreover, we will show that
\[\tilde{\phi}(s^2)=\phi(s^2)=(\tilde{\phi}(s))^2. \]
Then $\tilde{\phi}\in\Hom(WD_F, {}^LG_0)$ and $\tilde{\phi}|_{WD_E}=\phi.$
\par
Assume that  $<-,->$ is the $WD_E$-equivariant bilinear form associated to $$\phi_\tau:WD_E \rightarrow \GSp_4(\mathbb{C})=\GSp(V,<-,->).$$
 Set
$$B(v,w)=<v,A^{-1}w>$$
  for  $v,w\in V. $
Then \eqref{conjdual} implies that
\begin{align*}
B(\phi(t)v,\phi(sts^{-1})w)&=<\phi(t)v,\phi(t)\chi^{-1}(t)A^{-1}w>\\
&=\chi(t)\cdot<v,\chi^{-1}(t)A^{-1}w>\\
&=B(v,w).
\end{align*}
 So $B$ is a conjugate-self-dual bilinear form on $\phi$ and hence it has sign $-1$ by  Schur's lemma, i.e., $$-B(w,v)=B(v,\phi(s^2)w).$$
Therefore we have
\begin{align*}
<v,w>&=-<w,v>\\
&=-B(w,Av)\\
&=B(Av,\phi(s^2)w)\\
&=<Av,A^{-1}\phi(s^2)w>\\
&=<v,a^{-2}A^{-2}\phi(s^2)w>
\end{align*}
and so $\phi(s^2)=a^2A^2=(\tilde{\phi}(s))^2.$
\begin{prop}
	Assume that  $\tau=\theta(\Pi\boxtimes\chi)$ with $\phi_\Pi$ irreducible. Then there exists at most one extension $\tilde{\phi}:WD_F\longrightarrow {}^LG_0$
	such that $\tilde{\phi}|_{WD_E}=\phi_\tau.$
\end{prop}
\begin{proof}
	If there are two extensions $\tilde{\phi}_i(i=1,2)$ associated to $A_i\in \Sp_4(\mathbb{C})$ satisfying 
	\[\tilde{\phi}_i(sts^{-1})=\tilde{\phi}_i(s)\cdot\phi(t)\cdot\tilde{\phi}_i(s)^{-1}\]
	for all $t\in WD_E$, 
	then $A_1A_2^{-1}$ commutes with $\phi.$ So  $A_1A_2^{-1}$ is a scalar by Schur's Lemma. Hence $\tilde{\phi}_1=\tilde{\phi}_2.$
\end{proof}
Hence, if $\tau=\theta(\Pi\boxtimes\chi)$ with $\phi_\Pi$ irreducible and conjugate-symplectic, then there is one extension $\tilde{\phi}\in F(\phi_\tau)$ and 
\[\dim \Hom_{\GSp_4(F)}(\tau,\omega_{E/F})=1. \] 

If $\phi=\phi_\tau$ is conjugate-symplectic and reducible, then there are several cases.

\subsubsection{$\phi_\tau=\rho+{\rho}\nu $ with $\nu\neq\mathbf{1}$ and $\rho $ irreducible } If $\phi_\Pi=\rho+\rho\nu$ with $\rho$ irreducible and $\chi=\nu\cdot\det\rho$ conjugate-orthogonal,  thanks to \cite[Theorem 5.2]{matringe2009distinction}, there are two subcases: 
\begin{itemize}
	\item $\rho$ and $\rho\nu$ are both conjugate-symplectic or
	\item  $\rho^s=\rho^\vee\nu^{-1}.$
\end{itemize}
\begin{enumerate}[(i)]
	\item If $\rho$ and $\rho\nu$ are both conjugate-symplectic, then $\nu$ is conjugate-orthogonal and  there exist  $$\tilde{\rho}_i:WD_F\longrightarrow \GL_2(\mathbb{C})\rtimes<\sigma>$$
	 such that $\tilde{\rho}_1|_{WD_E}=\rho,~  \tilde{\rho}_2|_{WD_E}=\rho\nu$ and $\tilde{\rho}_i(s)=A_i\cdot\sigma$ for $A_i\in\SL_2(\mathbb{C})$ due to Proposition \ref{conj:GL(2)}. 
	Note that $\rho$ is irreducible. Then for $t\in WD_E$,
	\[\tilde{\rho}_1^s(t)\nu^s(t)=\tilde{\rho}_2^s(t)=A_2\rho^\vee(t)A_2^{-1}\cdot \nu^{-1}(t) \]
	and so $A_1\rho^\vee(t)A_1^{-1}=A_2\rho^\vee(t)A_2^{-1}$ which implies
	 $A_1A_2^{-1}\in\mathbb{C}^\times$. So $A_1=A_2$ or $-A_2$. Set 
	\[\tilde{\phi}(s)= \begin{pmatrix}
	&A_1\\-A_1
	\end{pmatrix}\cdot\sigma\in \Sp_4(\mathbb{C})\rtimes<\sigma>\mbox{ and }\tilde{\phi}(t)=\begin{pmatrix}
	\rho(t)\\&\rho(t)\nu(t)
	\end{pmatrix} \]
	for $t\in WD_E$.	Then $\tilde{\phi}\in F(\phi)$ is the unique  extension of $\phi_\tau$.
	\item If $\rho^s\cong \rho^\vee\nu^{-1},$ then $\phi_\tau^s\cong\phi_\tau^\vee=(\det\rho)^{-1}\cdot\nu^{-1}\phi_\tau=\rho^\vee\nu^{-1}+\rho^\vee$ so that
	\[\rho^\vee\cong \rho^s\nu^s\cong \rho^s\nu. \]
	Set $\rho^s(t)\nu(t)=(\det\rho(t))^{-1}\cdot A\rho(t)A^{-1}$ for $t\in WD_E$.
	Note that $\rho$ is irreducible. Then
	$$A^{-2}\rho(s^2)\in\mathbb{C}^\times \mbox{  and  }
	\det\rho^s\cdot\det\rho\cdot\nu^2=\mathbf{1}, $$ 
	which implies that $\nu=\nu^s.$ Since 
	\[\rho(s^2)\nu(s^2)=A\rho(s^2)A^{-1}(\det \rho(s^2))^{-1}=\rho(s^2)(\det\rho(s^2))^{-1}, \]
we obtain	$\nu(s^2)\cdot \det(\rho(s^2))=1$.
	Set 
	\[\tilde{\phi}(s)=\begin{pmatrix}
	A\\&A\cdot\det(A^{-1})
	\end{pmatrix}\cdot \sigma\mbox{  and  }\tilde{\phi}(t)=\begin{pmatrix}
	\rho(t)\\&\rho(t)\nu(t)
	\end{pmatrix} \]
	for $t\in WD_E$.	Then $\tilde{\phi}(sts^{-1})=\tilde{\phi}(s)\cdot \phi(t)\cdot\tilde{\phi}(s)^{-1}$ and $\tilde{\phi}|_{WD_E}=\phi.$
\end{enumerate}
\subsubsection{Endoscopic case} If $\phi_\tau=\rho_1+\rho_2$ is the endoscopic case, then $\det\rho_1=\det\rho_2$ are both conjugate-orthogonal. There are several subcases. Assume that $\tau=\theta(\pi_1\boxtimes\pi_2)$ is generic, $\rho_i=\phi_{\pi_i}(i=1,2)$ and
$\rho_0=\chi_1+\chi_2,$ with $\chi_1\neq\chi_2$ and $\chi_1|_{F^\times}=\chi_2|_{F^\times}=\omega_{E/F}.$ Assume that $\rho_1\neq\rho_2.$ Then
\begin{enumerate}[(i)]
	\item If $\rho_1$ and $\rho_2$ are both conjugate-symplectic and $\rho_i\neq\rho_0~(i=1,2),$ so that both $\pi_1$ and $\pi_2$ are $(D^\times(F),\omega_{E/F})$-distinguished,
	then \[\dim \Hom_{\GSp_4(F)}(\tau,\omega_{E/F})=2. \]
	Thanks to Proposition \ref{conj:GL(2)}, there exist $\tilde{\rho}_1$ and $\tilde{\rho}_2$ of $\mathrm{U(2,E/F)}$ such that $\tilde{\rho}_i|_{WD_E}=\rho_i.$ So there are two lifts $\tilde{\phi}_1=\tilde{\rho}_1+\tilde{\rho}_2$ and $\tilde{\phi}_2=\tilde{\rho}_1\omega_{E/F}+\tilde{\rho}_2$ such that $\tilde{\phi}_i|_{WD_E}=\phi.$
	\par
	If $\rho_1$ and $\rho_2$ are both irreducible, then every lift of $\phi$ should be of the form
	\[s\mapsto\begin{pmatrix}
	\lambda_1\tilde{\rho}_1(s)&\\&\lambda_2\tilde{\rho}_2(s)
	\end{pmatrix}\in\GSp_4(\mathbb{C})\rtimes<\sigma>\]
	 with  $\lambda_i^2=1. $
	It is known that $\tilde{\phi}=-\tilde{\phi}$ as  parameters of ${}^LG_0.$
	\par
	If $\rho_1=\chi^{-1}+\chi^s,$ then the centralizer $Z_{\GL_2(\mathbb{C})}(\rho_1)$ is  $\mathbb{C}^\times\times\mathbb{C}^\times$ or $\GL_2(\mathbb{C}).$ Moreover, \[\tilde{\rho}_1(s)=a\begin{pmatrix}
	1\\&\chi(s^2)
	\end{pmatrix}\cdot
	\sigma\mbox{ with }a^2\chi(s^2)=1.\]
	In this case, $\tilde{\rho}_1+\tilde{\rho}_2\neq\tilde{\rho}_1\omega_{E/F}+\tilde{\rho}_2$, which will be a different story if $\rho_1=\rho_0.$
	\item If $\rho_1= \rho_0$ and $\rho_2$ is conjugate-symplectic, then $\tilde{\rho}_1(s)=\begin{pmatrix}
	&1\\-1
	\end{pmatrix}\cdot\sigma$ and
	$\tilde{\phi}_1= \tilde{\phi}_2.$
	\item If $\rho_1^\vee= \rho_2^s,$ then there exist $A\in \SL_2(\mathbb{C})$ such that
	\[A^{-1}\rho_1^\vee(t)A=\rho_2^s(t) \]
	for $t\in WD_E$.
	Set $\tilde{\phi}(s)=\begin{pmatrix}
	& A\rho_2(s^2)  \\A^{-1}
	\end{pmatrix}\cdot\sigma$. Then $\tilde{\phi}(sts^{-1})=\tilde{\phi}(s)\cdot\tilde{\phi}(t)\cdot\tilde{\phi}(s^{-1}).$
\end{enumerate}
Now we assume $\rho_1=\rho_2.$
\begin{enumerate}[(i)]
	\item If $\rho_1$ is conjugate-symplectic but $\rho_1\neq\rho_0$, then $\tilde{\phi}_1=\tilde{\rho}_1+\tilde{\rho}_1$ and $\tilde{\phi}_2=\tilde{\rho}_1+\tilde{\rho}_1\omega_{E/F}.$
	\item If $\rho_1=\rho_0,$ there is only one lift $\tilde{\phi}=\tilde{\rho}_1+\tilde{\rho}_1.$
	\item If $\rho_1$ is not conjugate-symplectic but  conjugate-orthogonal, set $$\tilde{\phi}(s)=\begin{pmatrix}
	&A\\A
	\end{pmatrix}\cdot\sigma\in \Sp_4(\mathbb{C})\rtimes<\sigma> $$
	where $A\in \SL_2(\mathbb{C})$ satisfies $A\rho_1^\vee(t)A^{-1}=\rho_1^s(t) .$ Let us verify $$\phi(s^2)=\tilde{\phi}(s^2)=\tilde{\phi}(s)^2 $$
	i.e., $A^2=\rho_1(s^2)$.
	\begin{itemize}
		\item If $\rho_1$ is irreducible, then $A^{-2}\rho_1(s^2)\in\mathbb{C}^\times$. Note that $\rho_1$ is conjugate-orthogonal. Then $A^{-2}\rho_1(s^2)=1$, i.e.  $A^2=\rho_1(s^2)$.
		\item If $\rho_1=\mu_1+\mu_2$ with $\mu_1\mu_2^s=\mathbf{1}$, then $\rho_1$ is conjugate-symplectic.
		\item If $\rho_1=\mu_1+\mu_2$ with $\mu_1\neq\mu_2$ and $\mu_1|_{F^\times}=\mu_2|_{F^\times}=\mathbf{1}$, then $A=1$ and $A^2=1=\rho_1(s^2)$.
	\end{itemize}
\end{enumerate}

\subsubsection{Non-generic tempered} \label{7.4.4}
Let
$\tau $ be an irreducible  nongeneric tempered representation of $\GSp_4(E)$ and  $\tau=\theta(\pi_1\boxtimes \pi_2 ),$ where each $\pi_i$ is an  irreducible representations of $D_E^\times(E)$.
If the enhanced $L$-parameter of $\tau$ is $(\phi_\tau,\lambda)$, where $\phi_\tau=\rho_1+\rho_2$, $\rho_i=\phi_{\pi_i}$ and $\lambda$ is a nontrivial character  of the component group $\pi_0(Z_{\phi_\tau}/Z_{\GSp_{4}(\mathbb{C})})$,  then
\[\dim \Hom_{\GSp_4(F)}(\tau,\omega_{E/F} )=0. \]
On the Galois side, if $\phi_\pi=\rho_1+\rho_2$, then for arbitrary parameter $\tilde{\phi}$ satisfying $\tilde{\phi}|_{WD_E}=\phi_\tau,$ the restricted representation $\lambda|_{\pi_0(Z(\tilde{\phi})) }$ does not contain the trivial character $\mathbf{1},$ i.e.
\[m(\lambda,\tilde{\phi})=0. \]
Finally we can prove Theorem \ref{thm1.2}.
\begin{proof}[Proof of Theorem \ref{thm1.2}] It is obvious if $\tau$ is a non-generic tempered representation  of $\GSp_4(E).$
	Since the Levi subgroup of a parabolic subgroup in $\GSp_4$ are $\GL$-type, \cite[Lemma 14]{prasad2015arelative} implies that $\deg\Phi(\tilde{\phi})=1$ in our case. By the above discussions, we know that if $\tau$ is generic, then 
	$\dim\Hom_{\GSp_4(F)}(\tau,\omega_{E/F} )$ equals to the number of inequivalent lifts $|F(\phi_{\tau})|. $
\end{proof}

\section{Proof of Theorem \ref{prasadfordisc}}\label{secpgsp}
This section focuses on the Prasad conjecture for  $\PGSp_4$.
Let $\tau$ be a representation of $\PGSp_4(E),$ i.e., a representation $\tau$ of $\GSp_4(E)$
with trivial central character.
If the multiplicity
\[\dim \Hom_{\PGSp_4(F)}(\tau,\omega_{E/F} )=\dim \Hom_{\GSp_4(F)}(\tau,\omega_{E/F} ) \]
is nonzero, then we say $\tau$ is $(\PGSp_4(F),\omega_{E/F})$-distinguished. Let $\PGSp_{1,1}=\PGU_2(D)$ be the pure inner form of $\PGSp_4$ defined over $F$. Similarly,
\[\dim\Hom_{\PGSp_{1,1}(F)}(\tau,\omega_{E/F})=\dim\Hom_{ \GSp_{1,1}(F)}(\tau,\omega_{E/F}).  \]

\subsection{Notation}
\begin{itemize}
		\item $\tau,\pi^{++},\pi^{--},\pi^+$ and $\pi^-$ are  representations of $\PGSp_4(E)$.
	\item $s\in W_F\setminus W_E$ and $\phi_{\tau}^s(t)=\phi_{\tau}(sts^{-1}) $ for $t\in WD_E$.
\item $S_\phi=\pi_0(Z(\phi))$ is the component group associated to $\phi$.
	\item $\tau':WD_F\longrightarrow\Sp_4(\mathbb{C})$ and $\tau_i'$ are  Langlands parameters of $\PGSp_4(F)$.
	\item $\mathcal{C}_i$ is a coset of $A_G(\tau_i')$ in $H^1(W_F,\PGSp_4)$ and $1_{\mathcal{C}_i}$ denotes its characteristic function.
	\item $\PGSp_{1,1}$ (resp. $PD^\times$) is the pure inner form of $\PGSp_4$ (resp. $\PGL_2$) defined over $F$.
\end{itemize}

\subsection{The Prasad conjecture for $\PGL_2$}
If $\mathbf{G}=\PGL_2$, then $\chi_{\mathbf{G}}=\omega_{E/F} $ and $G^{op}=\PGL_2$.
\begin{thm}
	Let $\pi$ be a generic representation of $\PGL_2(E)$. Then the following are equivalent:
	\begin{enumerate}[(i)]
		\item $\dim\Hom_{\PGL_2(F)}(\pi,\omega_{E/F}) =1 $;
		\item the Langlands parameter $\phi_{\pi}$ is conjugate-symplectic;
		\item there exists a parameter $\tilde{\phi}:WD_F\rightarrow\SL_2(\mathbb{C})$ such that $\tilde{\phi}|_{WD_E}=\phi_{\pi}$;
		\item $\pi$ is $(PD^\times(F),\omega_{E/F} )$-distinguished or $\pi=\pi(\chi_E,\chi_E^{-1})$ with $\chi_E|_{F^\times}=\omega_{E/F}$ and $\chi_E^2\neq\mathbf{1}$.
	\end{enumerate}
\end{thm}
\begin{proof}
	See \cite[Theorem 6.2]{gan22arithmeticity} and \cite[the Main Theorem (Local)]{hengfei2016new}.
\end{proof}
\subsection{The Prasad conjecture for $\PGSp_4$}
Recall that if $\mathbf{G}=\PGSp_4$, then $\hat{\mathbf{G}}=\mathrm{Spin}_5(\mathbb{C})\cong \Sp_4(\mathbb{C}),~G^{op}=\PGSp_4$ and $\chi_\mathbf{G}=\omega_{E/F}.$
Let $\tau$ be a representation of $\PGSp_4(E)$ with enhanced $L$-parameter $(\phi_\tau,\lambda_\tau )$. Assume that the $L$-packet $\Pi_{\phi_\tau}$ is generic.
The Prasad conjecture for $\PGSp_4$ implies the following:
\begin{enumerate}[(i)]
	\item If $\tau$ is $(\PGSp_4(F),\omega_{E/F})$-distinguished, then 
	\begin{itemize}
		\item $\Pi_{\phi^s_\tau}=\Pi_{\phi^\vee_\tau} ,$ an equality of $L$-packets and
		\item $\phi_\tau={\tau'}|_{WD_E} $ for some parameter ${\tau' }:WD_F\longrightarrow \Sp_4(\mathbb{C}).$ 
	\end{itemize} 
	\item If $\tau$ is generic and there exists $\tau':WD_F\longrightarrow\Sp_4(\mathbb{C})$ such that $\tau'|_{WD_E}=\phi_\tau $, then $\tau$ is $(\PGSp_4(F),\omega_{E/F})$-distinguished.
	\item 	Assume that $\phi_\tau={\tau'}|_{WD_E}$ for some parameter ${\tau'}:WD_F\longrightarrow \Sp_4(\mathbb{C})$.
	If $\tau$ is a discrete series representation,
	then we set
	\[F(\phi_\tau)=\{\tau':\tau'|_{WD_E}=\phi_\tau \}=\bigcup_i \mathcal{O}(\tau_i') \]
	where $\mathcal{O}(\tau'_i)=\{\tau'_i,~\omega_{E/F}\cdot \tau'_i \},$ which may be a singleton.
	Given a parameter $\tau_i':W_F\longrightarrow \Sp_4(\mathbb{C})$ with $\phi_\tau$ its restriction to $WD_E$  and $\tau_i'\cdot\omega_{E/F}=\tau_i',$  there exists an element $g_i\in Z(\phi_\tau)$  such that
	\[(\tau_i'\cdot\omega_{E/F})(x)=g_i\tau_i'(x)g_i^{-1}\]
	for all $x\in WD_F$
	and so  $g_i$ normalizes $Z(\tau_i').$ Then
	$ \Hom_{\PGSp_4(F)}(\tau,\omega_{E/F} )\neq0$ if $\lambda_\tau(g_i)=1$ and
	$\Hom_{\PGSp_{1,1}(F)}(\tau,\omega_{E/F})\neq0$
	if  $\lambda_\tau(g_i)=-1.$
	In this case, $A_G(\tau_i')\subset H^1(W_F,\PGSp_4)$ is trivial and \[\mathcal{C}_i=\begin{cases}
	\{\PGSp_4\},& \mbox{ if }\lambda(g_i)=1,\\
	\{\PGSp_{1,1} \},&\mbox{ if }\lambda(g_i)=-1.
	\end{cases} \]
	If $\tau_i'\neq\tau_i'\omega_{E/F},$ then $A_G(\tau_i')=H^1(F,\PGSp_4)$ and $\mathcal{C}_i=\{\PGSp_4,\PGSp_{1,1} \}.$
	Set $G_\alpha$ to be $\PGSp_4$ or $\PGSp_{1,1}$.
	Then
	\[\dim \Hom_{G_\alpha(F)}(\tau,\omega_{E/F} )=\sum_im(\lambda,\tau_i') 1_{\mathcal{C}_i}(G_\alpha)/d_0(\tau'_i) ,\]
	where $m(\lambda,\tau_i')$ is the multiplicity of the trivial representation $\mathbf{1}$ contained in the restricted representation $\lambda|_{\pi_0(Z(\tau'_i))}.$
	\item If $\Pi_{\phi_{\tau}}$ is generic, then we have \eqref{equaforsum}, i.e.,
	\[\dim\Hom_{\PGSp_4(F)}(\tau,\omega_{E/F})+\dim \Hom_{\PGSp_{1,1}(F)}(\tau,\omega_{E/F})=\sum_{\varphi\in F(\phi_\tau)}m(\lambda,\varphi)\cdot\frac{\deg\Phi(\varphi)}{d_0(\varphi)}. \]
\end{enumerate}

Let us start to verify the Prasad conjecture for $\PGSp_4$.
\begin{thm}
	Let $\tau$ be a generic representation of $\PGSp_4(E)$. Then $\tau$ is $(\PGSp_4(F),\omega_{E/F})$-distinguished if and only if there exists a parameter $\tau':WD_F\rightarrow\Sp_4(\mathbb{C})$ such that $\tau'|_{WD_E}=\phi_{\tau}$.
\end{thm}
\begin{proof}
	Assume that $\tau=\theta(\Pi\boxtimes\chi)$ with $\chi=\mathbf{1}$. Fix $s\in W_F\setminus W_E. $
	\begin{enumerate}[(i)]
		\item If $\tau$ is $(\PGSp_4(F),\omega_{E/F})$-distinguished, then $\phi_{\Pi}$ is conjugate-symplectic and so $\Pi_{\phi_{\tau}^s}=\Pi_{\phi^\vee_\tau}=\Pi_{\phi_\tau}$. If $\phi_{\Pi}$ is irreducible, then we can repeat the process in \S\ref{7.4.1} to obtain that there exists a parameter $\tau':WD_F\rightarrow\Sp_4(\mathbb{C}) $ such that $\tau'|_{WD_E}=\phi_\tau $. If $\phi_{\Pi}=\rho_1\oplus\rho_2$ is reducible and $\rho_1$ is irreducible, then either $\rho_1^s=\rho_2^\vee$ or both $\rho_1$ and $\rho_2$ are conjugate-symplectic.
		\begin{itemize}
			\item If $\rho_1^s=\rho_2^\vee$, then there are two subcases. If $\rho_2^\vee=\rho_2$, then  $\rho_1^s=\rho_2$. Set $\tau'=\Ind_{WD_E}^{WD_F}\rho_1$ if $\rho_1\neq\rho_2$. If $\rho_1=\rho_2=\rho_2^\vee$, then $\rho_1^s=\rho_1$ and so there exists a parameter $\tilde{\rho}_1:WD_F\rightarrow\GL_2(\mathbb{C})$ such that $\tilde{\rho}_1|_{WD_E}=\rho_1$. Set $\tau'=\tilde{\rho}_1\oplus\tilde{\rho}_1^\vee$. If $\rho_2^\vee\neq \rho_2$, then $\rho_2^\vee=\rho_1$. Thus $\rho_1^s=\rho_1$ and $\tau'=\tilde{\rho}_1\oplus\tilde{\rho}_1^\vee$.
			\item If both $\rho_1$ and $\rho_2$ are conjugate-symplectic, then 
			\[\tau'=\begin{cases}
			\Ind_{WD_E}^{WD_F}\rho_1,&\mbox{ if }\rho_1^s=\rho_2\neq\rho_1; \\
			\tilde{\rho}_1\oplus\tilde{\rho}_1^\vee,&\mbox{ if }\rho_1^s=\rho_1.
			\end{cases} \]
		\end{itemize}
	If neither $\rho_1$ nor $\rho_2$ is irreducible, then $\phi_{\tau}$ belongs to the endoscopic case. Thanks to Theorem \ref{localgspperiod}(ii), either $\rho_1^s=\rho_2^\vee$ or both $\rho_1$ and $\rho_2$ are conjugate-symplectic. The argument is similar and we omit it here.
	Therefore, there exists $\tau':WD_F\rightarrow\Sp_4(\mathbb{C})$ such that $\tau'|_{WD_E}=\phi_{\tau}$.
	\item Conversely, if there exists $\tau':WD_F\rightarrow\Sp_4(\mathbb{C})$ such that $\tau'|_{WD_E}=\phi_\tau $, then it suffices to show that $\phi_{\Pi}$ is conjugate-symplectic. Because the nongeneric member in the $L$-packet $\Pi_{\phi_{\tau}}$ is not $(\GSp_4(F),\omega_{E/F})$-distinguished due to Theorem \ref{localgspperiod}(i) if $\Pi_{\phi_{\tau}}$ contains more than one representation. Assume that
	\[\phi_\tau:WD_E\longrightarrow\Sp(V,<-,->)=\Sp_4(\mathbb{C}) \]
	and $$\phi_{\Pi}=i\circ\phi_\tau:WD_E\longrightarrow\GL(V)  $$
	where $i:\Sp_4(\mathbb{C})\rightarrow\GL(V)$ is the embedding between the $L$-groups.
	Then we set $$B(m,n)=<m,\tau'(s)^{-1}n>$$ for $m,n\in V$. It is easy to check that $B(\phi_{\Pi}(t)m,\phi_{\Pi}^s(t)n )=B(m,n)$ and
	\[B(m,\phi_{\Pi}(s^2)n )=<m,\tau'(s)n>=-<\tau'(s)n,m>=-<n,\tau'(s)^{-1}m>=-B(n,m). \]
	Therefore, the bilinear form $B$ on $V$ implies that $\phi_{\Pi}$ is conjugate-symplectic.
	\end{enumerate}
We have finished the proof.
\end{proof}
\subsection{The proof of Theorem \ref{prasadfordisc}}Before we give the proof of Theorem \ref{prasadfordisc}, we will use the results in Theorem \ref{localgspperiod} and Theorem \ref{innerformperiod} to study the equality \eqref{equaforsum} in detail. According to the Langlands parameter $\phi_\tau $, we divive them into three cases:
\begin{itemize}
	\item the endoscopic case,
	\item the discrete series but non-endoscopic case and
	\item $\phi_\tau=\rho+\rho\nu $ with $\nu\neq\mathbf{1}$ and $\nu\det\rho=\mathbf{1}$.
\end{itemize}
 Set $S_\phi=\pi_0(Z(\phi))$ to be the
component group.			We identify the characters of $W_F$ and the characters of $F^\times$ via the local class field theory. 
\subsubsection{ Endoscopic case} Given $\phi_\tau=\phi_{1}\oplus\phi_{2},$ there are two cases: $\phi_1=\phi_2$ and $\phi_1\neq\phi_2.$ 
\begin{enumerate}[(A)]
	\item If $\phi_1=\phi_2=\rho$ are irreducible, then the L-packet $\Pi_{\phi_\tau}=\{\pi^+,\pi^- \}$ and $S_{\phi_\tau}=\mathbb{Z}/2\mathbb{Z}$, where $\pi^-$ (resp. $\pi^+$) is a nongeneric (resp. generic) representation of $\PGSp_4(E)$. There are two subcases:
	\begin{enumerate}[label=(A\arabic*)]
\item	If $\rho$ is conjugate-orthogonal, then 
 \[\dim \Hom_{\PGSp_{1,1}(F)}(\pi^+,\omega_{E/F} )=0=\dim \Hom_{\PGSp_4(F)}(\pi^-,\omega_{E/F}) \]
 and
	\[\dim \Hom_{\PGSp_{1,1}(F)}(\pi^-,\omega_{E/F})=1=\dim \Hom_{\PGSp_4(F)}(\pi^+,\omega_{E/F}). \]
	On the Galois side, there is only one extension
	$\tilde{\phi}=\bar{\rho}\oplus\bar{\rho}\cdot\omega_{E/F}$ with $$\deg\Phi(\tilde{\phi})=2\mbox{  and  } S_{\tilde{\phi}}=\{\mathbf{1}\}\rightarrow S_{\phi_\tau},$$ where
	$\bar{\rho}:WD_F\rightarrow \GL_2(\mathbb{C})\times W_F$ with $\det\bar{\rho}=\omega_{E/F}.$ Note that $\tilde{\phi}=\tilde{\phi}\cdot\omega_{E/F}$. Then  $\pi^+$ supports a period on the trivial pure inner form and $\pi^-$ supports a period on a nontrivial pure inner form.
\item 	
	If $\rho$ is conjugate-symplectic, then 
	\[\dim \Hom_{\PGSp_{1,1}(F)}(\pi^-,\omega_{E/F})=0=\dim \Hom_{\PGSp_4(F)}(\pi^-,\omega_{E/F}) \]
	and \[\dim \Hom_{\PGSp_{1,1}(F)}(\pi^+,\omega_{E/F} )=1,~\dim \Hom_{\PGSp_4(F)}(\pi^+,\omega_{E/F} )=2. \]
	In this case, $\rho$ has two extensions $\bar{\rho} $ and $\bar{\rho}\cdot\omega_{E/F},$ where $\bar{\rho}:WD_F\longrightarrow \SL_2(\mathbb{C}).$
	There are three choices for the extension $\tilde{\phi}:WD_F\longrightarrow \Sp_4(\mathbb{C})$ with $\deg\Phi(\tilde{\phi})=1 $
	\begin{itemize}
		\item $\tilde{\phi}^{++}=\bar{\rho}\oplus\bar{\rho}$ with $S_{\tilde{\phi}^{++}}=\mathbb{Z}/2\mathbb{Z}\cong S_{\phi_\tau};$
		\item $\tilde{\phi}^{+-}=\bar{\rho}\oplus\bar{\rho}\cdot\omega_{E/F}$ with $S_{\tilde{\phi}^{+-}}=\mathbb{Z}/2\mathbb{Z}\times\mathbb{Z}/2\mathbb{Z}\longrightarrow S_{\phi_\tau} $(sum map);
		\item $\tilde{\phi}^{--}=\bar{\rho}\cdot\omega_{E/F}\oplus \bar{\rho}\cdot\omega_{E/F}$ with $S_{\tilde{\phi}^{--}}=\mathbb{Z}/2\mathbb{Z}\cong S_{\phi_\tau}.$
	\end{itemize}
	The parameters $\tilde{\phi}^{++}$ and $\phi^{--}$ are in the same orbit under the twisting by $\omega_{E/F},$ which corresponds to both pure inner forms. The parameter $\tilde{\phi}^{+-}$ is fixed under twisting by $\omega_{E/F},$ which supports a period on the trivial pure inner form.
	\item If $\rho$ is not conjugate-self-dual, then both the Galois side and the automorphic side are $0$.
\end{enumerate}
	\item\label{B} If $\phi_1\neq\phi_2$  are both irreducible, then the L-packet of $\PGSp_4$ is $\Pi_{\phi_\tau}=\{\pi^{++},\pi^{--} \}$ and $$S_{\phi_\tau}=\mathbb{Z}/2\mathbb{Z}\times\mathbb{Z}/2\mathbb{Z}.$$
\begin{enumerate}[label=(B\arabic*)]
	\item 
	If $\phi_1$ and $\phi_2$ both extend to $L$-parameters of $\PGL_2(F),$ i.e. both are conjugate-symplectic, then one has $\phi_1^s\neq\phi_2,$
	$$	\dim \Hom_{\PGSp_{1,1}(F)}(\pi^{++},\omega_{E/F})=2=\dim \Hom_{\PGSp_4(F)}(\pi^{++},\omega_{E/F})$$
	and $$\dim \Hom_{\PGSp_{1,1}(F)}(\pi^{--},\omega_{E/F})=0=\dim \Hom_{\PGSp_4(F)}(\pi^{--},\omega_{E/F}).$$
	On the Galois side, there are also four ways of extending $\phi_\tau.$ For each such extension $\tilde{\phi},$ one has $\deg\Phi(\tilde{\phi})=1$ and the equality of component group
	\[S_{\tilde{\phi}}=S_{\phi_\tau}=\mathbb{Z}/2\mathbb{Z}\times\mathbb{Z}/2\mathbb{Z}. \]
	So only the representation $\pi^{++}$ in the L-packet can support a period. And there are $2$ orbits in $F(\phi_\tau)$ under twisting by $\omega_{E/F},$ each of size $2.$
	\item 
	If $\phi_1$ and $\phi_2$ do not extend to $L$-parameters of
	$\PGL_2(F),$ but $\phi_1^s=\phi_2=\phi_2^\vee,$ then
	$$\dim \Hom_{\PGSp_{1,1}(F)}(\pi^{++},\omega_{E/F})=0=\dim \Hom_{\PGSp_4(F)}(\pi^{--},\omega_{E/F})$$ and
	\[\dim \Hom_{\PGSp_{1,1}(F)}(\pi^{--},\omega_{E/F})=1=\dim \Hom_{\PGSp_4(F)}(\pi^{++},\omega_{E/F}) \]
	There is a unique way of extending $\phi_\tau=\phi_1\oplus\phi_2$ to $\tilde{\phi}:WD_F\rightarrow \Sp_4(\mathbb{C})$. Namely, $\tilde{\phi}=\Ind_{WD_E}^{WD_F}\phi_1$ is an irreducible $4$-dimensional symplectic representation, with a component group
	\[S_{\tilde{\phi}}=\mathbb{Z}/2\mathbb{Z}\hookrightarrow S_{\phi_\tau}(\mbox{diagonal embedding}). \]
	And $S_{\phi_{\tau}}^{\Gal(E/F)}=S_{\tilde{\phi}}$.
	So $\pi^{++}$ supports a period on the trivial pure inner form and $\pi^{--}$ supports a period on the nontrivial pure inner form.
\end{enumerate}
	\item If $\phi_1=\chi_{1}\oplus\chi_{1}^{-1}$ is reducible, then there is only one element in the L-packet, i.e. $|\Pi_{\phi_\tau}|=1.$ There are two cases: $\phi_1=\phi_2$ and $\phi_1\neq\phi_2.$
	\begin{enumerate}[label=(C\arabic*) ]
		\item If $\phi_1=\phi_2,$  there are three subcases.
		\begin{enumerate}[label=(C1.\roman*)]
			\item If $\chi_{1}=\chi_{1}^s=\chi_F|_{W_E},$ then $S_{\phi_\tau}=1$ 
			and
			\[\dim \Hom_{\PGSp_{1,1}(F)}(\tau,\omega_{E/F})=2= \dim \Hom_{\PGSp_4(F)}(\tau,\omega_{E/F} ) .\]
		\begin{itemize}
			\item 	If $\chi_F^2\neq\omega_{E/F},$ then there are two ways to extend $L$-parameters of $\PGL_2(F),$ denoted by $\bar{\rho}$ and $\bar{\rho}\cdot\omega_{E/F}$. Thus there are $3$ ways of extending $\phi_\tau,$ which are $\tilde{\phi}^{++},\tilde{\phi}^{--}$ and $\tilde{\phi}^{+-}.$ 
			Moreover, $\deg\Phi(\tilde{\phi}^{++})=1=\deg\Phi(\tilde{\phi}^{--})$ and
			$\deg\Phi(\tilde{\phi}^{+-})=2.$
			\item 
				If $\chi_F^2=\omega_{E/F},$ then there is only one way to extend $\phi_\tau.$ Denote it by $\tilde{\phi}.$ Then $$\deg\Phi(\tilde{\phi})=4.$$
		\end{itemize}	
			\item If $\chi_{1}\neq\chi_{1}^{-1}$ but $\chi_{1}|_{F^\times}=\omega_{E/F},$ then $S_{\phi_\tau}=1$ and
			\[\dim \Hom_{\PGSp_{1,1}(F)}(\tau,\omega_{E/F} )=0\mbox{ and }\dim \Hom_{\PGSp_4(F)}(\tau,\omega_{E/F})=1. \]
			There is only one way to extend $\phi_1,$ denoted by 
			$$\bar{\rho}=\Ind_{WD_E}^{WD_F}\chi_{1}:WD_F\rightarrow\SL_2(\mathbb{C}).$$ 
			Then
			$\tilde{\phi}=\bar{\rho}\oplus\bar{\rho}$ with $S_{\tilde{\phi}}=\mathbb{Z}/2\mathbb{Z}$ and $\deg\Phi(\tilde{\phi})=1.$ Note that $\tilde{\phi}\cdot\omega_{E/F}=\tilde{\phi}$.
			 Then $\tilde{\phi}$ supports a period on the trivial pure inner form.
			\item If $\chi_{1}\neq\chi_{1}^{-1}$ but $\chi_{1}|_{F^\times}=\mathbf{1}$, then $S_{\phi_\tau}=1$ and
			\[\dim \Hom_{\PGSp_{1,1}(F)}(\tau,\omega_{E/F} )=0\mbox{ and }\dim \Hom_{\PGSp_4(F)}(\tau,\omega_{E/F})=1. \]
			On the Galois side, there is only one choice $\tilde{\phi}=\bar{\rho}\oplus\bar{\rho}$ and $S_{\tilde{\phi}}=\mathbf{1},$ where
			$$\bar{\rho}=\Ind_{WD_E}^{WD_F}\chi_{1}:WD_F\rightarrow \GL_2(\mathbb{C})$$ with $\det\rho=\omega_{E/F}.$
			Since $\tilde{\phi}=\tilde{\phi}\cdot\omega_{E/F},$ it picks up only the trivial pure inner form.
		\end{enumerate} 
		\item If $\phi_1\neq\phi_2,$  there are several subcases:
		\begin{enumerate}[label=(C2.\roman*)]
			\item If $\chi_{1}=\chi_{1}^s=\chi_F|_{W_E}$ and $\phi_2$ is irreducible and conjugate-symplectic , then  
			$S_{\phi_\tau}=\mathbb{Z}/2\mathbb{Z}$ and $$\dim \Hom_{\PGSp_{1,1}(F)}(\tau,\omega_{E/F})=2=\dim \Hom_{\PGSp_4(F)}(\tau,\omega_{E/F} ).$$
		\begin{itemize}
			\item 	If $\chi_F^2\neq\omega_{E/F},$ then there are four ways of extending ${\phi_\tau}$ and for each such extension $\tilde{\phi},$ one has 
			$S_{\tilde{\phi}}=\mathbb{Z}/2\mathbb{Z}\cong S_{\phi_\tau}.$
			There are two orbits under the twisting by $\omega_{E/F},$
			each of size $2.$
			\item 	If $\chi_F^2=\omega_{E/F},$ then there are two ways of extending $\phi_\tau.$ For each such extension $\tilde{\phi},$
			one has $\deg\Phi(\tilde{\phi})=2.$ There is one orbit under the twisting by $\omega_{E/F}.$		
		\end{itemize}		 
		In this case, the identity
	\begin{equation}\label{stableequation}
	\dim\Hom_{G_\alpha(F)}(\tau,\chi_G)=\sum_i m(\lambda,\tilde{\phi}_i)\mathbf{1}_{\mathcal{C}_i}(G_\alpha)\cdot\frac{\deg\Phi(\tilde{\phi}_i)}{d_0(\tilde{\phi}_i)} 
	\end{equation}
	holds for $G_\alpha=\PGSp_4$ and $\PGSp_{1,1}.$
			\item If $\chi_{1}=\chi_{1}^s=\chi_F|_{W_E}$ and $\chi_{2}=\chi_{2}^s=\chi_F'|_{W_E},$ where $\phi_2=\chi_{2}\oplus\chi_{2}^{-1},$ then $S_{\phi_\tau}=1$ and
			\[\dim \Hom_{\PGSp_{1,1}(F)}(\tau,\omega_{E/F})=2=\dim \Hom_{\PGSp_4(F)}(\tau,\omega_{E/F} ). \]
		\begin{itemize}
	\item 
			If neither $\chi_F^2$ nor ${\chi_F'}^2$ equals $\omega_{E/F},$ then	there are four ways of extending $\phi_{\tau}.$
			There are two orbits under the twisting by $\omega_{E/F},$
			each of size $2.$
			\item
			If $\chi_F^2=\omega_{E/F}$ and $\chi_F'^2\neq\omega_{E/F},$ then
			there are two ways to extend $\phi_\tau$ and for each such extension $\tilde{\phi},$ one has $S_{\tilde{\phi}}=1=S_{\phi_\tau}$ and $\deg\Phi(\tilde{\phi})=2.$ There is one orbit under the twisting by $\omega_{E/F},$ which corresponds to both pure inner forms.
			\item
			If $\chi_F^2=\chi_F'^2=\omega_{E/F},$ then there is only one way to extend $\phi_\tau.$ For this extension $\tilde{\phi},$ one has $\deg\Phi(\tilde{\phi})=4.$
		\end{itemize}
			\item If $\chi_{1}\neq\chi_{1}^{-1}$ but $\chi_{1}$ is conjugate-symplectic, and $\phi_2$ is irreducible and conjugate-symplectic, then $S_{\phi_\tau}=\mathbb{Z}/2\mathbb{Z}$ and
			\[\dim \Hom_{\PGSp_{1,1}(F)}(\tau,\omega_{E/F} )=1=\dim \Hom_{\PGSp_4(F)}(\tau,\omega_{E/F}). \]
			There are two extensions $\tilde{\phi}=\bar{\rho}_1\oplus\bar{\rho}_2$
			or $\bar{\rho}_1\oplus\bar{\rho}_2\omega_{E/F}$ with $S_{\tilde{\phi}}=\mathbb{Z}/2\mathbb{Z}\times\mathbb{Z}/2\mathbb{Z},$ where $\bar{\rho}_i:WD_F\rightarrow \SL_2(\mathbb{C})$ satisfies $\bar{\rho}_i|_{WD_E}=\phi_i.$ Here the map $S_{\tilde{\phi}}\rightarrow S_{\phi_\tau}$ is given by $$(x,y)\mapsto x+y.$$ 
			There is one orbit under the twisting by $\omega_{E/F},$
			which corresponds to both pure inner forms.
			\item If $\chi_{1}\neq\chi_{1}^{-1}$ but $\chi_{1}$ is conjugate-symplectic, and $\chi_{2}=\chi_{2}^s=\chi_F'|_{W_E}$
			where $\phi_2=\chi_{2}\oplus\chi_{2}^{-1},$ then $S_{\phi_\tau}=1$ and
			\[\dim \Hom_{\PGSp_{1,1}(F)}(\tau,\omega_{E/F} )=1=\dim \Hom_{\PGSp_4(F)}(\tau,\omega_{E/F}). \]
		\begin{itemize}
	\item		If $\chi_F'^2\neq\omega_{E/F},$ 
			then	there are two ways to extend $\phi_\tau$.
			 Set $\tilde{\phi}=\bar{\rho}_1\oplus\bar{\rho}_2$ or $\bar{\rho}_1\oplus\bar{\rho}_2\omega_{E/F}$ with $S_{\tilde{\phi}}=\mathbb{Z}/2\mathbb{Z} .$
			There is one orbit under the twisting by $\omega_{E/F},$
			which corresponds to both pure inner forms.
			\item
			If $\chi_F'^2=\omega_{E/F},$ there is one way to extend $\phi_{\tau}$. Set $\tilde{\phi}=\bar{\rho}_1\oplus\chi_F'\oplus\chi_F'\omega_{E/F}.$ And
			\[\deg\Phi(\tilde{\phi})=2. \]
			Note that the identity \eqref{stableequation} fails in this case while the identity \eqref{equaforsum}    still holds.
		\end{itemize}
			\item If $\phi_1$ and $\phi_2$ are reducible and four different characters $\chi_{1},\chi_{1}^{-1},\chi_2$ and $\chi_2^{-1}$ satisfy $$\chi_{1}|_{F^\times}=\omega_{E/F}=\chi_2|_{F^\times},$$ 
			then $S_{\phi_\tau}=1$ and
			\[\dim \Hom_{\PGSp_{1,1}(F)}(\tau,\omega_{E/F})=0,\]
			$\dim \Hom_{\PGSp_4(F)}(\tau,\omega_{E/F})=1. $
			There is only one extension $\tilde{\phi}=\bar{\rho}_1\oplus\bar{\rho}_2$
			with $S_{\tilde{\phi}}=\mathbb{Z}/2\mathbb{Z}\times\mathbb{Z}/2\mathbb{Z}. $
			Since $\tilde{\phi}=\tilde{\phi}\cdot\omega_{E/F},$ it picks up the trivial pure inner form.
			\item If $\phi_1^s=\phi_2^\vee=\phi_2$ and $\phi_1$ is not conjugate-symplectic, then  $S_{\phi_\tau }=1$ and
			\[\dim \Hom_{\PGSp_{1,1}(F)}(\tau,\omega_{E/F})=0,\dim \Hom_{\PGSp_4(F)}(\tau,\omega_{E/F}) =1.\]
			There is only one extension $$\tilde{\phi}= \Ind_{WD_E}^{WD_F}\phi_1:WD_F\rightarrow \Sp_4(\mathbb{C})$$
			with the component group $S_{\tilde{\phi}}=\mathbb{Z}/2\mathbb{Z}.$ Since $\tilde{\phi}=\tilde{\phi}\cdot\omega_{E/F},$ it picks up the trivial pure inner form.
		\end{enumerate}
		It is easy to check that the identity \eqref{equaforsum} holds when $\Pi_{\phi_{\tau}}$ is generic, i.e.,
		\[\dim\Hom_{\PGSp_4(F)}(\tau,\omega_{E/F})+\dim\Hom_{ \PGSp_{1,1}(F)}(\tau,\omega_{E/F})=\sum_{\tilde{\phi}\in F(\phi_\tau)}m(\lambda,\tilde{\phi})\cdot\frac{\deg\Phi(\tilde{\phi})}{d_0(\tilde{\phi})}.\]
	\end{enumerate}
\end{enumerate}
\subsubsection{ Discrete and non-endoscopic case}\label{subsec:non-end} Assume that $\phi_\tau $ is irreducible and so $\Pi_{\phi_{\tau}}$ is a singleton.
Given a parameter $\phi_\tau,$  which is non-endoscopic, the theta lift $\Theta_4^+(\tau)$ from $\PGSp_4(E)$ to $\rm PGSO_{2,2}(E)$ is zero. 

If $\phi_\tau$ is conjugate-symplectic, then
\[\dim \Hom_{\PGSp_{1,1}(F)}(\tau,\omega_{E/F})=1=\dim \Hom_{\PGSp_4(F)}(\tau,\omega_{E/F}) .\]
There are two extensions $\tilde{\phi}$ and $\tilde{\phi}\cdot\omega_{E/F}$ with a component group
$S_{\tilde{\phi}}=S_{\phi_\tau}=\mathbb{Z}/2\mathbb{Z}. $			
There is one orbit under the twisting by $\omega_{E/F},$
which corresponds to both pure inner forms.
\subsubsection{ Generic but neither discrete nor endoscopic case}
If $\phi_\tau=\rho\oplus\rho\nu,\det\rho=\nu^{-1}\neq\mathbf{1},$
then $S_{\phi_\tau}=1.$ 
There are two cases:
\begin{itemize}
	\item 			 If $\phi_\tau$ is conjugate-symplectic and $\rho^s=\rho$, then
	\[\dim \Hom_{\PGSp_{1,1}(F)}(\tau,\omega_{E/F})=1=\dim \Hom_{\PGSp_4(F)}(\tau,\omega_{E/F}) .\]
	There are two extensions
	$\tilde{\phi}=\tilde{\rho}+\tilde{\rho}^\vee$ and $\tilde{\phi}\cdot\omega_{E/F}$ where $\tilde{\rho}:WD_F\rightarrow \GL_2(\mathbb{C})$  satisfies $\tilde{\rho}|_{WD_E}=\rho.$
	\item If $\phi_\tau$ is conjugate-symplectic and $\rho^s\neq\rho,$ then
	\[\dim \Hom_{\PGSp_4(F)}(\tau,\omega_{E/F})=1\mbox{  and  } \dim \Hom_{\PGSp_{1,1}(F)}(\tau,\omega_{E/F})=0. \]
	There is only one extension $\tilde{\phi}=\Ind_{WD_E}^{WD_F}\rho$ such that $\tilde{\phi}|_{WD_E}=\phi_\tau.$
\end{itemize}
\begin{proof}
	[Proof of Theorem \ref{prasadfordisc}] It follows from the discussions in the endoscopic cases \ref{B} in \S7.4.1 
	and the discrete and non-endoscopic case in \S\ref{subsec:non-end}.
\end{proof}

\subsection{Further discussion} Let $E$ be a quadratic extension over a non-archimedean field $F.$ Let $\mathbf{G}$ be a quasi-split reductive group defined over $F.$
Let $\tau$ be an irreducible representation of $\mathbf{G}(E)$ with an enhanced $L$-parameter $(\phi_\tau,\lambda).$ 
Assume that $F(\phi_\tau)=\cup_i\mathcal{O}(\tilde{\phi}_i)$ where $\tilde{\phi}_i|_{WD_E}=\phi_\tau.$
\par
If for each orbit $\mathcal{O}(\tilde{\phi}_i),$  the coset $\mathcal{C}_i\subset H^1(W_F,\mathbf{G})$ contains all pure inner forms satisfying $G_\alpha(E)=\mathbf{G}(E),$ then  $\phi_\tau$ is called  a '$full$' L-parameter of $\mathbf{G}(E)$, in which case $1_{\mathcal{C}_i}(G_\alpha)\equiv1$ in \eqref{stableequation}. 
\par
Assume that $\tau$ belongs to a generic L-packet with Langlands parameter $\phi_\tau:WD_E\rightarrow {}^L\mathbf{G}$ and that $\phi_\tau$ is '$full$'. Then there is a conjectural identity
\begin{equation}\label{conjectureiden}\dim\Hom_{G_{\alpha}}(\tau,\chi_\mathbf{G})=\sum_im(\lambda,\tilde{\phi}_i)\cdot\frac{\deg\Phi_\ast(\tilde{\phi}_i)}{d_0(\tilde{\phi}_i)} \end{equation}
for any pure inner form $G_\alpha\in H^1(W_F,\mathbf{G})$ satisfying $G_\alpha(E)=\mathbf{G}(E).$ 
\par
If $H^1(W_F,\mathbf{G})$ is trivial, then any $L$-parameter $\phi_\tau$ is '$full$'. So the conjectural identity \eqref{conjectureiden} holds for   $\mathbf{G}=\GL_2.$ In fact, it holds for $\mathbf{G}=\PGL_2$  as well.

\bibliographystyle{amsplain}
\bibliography{GSp}

\end{document}

%% file: localtheta.tex
\section{The Local Theta Correspondences for Similitudes}
In this section, we will briefly recall some results about the local theta correspondence, following \cite{gan2011theta,kudla1996notes,roberts2001global}.
\par
Let $F$ be a local field of characteristic zero.
Consider the dual pair $\Oo(V)\times \Sp(W).$
For simplicity, we may assume that $\dim V$ is even. Fix a nontrivial additive character $\psi$ of $F.$
Let $\omega_\psi$ be the Weil representation for $\Oo(V)\times \Sp(W).$ 
If $\pi$ is an irreducible smooth representation of $\Oo(V)$ (\resp $\Sp(W)$), the maximal $\pi$-isotypic quotient of $\omega_\psi$ has the form 
\[\pi\boxtimes\Theta_\psi(\pi) \]
for some smooth  representation $\Theta_\psi(\pi)$ of $\Sp(W)$ (resp. some smooth representation $\Theta_\psi(\pi)$ of $\Oo(V)$). We call $\Theta_\psi(\pi )$ or $\Theta_{V,W,\psi}(\pi)$
the big theta lift of $\pi.$ It is known that $\Theta_\psi(\pi)$ is of finite length and hence is admissible. Let $\theta_\psi(\pi)$ or $\theta_{V,W,\psi}(\pi)$ be the maximal semisimple quotient of $\Theta_\psi(\pi),$ which is called the small theta lift of $\pi.$ 
\begin{theorem}[Howe duality conjecture]\cite{gan2014howe,gan2014proof}\label{localhowe}
 One has
 \begin{itemize}
 	\item $\theta_\psi(\pi)$ is irreducible whenever $\Theta_\psi(\pi)$ is non-zero.
 	\item the map $\pi\mapsto \theta_\psi(\pi)$ is injective on its domain.
 \end{itemize}
\end{theorem}
It has been proved by Waldspurger \cite{waldspurger1990demonstration} when $p\neq2$.
\par
Then we 
extend the Weil representation to the case of similitude groups. Let $\lambda_V$ and $\lambda_W$
be the similitude factors of $\GO(V)$ and $\GSp(W)$ respectively. 
We shall consider the group
\[R=\GO(V)\times \GSp^+(W) \]
where $\GSp^+(W)$ is the subgroup of $\GSp(W)$ consisting of elements $g$ such that $\lambda_W(g)$
lies in the image of $\lambda_V.$ Define
\[R_0=\{(h,g)\in R|~\lambda_V(h)\lambda_W(g) =1\} \]
to be the subgroup of $R.$ The Weil representation $\omega_\psi$ extends naturally to the group $R_0$
via \[\omega_\psi(g,h)\phi=|\lambda_V(h)|_F^{-\frac{1}{8}\dim V\cdot\dim W } \omega(g_1,1)(\phi\circ h^{-1}) \]
where $|-|_F$ is the absolute value on $F$ and \[g_1=g\begin{pmatrix}
\lambda_W(g)^{-1}&0\\0&1
\end{pmatrix}\in \Sp(W). \]
Here the central elements $(t,t^{-1})\in R_0$ acts by the quadratic character $\chi_V(t)^{\frac{\dim W}{2}},$
which is slightly different from the normalization used in \cite{roberts2001global}.
\par
Now we consider the compactly induced representation
\[\Omega=ind_{R_0}^R\omega_\psi. \]
As  a representation of $R,$ $\Omega$ depends only on the orbit of $\psi$ under the evident action of $Im\lambda_V\subset F^\times.$ For example, if $\lambda_V$ is surjective, then $\Omega$ is independent of $\psi.$ For any irreducible representation $\pi$ of $\GO(V)$ (\resp $\GSp^+(W)$),
the maximal $\pi$-isotropic quotient of $\Omega$ has the form
\[\pi\otimes\Theta_\psi(\pi), \]
where $\Theta_\psi(\pi)$ is some smooth representation of $\GSp^+(W)$ (resp. $\GO(V)$). Similarly, we let
$\theta_\psi(\pi)$ be the maximal semisimple quotient of $\Theta_\psi(\pi).$
Note that though $\Theta_\psi(\pi)$ may be reducible, it has a central character $\omega_{\Theta_\psi(\pi)}$
given by
\[\omega_{\Theta_\psi(\pi) }=\chi_V^{\frac{\dim W}{2}}\omega_\pi . \] 
There is an extended Howe conjecture for similitude groups, which says that $\theta_\psi(\pi)$ is irreducible whenever $\Theta_\psi(\pi)$
is non-zero and the map $\pi\mapsto\theta_\psi(\pi)$ is injective on its domain.
It was shown by Roberts \cite{roberts1996theta} that this follows from Theorem \ref{localhowe}. 
\par
Since $\lambda_V$ is surjective in this paper, we have $\GSp^+(W)=\GSp(W).$
 \begin{proposition}\cite[Proposition 2.3]{gan2011locallanglands}
	Suppose that $\pi$ is a supercuspidal representation of $\GO(V)$ (resp. $\GSp(W)$). Then 
	$\Theta_\psi(\pi)$ is either zero or is an irreducible representation of $\GSp(W)$ (\resp $\GO(V)$).
\end{proposition}

 \subsection{First Occurence Indices for pairs of orthogonal Witt Towers} Let $W_n~ (n\geq1)$ be the $2n$-dimensional symplectic vector space with associated symplectic group $\Sp(W_n)$ and consider the two towers of orthogonal groups attached to the quadratic spaces with trivial discriminant. More precisely, let $\mathbb{H}$ be the split $2$-dimensional quadratic space over $F$ and $D$ be the quaternion division algebra over $F$. Let
\[V_{2r}^+= \mathbb{H}^{r}\quad\mbox{and}\quad V_{2r}^-= D(F)\oplus\mathbb{H}^{r} \]
and denote the orthogonal groups by $\Oo(V_{2r}^+)=\Oo_{r,r}$ and $\Oo(V_{2r}^-)=\Oo_{r+4,r}$ respectively. For an irreducible representation $\pi$ of $\Sp(W_n),$ one may consider the theta lifts $\theta_{2r}^+(\pi)$ and $\theta_{2r}^-(\pi)$ to
$\Oo(V^+_{2r})$ and $\Oo(V_{2r}^-)$ respectively, with respect to a fixed non-trivial additive character $\psi.$ Set
\[\begin{cases}
r^+(\pi)=\inf\{r:\theta_{2r}^+(\pi)\neq0 \};\\
r^-(\pi)=\inf\{r:\theta_{2r}^-(\pi)\neq0 \}.
\end{cases} \]
Then Kudla-Rallis \cite{kudla2005first} and Sun-Zhu \cite{sun2012conservation} showed:
\begin{theorem}
	[Conservation Relation] For any irreducible representation $\pi$ of $\Sp(W_n),$ we have
	\[r^+(\pi)+r^-(\pi)=2n=\dim W_n. \]
\end{theorem}
On the other hand, one may consider the mirror situation, where one fixes an irreducible representation $\pi$ of $\Oo(V_{2r})$ and consider its theta lift $\theta_n(\pi)$ to the tower of symplectic groups $\Sp(W_n).$ Then,
with $n(\pi)$ defined in the analogous fashion
\[n(\pi)=\inf\{n:\theta_n(\pi)\neq0 \},  \]
 one has
\[n(\pi)+n(\pi\otimes\det )=2r=\dim V_{2r}. \] 
For similitude groups, this implies that
\[n(\pi)+n(\pi\otimes\nu)=2r, \]
where $\nu$ is the nontrivial character of $\GO(V_{2r})/\GSO(V_{2r}).$

%% file: corestriction.tex
\section{The $\GSp_4(F)$-distinguished representations}\label{sect:GSp(4)}
\subsection{Notation}
\begin{itemize}
	\item $\mathbb{C}$ or $\mathbf{1}$ is the trivial representation.
	\item $\mathbb{H}$ (resp. $\mathbb{H}_E$) is the split $2$-dimensional quadratic space over $F$ (resp. $E$).
	\item $(-,-)_E$ is the Hilbert symbol on $E^\times\times E^\times$.
	\item $\Res_{E/F}V$ is a quadratic space over $F$ while $V$ is a quadratic space over $E$.
	\item $\GSp(W_n)=\GSp_{2n}(F)$ is the symplectic similitude group.
	\item $\GU_2(D)=\GSp_{1,1}$ is the unique inner form of $\GSp_4$.
	\item $\lambda_W$ (resp. $\lambda_V$) is the similitude character of $\GSp_{4}(E)$ (resp. $\GO(V)$).
	\item $\GSp_4(E)^\natural=\{g\in\GSp_4(E)|\lambda_W(g)\in F^\times \}$ is the subgroup of $\GSp_4(E)$ and similarly for $\GO_{2,2}(E)^\natural$.
	\item $P'$ (resp. $P^\natural$) is a parabolic (resp. Siegel parabolic) subgroup of $\GSp_4(E)^\natural$ and $Q^\natural$ is the Siegel parabolic subgroup of $\GO_{2,2}(E)^\natural$.
	And $R_{\bar{P'}}$ (resp. $R_{\bar{P}^\natural}$) is the Jacquet functor with respect to the parabolic subgroup opposite to $P'$ (resp. $P^\natural$).
	\item $R_r(\mathbf{1})$ is the big theta lift to $\GO_{4,4}(F)$ of the trivial representation of $\GSp(W_r)$.
	\item $R^{m,n}(\mathbf{1})$ is the big theta lift to $\GSp_{8}(F)$ of the trivial representation of $\GO_{m,n}(F)$.
	\item $\Sigma$ is a generic representation of $\GO(V)$.
		\item $Q_r$ is the Siegel parabolic subgroup of $H_r=\GO_{r,r}(F)$.
	\item $I_{Q_r}^{H_r}(s)$ is the degenerate Siegel principal series of $H_r$.
	\item $X_4=Q_4\backslash H_4$ is the projective variety.
	\item $\mathcal{I}(s)$ is the degenerate Siegel principal series of $\GSp_{8}(F)$.
	\item $Mat_{m,n}(F)$ is the matrix space over $F$ consisting of all $m\times n$ matrices.
\end{itemize}
\subsection{See-saw identity for orthogonal-symplectic dual pairs}
Following the notation in \cite{prasad1996some},
for a quadratic space $(V,q)$ of even dimension over $E,$
let $\Res_{E/F}V$ be the same space $V$ but now thought of as a vector space over $F$ with a quadratic form
$$q_F(v)=\frac{1}{2}tr_{E/F}q(v).$$
If $W_0$ is a symplectic vector space over $F,$
then $W_0\otimes_F E$ is a symplectic vector space over $E.$
Then we have the following isomorphism of symplectic spaces over $F$
\[\Res_{E/F}[(W_0\otimes_FE)\otimes_EV]\cong W_0\otimes_F \Res_{E/F}V=:\mathbf{W}. \]
There is a pair
\[\Big(\GSp(W_0),\GO(\Res_{E/F}V)\Big)\mbox{  and  }\Big(\GSp(W_0\otimes_F E),\GO(V)\Big) \]
of similitude dual reductive pairs in the symplectic similitude group $\GSp(\mathbf{W}).$
A pair $(G,H)$ and $(G',H')$ of dual reductive pairs in a symplectic similitude group is called a see-saw pair if $H\subset G'$
and $H'\subset G.$ The following lemma is quite useful in this section. See \cite[Lemma p. 6]{prasad1996some}.

\begin{lem}\label{localseesaw}
	For a see-saw pair of dual reductive pairs $(G,H)$
	and $(G', H')$,
	let $\pi$ be an irreducible representation of $H$
	and $\pi'$ of $H'$. Then we have the following isomorphism:
	\[\Hom_H(\Theta_\psi(\pi'),\pi)\cong \Hom_{H'}(\Theta_\psi(\pi),\pi'). \]
\end{lem}

Let $\GSp(W_0\otimes_F E)^\natural$
to be the subgroup of $\GSp(W_0\otimes_F E)$ where the similitude factor takes values in $F^\times$.
Similarly we define $$\GO(V)^\natural=\{h\in \GO(V)|\lambda_V(h)\in F^\times \} .$$

Then we have a see-saw diagram
\[\xymatrix{\GSp(W_0\otimes_F E)^\natural\ar@{-}[rrd]&& \GO(\Res_{E/F}V)\ar@{-}[lld]\\
	\GSp(W_0)\ar@{-}[u]&& \GO(V)^\natural.\ar@{-}[u] } \]

Replace $W_0$ by a $4$-dimensional symplectic space $W_2$ over $F$ with a symplectic similitude  group $\GSp_4(F)$.
Then there is a see-saw pair
\[\Big(\GSp_4( E)^\natural,\GO(V)^\natural \Big)\mbox{   and   }\Big(\GSp_4(F),\GO(\Res_{E/F}V) \Big) \]
in the similitude symplectic group $\GSp(\mathbf{W})$ where $\mathbf{W}=\Res_{E/F}((W_2\otimes_F E)\otimes_E V)$ and
\[\GSp_4(E)^\natural=\{g\in\GSp_4(E)|\lambda_W(g)\in F^\times \}. \]

In order to use Lemma \ref{localseesaw},
we need to figure out the discriminant and Hasse invariant of the quadratic space $\Res_{E/F}V$ over $F.$

Assume that  $E=F(\sqrt{d})$ is a quadratic field extension of $F$,
where $d\in F^\times\setminus {F^\times}^2.$ Let $D_E=(\frac{a,b}{E})$ be the  nonsplit quaternion algebra with involution $\ast$ defined over $E$ with a norm map $N_{D_E},$ which is  a $4$-dimensional quadratic space $V$ over $E$. Then there is an isomorphism for
the vector space $\mbox{Res}_{E/F}V,$ 
$$\mbox{Res}_{E/F}D_E\cong\mbox{Span}_F\{1,\sqrt{d},i,\sqrt{d}i,j,\sqrt{d}j,ij,\sqrt{d}ij \}$$
as $F$-vector spaces,
where $i^2=a,j^2=b,ij=-ji.$ Given a vector $v\in V,$
set $$q_F(v)=\frac{1}{2}tr_{E/F}\circ N_{D_E}(v)\mbox{  and   } (v_i,v_j)=q(v_i+v_j)-q(v_i)-q(v_j).$$
\begin{lem}\label{quadraticrestriction}
	The quadratic space $\Res_{E/F}D_E$ with quadratic form $\frac{1}{2}tr_{E/F}\circ N_{D_E}$ over $F$ has dimension $8,$ discriminant $1$ and Hasse-invariant $-1.$
\end{lem}
\begin{proof}
The nonsplit quaternion algebra over a nonarchimedean local field is unique. We may assume that
$$i^2=a\in F^\times$$ and $j^2=b=b_1+b_2\sqrt{d},N_{E/F}(b)=b_1^2-b_2^2d,~b_i\in F.$

For an element $v=x_1+x_2i+x_3j+x_4ij$ in $D_E$ with $x_i\in E$, we have
$$\frac{1}{2}(v,v)=N_{D_E}(v)=vv^\ast=x_1^2-ax_2^2-bx_3^2+abx_4^2$$ and the corresponding matrix for the
quadratic space $(\mbox{Res}_{E/F}D_E,q_F)$ is
\[\begin{pmatrix}
{2}&0&0&0&0&0&0&0\\
0&{2d}&0&0&0&0&0&0\\
0&0&-{2a}&0&0&0&0&0\\
0&0&0&-{2ad}&0&0&0&0\\
0&0&0&0&-2{b_1}&-2b_2d&0&0\\
0&0&0&0&-2b_2d&-2{b_1d}&0&0\\
0&0&0&0&0&0&2ab_1&2dab_2\\
0&0&0&0&0&0&2dab_2&2{dab_1}
\end{pmatrix}.
\]
The discriminant of $\Res_{E/F} D_E$ is trivial in $F^\times/{F^\times}^2$. 
If $b_1=0,$ then the Hasse-invariant is 
$$(-d,~a)=-1$$
since  $(b_2\sqrt{d},~ a)_E=-1$, where $(-,-)$ (resp. $(-,-)_E$) is the Hilbert symbol defined on $F^\times\times F^\times$ (resp. $E^\times\times E^\times$).
If $b_1\neq0,$ then the Hasse-invariant is
\[(d,d)(-a,-ad)(-b_1,\frac{N_{E/F}(b)d}{-b_1})(N_{E/F}(b)d,-1)(ab_1,\frac{N_{E/F}(b)d}{ab_1})=(a,N_{E/F}(b))=(a,b)_E=-1, \]
because $(a,b)_E=(a,N_{E/F}(b))$ for all $a\in F^\times$ and $b\in E^\times.$ 
\end{proof}

Now let $V$ be the split $2n$-dimensional quadratic space $\mathbb{H}_E^n$ over $E$.
There is a basis $\{e_i,e_j' \}_{1\leq i,j\leq n}$ for the quadratic space
$V$  satisfying $<e_i,e_j'>=\delta_{ij}$ and the other inner products are zero.
Then we fix the basis
\[\{e_i,\sqrt{d}e_i,e_j',e_j'/\sqrt{d} \}_{1\leq i,j\leq n} \]
for $\Res_{E/F}V.$ It is straightforward to check that the vector space $\Res_{E/F}V$ is isomorphic to the split $4n$-dimensional quadratic space $\mathbb{H}^{2n}$ over $F.$

\subsection{The Structure of Degenerate Principal Series}
In this subsection, we follow the notation in \cite{gan2011endoscopy,kudla1996notes}.
Let $H_n=\GO(\mathbb{H}^n)$ be the orthogonal similitude group.
Define the quadratic character $\nu$ to be
\[\nu(h)=\det(h)\cdot\lambda_V^{-n}(h)\mbox{  for  }h\in \GO(\mathbb{H}^n ) \]
so that $\nu|_{\Oo(\mathbb{H}^n )}=\det.$ Define
\[\GSO(\mathbb{H}^n )=\ker\nu=\{h\in \GO(\mathbb{H}^n )|\lambda(h)^n=\det(h) \}. \]
Assume that  $Q_n$ is the standard Siegel parabolic subgroup of $H_n,$ i.e.,
\[Q_n= \Bigg\{\begin{pmatrix}
A^{-1}\\&\lambda A^t
\end{pmatrix}\begin{pmatrix}
I&X\\&I
\end{pmatrix}\big|A\in \GL_n(F), X\in Mat_{n,n}(F) \mbox{ and }X+X^t=0 \Bigg\} \]
with  modular character $|\det A|_F^{1-n}|\lambda|_F^{-n(n-1)/2}$. Then $Q_n\backslash H_n$
is a projective variety and a homogenous space equipped with $H_n$-action.  Each
point on $Q_n\backslash H_n$ corresponds to an isotropic subspace in $\mathbb{H}^n$ of dimension $n.$
Set the degenerate normalized induced representation $I_{Q_n}^{H_n}(s)$ as follows
\[I_{Q_n}^{H_n}(s)=\{f:H_n\rightarrow\mathbb{C}|f(xg)=\delta_{Q_n}(x)^{\frac{1}{2}+\frac{s}{n-1}}f(g)\mbox{ for }x\in Q_n, g\in H_n \}. \]

Let $W_{r}$ be the symplectic space with a symplectic similitude  group 
$\GSp(W_{r}).$ Set $\mathbf{1}_W$ to be the trivial representation of $\GSp(W_{r}).$ Then the big theta lift $\Theta_r(\mathbf{1}_W)$ to $H_n$ of the trivial representation $\mathbf{1}_W$ is isomorphic to a subrepresentation of $I_{Q_n}^{H_n}(s_0)$
where $$s_0=r-\frac{n-1}{2}.$$ The image of $\Theta_r(\mathbf{1}_W)$ in $I_{Q_n}^{H_n}(s_0)$ is denoted by $R_r(\mathbf{1})$, i.e., 
\[\Theta_r(\mathbf{1}_W)=R_{r}(\mathbf{1})\subset I_{Q_n}^{H_n}(s_0). \]

Let us come back to the $\GSp_4$-cases. Assume that  $r=2$ and $n=4$. 

\begin{prop}\label{degenerateseries}
	There is an exact sequence of $H_4$-modules
	\[\xymatrix{0\ar[r]&R_2(\mathbf{1})\ar[r]& I_{Q_4}^{H_4}(1/2)\ar[r]&R_1(\mathbf{1})\otimes\nu\ar[r]&0 }. \]
	\end{prop}
	\begin{proof}
		Note that
		$R_2(\mathbf{1})|_{\Oo_{4,4}(F)}$
		is isomorphic to the big theta lift of the trivial representation $\mathbf{1}_W$
		from $\Sp_4(F)$ to $\Oo_{4,4}(F),$ similar for the big theta lift $R_1(\mathbf{1}).$
		 Mackey theory implies that there is only one orbit for the double coset
		\[Q_4\backslash H_4/\Oo_{4,4}(F)=(Q_4\cap \Oo_{4,4}(F))\backslash \Oo_{4,4}(F)/\Oo_{4,4}(F) \]
		which implies $I_{Q_4}^{H_4}(1/2)|_{\Oo_{4,4}(F)}\cong I_{Q_4\cap \Oo_{4,4}(F)}^{\Oo_{4,4}(F)}(1/2)$.
		Then the sequence is still the same when restricted to the orthogonal group $\Oo_{4,4}(F).$
	 The sequence is exact when restricted to the orthogonal group $\Oo_{4,4}(F)$ 	due to the structure of degenerate principal series (see \cite[Proposition 7.2]{gan2014formal}).
		By the construction of the extended Weil representation, the sequence is exact as $H$-modules.
	\end{proof}
Similarly, let $P_4=M_4N_4$ be the Siegel parabolic subgroup of $\GSp(W_4)=\GSp_8(F)$ where $M_4\cong\GL_1(F)\times\GL_4(F)$. Let $\mathcal{I}(s)$ be the degenerate normalized induced representation of $\GSp_8(F)$ associated to $P_4$, i.e.,
\[\mathcal{I}(s)=\{f:\GSp_{8}(F)\rightarrow\mathbb{C}|f(pg)=\delta_{P_4}(p)^{\frac{1}{2}+\frac{s}{5}}f(g)\mbox{ for }p\in P_4,g\in \GSp_{8}(F) \}. \]
 Then we have
\begin{prop}
	There is an exact sequence of $\GSp_{8}(F)$-modules
	\[\xymatrix{0\ar[r]& R^{3,3}(\mathbf{1})\ar[r]&\mathcal{I}(1/2)\ar[r]&R^{4,0}(\mathbf{1})\ar[r]&0  }, \]
	where $\mathcal{I}(s)$ is the degenerate normalized induced representation of $\GSp_8(F)$ and $R^{3,3}(\mathbf{1})$ (resp. $R^{4,0}(\mathbf{1})$) is the big theta lift to $\GSp_{8}(F)$ of the trivial representation of $\GO_{3,3}(F)$ (resp. $\GO_{4,0}(F)$).  
	\end{prop}
Now we use Mackey theory to study $I_{Q_4}^{H_4}(1/2)|_{\GO_{2,2}(E)^\natural}$ which involves the computation for the double coset $Q_4\backslash H_4/\GO_{2,2}(E)^\natural$.
Denote $X_4=Q_4\backslash H_4$ as the projective variety. 
\subsubsection{Double cosets}
Now let us consider the double coset 
\[Q_4\backslash H_4/\GO_{2,2}(E)^\natural. \] 
Assume that  $V=\mathbb{H}_E^2$
 with basis $\{e_i,e_j' \}_{1\leq i,j\leq 2}$
and
$<e_i,e_j'>=\delta_{ij}.$ 
Fix the basis
$$\{e_1,\sqrt{d}e_1,e_2,\sqrt{d}e_2,e_1',e_1'/\sqrt{d},e_2',e_2'/\sqrt{d} \}$$ for $V_F=\Res_{E/F}V$. 
 The inner product $\langle\langle-,-\rangle\rangle$ on $V_F$ is given by
\[\langle\langle x,y\rangle\rangle:=\frac{1}{2}tr_{E/F}(<x,y>)\]
for $x,y\in V$.
Let us fix an embedding  $i:\GO_{2,2}(E)^\natural\rightarrow \GSO_{4,4}(F)$. 

The double coset decomposition for the case at hand can be obtained from
 more general case.
Assume that  $\mathbf{V}$ is a symplectic space or a split quadratic space over $E$ of dimension $2n,$
with a non-degenerate bilinear form 
$B:\mathbf{V}\times \mathbf{V}\rightarrow E .$
Let $U(\mathbf{V})$ be the isometry group, i.e.,
\[U(\mathbf{V})=\{g\in \GL(\mathbf{V})| B(gx,gy)=B(x,y)\mbox{ for all } x,y\in \mathbf{V} \} \]
which is a symplectic group or an orthogonal group.
Then $\Res_{E/F}\mathbf{V}$ is a vector space over $F$ of dimension $4n$
with a non-degenerate bilinear form $\frac{1}{2}tr_{E/F}\circ B.$
\begin{lem}\label{orbitdecomp}
	Let $P$ be a Siegel parabolic subgroup of $U(\Res_{E/F}\mathbf{V})$. Then each point in the homogeneous space
	$X=P\backslash U(\Res_{E/F}\mathbf{V})$ corresponds to a $2n$-dimensional maximal isotropic subspace in $\Res_{E/F}\mathbf{V}$ and  the finite double cosets $X/U(\mathbf{V})$ can be parametrized by a pair
	$$(\dim_E E\cdot L,B_L )$$ where $L\subset \Res_{E/F}\mathbf{V}$ is a maximal isotropic subspace with respect to the inner product $\langle\langle-,-\rangle\rangle$ over $F$, 
	$$E\cdot L:=\{e\cdot x|e\in E,x\in L \}$$ is a linear $E$-subspace in $\mathbf{V}$ and $$B_L:L/L_0\times L/L_{0}\rightarrow \sqrt{d}\cdot F$$
	  is a non-degenerate bilinear form inherited from $V$, where $$L_0=\{x\in L: B(x,y)=0 \mbox{ for all } y\in L \}.$$
	Moreover, if $ L=L_0,$ then $L$ lies in the closed orbit.
	 If $L_0=0,$ then $L$ lies in the open orbit.
	\end{lem}
\begin{proof}
	Under a suitable basis for $L,$ the bilinear form for $B|_{ L}$ corresponds to a matrix $\sqrt{d}\cdot T$, where $T\in M_{2n}(F).$ Moreover,
	we can choose $T$ such that it is a diagonal (resp. an anti-diagonal) matrix if $B(x,y)=B(y,x)$ (resp. $B(y,x)=-B(x,y)$).
	Then $$\dim_E E\cdot L=n+\frac{1}{2}\cdot rank(T)$$ which is invariant under $U(\mathbf{V})$-action. The bilinear form $B_L$ corresponds to a matrix $\sqrt{d}\cdot T',$ i.e.,
$$T=	\begin{pmatrix}
	0&0&0\\0&	T'&0\\0&0&0
	\end{pmatrix}$$
	 where $T'$ is invertible and $rank (T)=rank (T')$.
	
	Assume that there are two isotropic subspaces $L_1$ and $L_2$ satisfying
	\[\dim_E E\cdot L_1=\dim_E E\cdot L_2=l\mbox{  and  }B_{L_1}\cong B_{L_2}. \]
	This means that there exists $g\in \GL_l(E)$ such that $g:E\cdot L_1\rightarrow E\cdot L_2$ satisfying $$B_{L_1}(x,y)=B_{L_2}(gx,gy).$$
	It is easy to lift $g$ to $g_E\in U(\mathbf{V})$ such that $g_EL_1=L_2.$
	
	In fact, $g=\begin{pmatrix}
	g_1&0\\0&g_2
	\end{pmatrix}$ lies in a subgroup  of $\GL_l(E),$ which can be regarded as a Levi subgroup of  $U(\mathbf{V}),$ and
	\[B_L(gx,gy)=B_L(g_2x',g_2y')\]
	 when $x-x',y-y'\in L_0.$  Then $g_E=\begin{pmatrix}
	g_1\\&g_2\\&& g_1^\ast
	\end{pmatrix}\in U(\mathbf{V}),$ where $g_1^\ast$ depends on $g_1$ and $\mathbf{V}$.
\end{proof}
\begin{rem}
	In fact, there is only one closed orbit in the double coset $P\backslash U(\Res_{E/F}\mathbf{V} )/U(\mathbf{V}).$
	\end{rem}
Consider  the double coset
 $$Q_4\backslash H_4/\GO_{2,2}(E)^\natural.$$
There are several $\GO_{2,2}(E)^\natural$-orbits in $Q_4\backslash H_4/\GO_{2,2}(E)^\natural$.
By Lemma \ref{orbitdecomp}, there are two invariants for the orbit $\GO_{2,2}(E)^\natural\cdot L:$
\begin{itemize}
	\item the dimension $\dim_E(E\cdot L)$ and
	\item  the quadratic form $q_E|_{L}$ up to scaling in $F^\times.$
\end{itemize}  

By the classification of $4$-dimensional quadratic spaces over $F,$
there are $4$ elements  lying in the kernel
\[\ker \{H^1(F,\Oo_4)\rightarrow H^1(E,\Oo_4)\}, \]
which are 
\begin{itemize}
	\item the split quaternion algebra $Mat_{2,2}(F)$  with $q(v)=\det(v)$ for $v\in Mat_{2,2}(F);$
	\item  the quaternion division algebra $D(F)$ with the norm map $N_{D/F};$
	\item the non-split $4$-dimensional quadratic space $V_3=E\oplus\mathbb{H}$ with $q(e,x,y)=N_{E/F}(e)-xy$ and
	\item $V_4=\epsilon V_3$ with $\epsilon\in F^\times\setminus N_{E/F}(E^\times).$
\end{itemize}
 However, we consider the double coset 
\[Q_4\backslash H_4/\GO_{2,2}(E)^\natural\]
for the similitude groups and we observe that
$V_3$ and $V_4$ are in the same $\GO_{2,2}(E)^\natural$-orbit in $Q_4\backslash H_4/\GO_{2,2}(E)^\natural$.
\begin{prop}
	Pick a point $L\in X_4/\GO_{2,2}(E)^\natural $ lying in an open orbit.
 Then the stabilizer of $L$ in $\GO_{2,2}(E)^\natural$ is isomorphic to the similitude group $\GO(L).$
 \end{prop}
 \begin{proof}
 	For $g\in \GO_{2,2}(E)^\natural$ with $g(L)=L,$ we have
 	\[<gl_1,gl_2>=\lambda(g)\cdot<l_1,l_2>\]
 	 and so $\langle\langle gl_1,gl_2\rangle\rangle=\lambda(g)\cdot \langle \langle l_1,l_2\rangle\rangle$. 
 	This means $g\in \GO(L).$ Conversely, if $h\in \GO(L,\frac{1}{\sqrt{d}}q_E|_L),$ set $$h_E:x\otimes e\mapsto h(x)\otimes e$$
 	for $x\otimes e\in L\otimes E\cong L\cdot E=V$. Then $h_E(L)=L$ and
 	\[<h_E(x_1\otimes e_1),h_E(x_2\otimes e_2)>=e_1e_2\lambda(h)\langle \langle x_1,x_2\rangle\rangle=\lambda(h)<x_1\otimes e_1,x_2\otimes e_2>, \]
 	i.e. $h_E\in \GO_{2,2}(E)^\natural.$ Then we get a bijection between the similitude orthogonal group $\GO(L)$ and the stabilizer  of $L$ in $\GO_{2,2}(E)^\natural.$
 	Observe that the map $h\mapsto h_E$ is a group homomorphism. Then $\GO(L)$
 	is isomorphic to the stabilizer of $L$ via the map $h\mapsto h_E.$
 \end{proof}
 There are three $F$-rational open orbits, whose  stabilizers are $\GO_{2,2}(F),\GO_{4,0}(F)$ and $\GO_{3,1}(F)$ respectively.
There is one closed orbit  
which has   stabilizer  $$\GO_{2,2}(E)^\natural\cap Q_4=:Q^\natural\cong\Bigg\{\begin{pmatrix}
A^{-1}&\ast\\0&\lambda A^t
\end{pmatrix}|A\in \GL_2(E),\lambda\in F^\times \Bigg\}.$$
There are two intermediate  orbits with representatives $L_1,L_2$ and $\dim_E(E\cdot L_i)=3$.
The stabilizers are isomorphic to $$(\GL_1(E)\times \GO_{1,1}(F))\cdot Mat_{2,2}(F)\mbox{  and } (\GL_1(E)\times \GO(\mathcal{V}_E))\cdot Mat_{2,2}(F),$$ where 
 $\mathcal{V}_E$ is the $2$-dimensional quadratic space over $F$ with nontrivial discriminant which corresponds to $E.$
\begin{rem}
For $(g,t)\in \GL_2(E)\times F^\times, $ we set $$\beta((g,t))=(g,\bar{g}\cdot t)\in \GL_2(E)\times \GL_2(E).$$
Then $\beta:\GSO_{3,1}(F)\rightarrow \GSO_{2,2}(E)^\natural$ is an embedding due to the following exact sequences
\[\xymatrix{1\ar[r]& E^\times\ar[r]^-{i_1}\ar@{=}[d]&\GL_2(E)\times F^\times\ar[r]\ar[d]^\beta&\GSO_{3,1}(F)\ar[d]\ar[r]&1\\
1\ar[r]&E^\times\ar[r]^-{i_2}&\GL_2(E)\times\GL_2(E)\ar[r]&\GSO_{2,2}(E)\ar[r]&1 } \]
where $i_1(e)=(e,N_{E/F}(e)^{-1})$ and $i_2(e)=(e,e^{-1})$ for $e\in E^\times$.
\end{rem}
There are several orbits for $X_4/\GO_{2,2}(E)^\natural$.
So there is a decreasing filtration of $\GO_{2,2}(E)^\natural$-modules for  $I_{Q_4}^{H_4}(s)|_{\GO_{2,2}(E)^\natural}.$
\subsubsection{Filtration} Consider the filtration 
\[I_{Q_4}^{H_4}(s)= I_2(s)\supset I_1(s)\supset I_0(s)\supset 0\]
of the degenerate principal series $I_{Q_4}^{H_4}(s)|_ {\GO_{2,2}(E)^\natural}$
 with a sequence of sub-quotients
 $$I_0(s)=ind_{\GO_{2,2}(F) }^{\GO_{2,2}(E)^\natural}\mathbb{C}\oplus ind_{\GO_{4,0}(F)}^{\GO_{2,2}(E)^\natural}\mathbb{C}\oplus ind_{\GO_{3,1}(F)}^{\GO_{2,2}(E)^\natural}\mathbb{C},$$ 
$$I_2(s)/I_1(s)\cong ind_{Q^\natural}^{\GO_{2,2}(E)^\natural}\delta_{Q^\natural}^{s+1}$$  where  $Q^\natural$  is the Siegel parabolic subgroup of $\GO_{2,2}(E)^\natural$ with modular character $\delta_{Q^\natural}$
and
\[I_1(s)/I_0(s)\cong ind_{(\GL_1(E)\times \GO_{1,1}(F))\cdot N }^{\GO_{2,2}(E)^\natural}\delta_Q^{\frac{1}{2}+\frac{s}{3}}\delta_1^{-\frac{1}{2}}\oplus ind_{Q'}^{\GO_{2,2}(E)^\natural }\delta_Q^{\frac{1}{2}+\frac{s}{3}}\delta_2^{-\frac{1}{2}}  \]
where $Q'=(\GL_1(E)\times \GO(\mathcal{V}_E))\cdot N,~N\cong  Mat_{2,2}(F)$ and \[\delta_i\Big(t,h\Big)=|N_{E/F}(t^2) \cdot \lambda_V (h)^{-2}|_F\]
   for  $ t\in \GL_1(E)$ and $h\in\GO_{1,1}(F)$ or $\GO(\mathcal{V}_E)$, where $\mathcal{V}_E$ is the non-split $2$-dimensional quadratic space.
\begin{rem}
We would like to highlight the fact  that on the open orbits related to $I_0(s),$ the group embedding $\GO_{2,2}(F)\hookrightarrow\GO_{2,2}(E)^\natural$ (similar for the other two group embeddings) is not induced from the geometric embedding $i:\GO(L)\hookrightarrow\GO(L\otimes_FE),$ but the composite map $Ad_{h^\delta}\circ i$ of the adjoint map $Ad_{h^\delta}$ and the geometric embedding $i$ where $$h^\delta=\begin{pmatrix}
\sqrt{d}\\&1
\end{pmatrix}\in\GO(2,2)(E).$$ However, it does not affect the results when we consider the distinction problems  for the similitude groups. We will show that the results on the open orbits determine the distinction problems $\dim\Hom_{\GO_{2,2}(E)}(I_{Q_4}^{H_4}(1/2),\Sigma) $ when $\Sigma$ is a generic representation. (See \cite[Theorem 2.8]{blanc-delorme}.)
\end{rem}

Recall that $\GSp_4(E)^\natural=\{g\in\GSp_4(E)|\lambda_W(g)\in F^\times \}.$
When we deal with the case
\[\Ind_{P_4}^{\GSp_8(F)}\delta_{P_4}^{\frac{s}{5} }|_{\GSp_4(E)^\natural }, \]
where $P_4$ is the Siegel parabolic of $\GSp_{8}(F)$ with modular character $\delta_{P_4}$, the above results still hold. More precisely, set
\[\mathcal{I}(s)=\{f:\GSp_8(F)\rightarrow\mathbb{C}|f(xg)=\delta_{P_4}(x)^{\frac{1}{2}+\frac{s}{5}}f(g)\mbox{ for }x\in P_4, g\in \GSp_8(F) \}. \]
There is a filtration 
\[\mathcal{I} _0(s)\subset\mathcal{I}_1(s)\subset \mathcal{I}_2(s)=\mathcal{I}(s)|_{\GSp_4(E)^\natural }   \]
of $\mathcal{I}(s)|_{\GSp_4(E)^\natural}$ such that
\begin{itemize}
	\item $\mathcal{I}_0(s)\cong ind_{\GSp_4(F) }^{\GSp_4(E)^\natural }\mathbb{C},$
	\item $\mathcal{I}_1(s)/\mathcal{I}_0(s)\cong ind_{M'N' }^{\GSp_4(E)^\natural }\delta_{P_4}^{\frac{1}{2}+\frac{s}{5}}\delta_{M'N'}^{-\frac{1}{2}}$ and
	\item $\mathcal{I}_2(s)/\mathcal{I}_1(s)\cong ind_{P^\natural }^{\GSp_4(E)^\natural }\delta_{P^\natural}^{\frac{s+1}{3}}, $
\end{itemize}  
where $P^\natural$ is the Siegel parabolic subgroup of $\GSp_4(E)^\natural$, $M'\cong\GL_1(E)\times \GL_2(F),$  $N'\cong Mat_{1,1}(E)\oplus Mat_{2,2}(F)$ and $$\delta_{M'N'}(t,g)=|N_{E/F}(t)^4\cdot \lambda_W(g)^{-4}|_F $$
for $(t,g)\in\GL_1(E)\times\GL_2(F)$. Here the group embedding $\GSp_4(F)\hookrightarrow\GSp_4(E)^\natural$ in $\mathcal{I}_0(s)$ is the composition map $Ad_{g^\delta}\circ i'$ where $i':\GSp(W_2)\hookrightarrow\GSp(W_2\otimes_FE)$ is the geometric embedding and 
\[g^\delta=\begin{pmatrix}
\sqrt{d}\\&1
\end{pmatrix}\in\GSp_4(E). \]

\begin{defn}
	An irreducible representation $\tau$ of $\GSp_{4}(E)^\natural$
	occurs on the boundary of $\mathcal{I}(s)$ if $$\Hom_{\GO(2,2)(E)^\natural}(\mathcal{I}_{i+1}(s)/\mathcal{I}_i(s),\tau)\neq0
	\mbox{ for } i=0\mbox{  or  }1.$$
	\end{defn}
In \cite{hengfei2017}, we have proved that the tempered representation $\tau$ does not occur on the boundary of $\mathcal{I}(1/2)$. Then we can verify the Prasad conjecture for $\GSp_4$ when $\tau$ is a tempered representation. After discussing with Dmitry Gourevitch, we realized that
\cite[Proposition 4.9]{gurevich2015non}  can imply the Prasad conjecture for $\GSp_4$ when the $L$-packet $\Pi_{\phi_\tau }$ is  generic. Thus we will give a slightly different proof of Theorem \ref{maingsp(4)theorem} from the one in \cite{hengfei2017}.

			\subsection{The distinction problem for  $\GSp_4$}\label{subsect:proofof1.1}
			Let us recall what we have obtained.
			Let $\tau$ be an irreducible  representation of
			$\GSp_4(E).$ Since $\tau|_{\Sp_4(E)}$ is multiplicity-free due to \cite[Theorem 1.4]{adler2006on},  $\tau|_{\GSp_4(E)^\natural}$ is multiplicity-free. 
			Assume that  $\tau=\theta(\pi_1\boxtimes\pi_2)$ participates in the theta correspondence with $\GSO_{2,2}(E)$.
			Then  the see-saw identity implies that
			\[\Hom_{\GSp_4(F)}(\tau,\mathbb{C} )\subset \Hom_{\GSp_4(F)}(\Theta_2(\Sigma),\mathbb{C})\cong \Hom_{\GO_{2,2}(E)^\natural}(R_2(\mathbf{1}),\Sigma ) \]
			where $R_2(\mathbf{1} )$ is the image of the big theta lift to $H_4$ of the trivial representation of $\GSp_4(F) $
			in $I_{Q_4}^{H_4}(1/2)$ and $\Sigma$ is the irreducible representation of $\GO_{2,2}(E)$ such that $\tau=\theta(\Sigma)$. In fact, if $\pi_1\ncong\pi_2$, then $\Sigma=\Ind_{\GSO_{2,2}(E)}^{\GO_{2,2}(E)}(\pi_1\boxtimes\pi_2)$. If $\pi_1\cong\pi_2$, then there are two extensions to $\GO_{2,2}(E)$ of $\pi_1\boxtimes\pi_2$. The representation 
			$\Sigma$ is the unique extension of $\pi_1\boxtimes\pi_2$ which participates into the theta correspondence with $\GSp_4(E)$, denoted by $(\pi_1\boxtimes\pi_2)^+$.
			
		
				\begin{lem}
					Assume that  $\pi_1\boxtimes\pi_2$ is an irreducible representation of $\GSO_{2,2}(E),$ associated with an irreducible representation $\Sigma^+$ of $\GO_{2,2}(E)$ that has a nonzero theta lift to $\GSp_4(E)$. Then 
					\[\dim \Hom_{\GO(L)}(\Sigma^+,\mathbb{C})=\dim \Hom_{\GSO(L)}(\pi_1\boxtimes\pi_2,\mathbb{C} )\]
					where $\GO(L) \hookrightarrow \GO(L\otimes_F E)=\GO_{2,2}(E)$ and  the $4$-dimensional quadratic space $L$ is one of the following options: $Mat_{2,2}(F)$,  $D(F)$ and $E\oplus\mathbb{H}_F$
				\end{lem}
					\begin{proof}
						If $\pi_1\neq\pi_2,$ then it follows from Frobenius Reciprocity. If $\pi_1=\pi_2,$ then
						we consider the following see-saw diagram
						\[\xymatrix{ \GO_{2,2}(E)^\natural\ar@{-}[rd]\ar@{-}[d] & \GSp_4(F)\ar@{-}[ld]\\ \GO(L)&\GSp_2(E)^\natural\ar@{-}[u]  } \]
						where $\GSp_2(E)^\natural=\{g\in\GSp_2(E)|\lambda_W(g)\in F^\times \}$.
					Then we have $$\Hom_{\GO(L)}(\Sigma^-,\mathbb{C})=\Hom_{\GO(L)}(\Sigma^+,\nu)=\Hom_{\GSp_2(E)^\natural}(\Theta_2(\nu),\pi_1 )=0,$$
						because the big theta lift $\Theta_2(\nu)$ to $\GSp_4(F)$ is zero by the conservation relation.
						Hence
						\[\Hom_{\GSO(L) }(\pi_1\boxtimes\pi_2 ,\mathbb{C})=\Hom_{\GO(L)}(\Sigma^+\oplus\Sigma^-,\mathbb{C} )=\Hom_{\GO(L)}(\Sigma^+,\mathbb{C}). \]
						This finishes the proof.
				\end{proof}
		
\begin{lem}
	Given an irreducible admissible   representation $\tau$ of $\GSp_4(E)$,  
	we have 
	\[\dim \Hom_{\GSp_4(F) }(\tau^g,\mathbb{C} )= \dim \Hom_{\GSp_4(F)}(\tau,\mathbb{C})=\dim \Hom_{\GSp_4(F)}(\tau^\vee,\mathbb{C}) \]
	where $\tau^g(x)=\tau(gxg^{-1})$ for $g\in \GSp_4(E).$
\end{lem}
\begin{proof}Note that $\tau^g\cong\tau$
	and $\tau^\vee\cong\tau^{g_{-1}},$ where $g_{-1}\in\GSp_4(F)$ with similitude $\lambda_W(g_{-1})=-1$. 
	Then 
	\[\dim \Hom_{\GSp_4(F)}(\tau,\mathbb{C})=\dim \Hom_{\GSp_4(F)}(\tau^\vee,\mathbb{C}). \]
\end{proof}
\begin{rem}
		We have a similar statement for the  group $\GSO(V)$.
\end{rem}
	There is another key input for the $\GL_4$-distinction problems in our proof of Theorem \ref{maingsp(4)theorem}.
	\begin{thm}\cite[Theorem 5.2]{matringe2009distinction}
		Given a generic representation $\pi$ of $\GL_n(E)$ with a Langlands parameter $\phi_\pi=\triangle_1\oplus\triangle_2\oplus\cdots\oplus\triangle_t$ with $\triangle_i:WD_E\rightarrow \GL_{n_i}(\mathbb{C})$  irreducible and $\sum_{i=1}^t n_i=n,$ then $\pi$  is $\GL_n(F)$-distinguished if and only if  there is a reordering of $\triangle_i'$s and an integer $r$ between $1$ and $\frac{t}{2}$ such that $\triangle_{i+1}^\sigma=\triangle_i^\vee$ for $i=1,3,\cdots,2r-1 $ and $\triangle_i$ is conjugate-orthogonal for $i>2r.$ 
		\end{thm}

\begin{lem}\label{GL:period}
	 Let $\pi$ be a square-integrable representation of $\GL_2(E)$. Then $\pi$ is $\GL_2(F)$-distinguished if and ony if $\pi$ is $D^\times(F)$-distinguished. If $\pi=\pi(\chi^{-1},\chi^\sigma)$, then $\pi$ is both $\GL_2(F)$-distinguished and $D^\times(F)$-distinguished.
	Let $\pi_0=\pi(\chi_1,\chi_2 )$ with $\chi_1\neq\chi_2,\chi_1|_{F^\times}=\chi_2|_{F^\times}=\mathbf{1}$ be an irreducible smooth representation of $\GL_2(E).$ Then
	$\pi_0 $ is $\GL_2(F)$-distinguished but not $D^\times(F)$-distinguished.  
\end{lem}
\begin{proof}
	If $\pi$ is square-integrable, then it follows from \cite[Theorem C]{prasad1992gl(2)}. Let $\pi_0=\pi(\chi_1,\chi_2)$.
	By  Mackey theory, we know that
	\[\dim\Hom_{D^\times(F) }(\pi_0,\mathbb{C} )=\dim\Hom_{E^\times}(\chi_1\chi_2^\sigma,\mathbb{C} )=\begin{cases}
	1&\mbox{ if }\chi_1\chi_2^\sigma=\mathbf{1};\\
	0&\mbox{ otherwise.}
	\end{cases}  \]
	 If $\chi_1\neq\chi_2$ and $\chi_1|_{F^\times}=\chi_2|_{F^\times}=\mathbf{1}$, then $\chi_1\chi_2^\sigma\neq\mathbf{1}$. Thus $\pi_0$ is not $D^\times(F)$-distinguished. Since the Langlands parameter $\phi_{\pi}=\chi^{-1}\oplus\chi^\sigma$ (resp. $\phi_{\pi_0}$)  is conjugate-orthogonal in the sense of \cite[\S3]{gan2011symplectic}, $\pi$ (resp. $\pi_0$) is $\GL_2(F)$-distinguished due to \cite[Theorem 6.2]{gan22arithmeticity} or \cite[Theorem 5.2]{matringe2009distinction}.
\end{proof}
\begin{lem} \label{arnab}
	Let $\pi$ be an essentially discrete series representation of $\GL_2(E)$.
	Let $\Pi=J_{P}(\pi|-|_E,\pi)$ be the nongeneric representation of $\GL_4(E)$. Then the following statements are equivalent:
	\begin{enumerate}[(i)]
		\item 	$\Pi$ is both $\GL_4(F)$-distinguished and $(\GL_4(F),\omega_{E/F})$-distinguished;
		\item $\Pi^\vee\cong\Pi^\sigma$.
	\end{enumerate}
\end{lem}
\begin{proof}
	See \cite[Theorem 6.5]{arnab2018}.
\end{proof}
\subsubsection{The Langlands correspondence for $\GSp_4$}
In this part, we will recall the Langlands correspondence for $\GSp_4$ which has been set up by Gan-Takeda in \cite{gan2011locallanglands}.

Let $\Pi(\GSp_4)$ be the set of (equivalence classes of) irreducible smooth representation of $\GSp_4(F)$. Let $\Hom(WD_F,\GSp_4(\mathbb{C}))$ be the set of (equivalence classes of) admissible homomorphisms
\[WD_F\longrightarrow \GSp_4(\mathbb{C}). \]
\begin{thm}[Gan-Takeda]
	There is a surjective finite to one map
	\[L:\Pi(\GSp_4)\longrightarrow \Hom(WD_F,\GSp_4(\mathbb{C})) \]
	with the following properties:
	\begin{enumerate}[(i)]
		\item $\tau$ is a (essentially) discrete series representation of $\GSp_4(F)$ if and only if its $L$-parameter $\phi_\tau=L(\tau) $ does not factor through any proper Levi subgroup of $\GSp_4(\mathbb{C})$.
		\item For an $L$-parameter $\phi\in \Hom(WD_F,\GSp_4(\mathbb{C}))$, its fiber $\Pi_\phi $ can be naturally parametrized by the set of irreducible characters of the component group
		\[\pi_0(Z(Im(\phi))/Z_{\GSp_4(\mathbb{C})} ). \]
		This component group is either trivial or equal to $\mathbb{Z}/2\mathbb{Z}$. When it is $\mathbb{Z}/2\mathbb{Z}$, exactly one of the two representations in $\Pi_\phi$ is generic and it is the one indexed by the trivial character of $\pi_0(Z(Im(\phi))/Z_{\GSp_4(\mathbb{C})} )$. 
		\item The similitude character sim$(\phi_\tau)$ of $\phi_\tau $ is equal to the central character $\omega_\tau$ of $\tau$. Here sim:$\GSp_4(\mathbb{C})\longrightarrow \mathbb{C}^\times$ is the similitude character of $\GSp_4(\mathbb{C})$.
		\item The $L$-parameter of $\tau\otimes(\chi\circ\lambda_W)$ is equal to $\phi_{\tau}\otimes\chi$. Here $\lambda_W:\GSp_4(F)\longrightarrow F^\times$ is the similitude character of $\GSp_4(F)$, and we have regarded $\chi$ as both a character of $F^\times$ and a character $W_F$ by local class field theory.
	\end{enumerate}
\end{thm}
We repeat the statements of Theorem 1.1 as below.
	\begin{thm}\label{localgspperiod} Assume that $\tau$ 
		is an irreducible representation of $\GSp_4(E)$ with
		a central character $\omega_\tau$ satisfying $\omega_\tau|_{F^\times}=\mathbf{1}.$
		\begin{enumerate}[(i)]
			\item If $\tau=\theta(\Sigma)$ is an irreducible representation of $\GSp_4(E),$ where $\Sigma$ is an irreducible representation of $\GO_{4,0}(E)$, then  $\tau$ is not $\GSp_4(F)$-distinguished.
			\item Suppose that  $\Sigma=(\pi_1\boxtimes\pi_1)^+$ is an irreducible representation of $\GO_{2,2}(E)$ and  $\Sigma=\Ind_{\GSO_{2,2}(E)}^{\GO_{2,2}(E)}(\pi_1\boxtimes\pi_2)$ if $\pi_1\neq\pi_2.$ If $\tau=\theta(\Sigma)$ is generic, 
			%
			then 
			\[\dim \Hom_{\GSp_4(F)}(\tau,\mathbb{C})=\begin{cases}
		2,&\mbox{ if }\pi_i\ncong\pi_0 \mbox{  are both }\GL_2(F)\mbox{-distinguished};\\
		1,&\mbox{ if }\pi_1\ncong\pi_2\mbox{ but }\pi_1^\sigma\cong\pi_2^\vee;\\
		1,&\mbox{ if }\pi_1\cong\pi_2 \mbox{ is }\GL_2(F)\mbox{-distinguished but not }(\GL_2(F),\omega_{E/F})\mbox{-distinguished};\\
		1,&\mbox{ if }\pi_2 \mbox{ is }\GL_2(F)\mbox{-distinguished and }\pi_1\cong\pi_0;
		\\
		0,&\mbox{the other cases.}
		\end{cases}\]
			Here $\pi_0=\pi(\chi_1,\chi_2)$ with $\chi_1\neq\chi_2,\chi_1|_{F^\times}=\chi_2|_{F^\times}=1.$ 
				\item Assume that $\tau$ is not in case (i) or (ii), so that $\tau=\theta(\Pi\boxtimes\chi),$ where $\Pi\boxtimes\chi$ is a representation of $\GSO_{3,3}(E)$. If $\tau$ is generic,		then
			\[\dim \Hom_{\GSp_4(F) }(\tau,\mathbb{C})=\begin{cases}
			1,&\mbox{ if }\Pi\mbox{ is }\GL_4(F)\mbox{-distinguished;}\\
			0,&\mbox{ otherwise.}
			\end{cases} \]
		\end{enumerate}
	\end{thm}
	\begin{proof}
		\begin{enumerate}[(i)]
			\item If $\Sigma$ is a representation of $\GSO_{4,0}(E)$,
			then $\tau=\theta(\Sigma)=\Theta(\Sigma)$ and $$\Hom_{\GSp_4(F)}(\Theta(\Sigma),\mathbb{C})\cong \Hom_{\GO_{4,0}(E)^\natural }(\Theta_{W,D',\psi}(\mathbf{1}),\Sigma^+),$$
			where $D'=\mbox{Res}_{E/F}D_E=D(F)\oplus\mathbb{H}^2 $ is the 8-dimensional quadratic vector space over $F$ with determinant $1$ and Hasse invariant $-1$ due to Lemma \ref{quadraticrestriction} and $\Theta_{W,D',\psi}(\mathbf{1})$ is the big theta lift to $\GO(V')$ of the trivial representation $\mathbf{1}$.
			Note that the first occurrence of the trivial representation is $\dim_F W=4$ in the Witt tower $D\oplus\mathbb{H}^{r}$, which is bigger than $2$. Thus  $\Theta_{W,D',\psi}(\mathbf{1})=0.$  Hence,  $$\Hom_{\GSp_4(F)}(\Theta(\Sigma),\mathbb{C})=0$$ 
			and so $\tau=\theta(\Sigma)$ is not $\GSp_4(F)$-distinguished.
			\item By Proposition \ref{degenerateseries}, there is an exact sequence of $H_4$-representations 
			\begin{equation}\label{I(1/2)}
			\xymatrix{0\ar[r]&R_2(\mathbf{1})\ar[r]&I_{Q_4}^{H_4}(1/2)\ar[r]&\nu\otimes R_1(\mathbf{1})\ar[r]&0 }
			\end{equation}
			where $R_i(\mathbf{1} )$ is the big theta lift to $H_4$ of the trivial representation $\mathbf{1}$ of $\GSp_{2i}(F) $. We take the right exact contravariant functor $\Hom_{\GO_{2,2}(E)^\natural}(-,\Sigma)$
			 with respect to \eqref{I(1/2)} 
		and get a short exact sequence
			\begin{equation}\label{injection}
			0\rightarrow \Hom_{\GO_{2,2}(E)^\natural}(R_1(\mathbf{1}))\otimes\nu,\Sigma)\rightarrow \Hom_{\GO_{2,2}(E)^\natural}(I_{Q_4}^{H_4}(1/2),\Sigma)\rightarrow \Hom_{\GO_{2,2}(E)^\natural}(R_2(\mathbf{1}),\Sigma).
			\end{equation} 
			Consider the following double see-saw diagrams
				$$	\xymatrix{\GSp_4(E)^\natural\ar@{-}[d]\ar@{-}[rd]&H_4\ar@{-}[d]\ar@{-}[ld]\ar@{-}[rd]&\GSp_2(E)^\natural\ar@{-}[ld]\ar@{-}[d] \\\GSp_4(F)&\GO_{2,2}(E)^\natural&\GL_2(F). }$$
		Note that $\Hom_{\GO_{2,2}(E)^\natural}(R_2(\mathbf{1}),\Sigma)\cong\Hom_{\GSp_4(F)}(\tau,\mathbb{C})$.			 Since $\GO_{2,2}(E)^\natural$ is a subgroup of $\GSO_{4,4}(F),$ one has
\begin{equation*}\Hom_{\GO_{2,2}(E)^\natural}(R_1(\mathbf{1})\otimes\nu,\Sigma)= \Hom_{\GO_{2,2}(E)^\natural}(R_1(\mathbf{1}),\Sigma)\cong \Hom_{\GSp_2(F)}(\Theta_1(\Sigma),\mathbb{C}) .\end{equation*}
Here $\Theta_1(\Sigma)$ is the big theta lift  to $\GSp_2(E)$ of $\Sigma$ and it is zero unless $\pi_1=\pi_2.$
Then 
\begin{equation}\label{exact:inequ}
\dim\Hom_{\GSp_4(F)}(\tau,\mathbb{C})+\dim\Hom_{\GSp_2(F)}(\Theta_2(\Sigma),\mathbb{C})\geq\dim\Hom_{\GO_{2,2}(E)^\natural}(I_{Q_4}^{H_4}(1/2),\Sigma)  . 
\end{equation}
Observe that $\GO_{2,2}(E)^\natural$ is the fixed point of a involution on $H_4$, which is given by the scalar matrix $$h=\sqrt{d}\in \GO_{2,2}(E)^\natural\subset H_4$$ acting on $H_4$ by conjugation.
Due to \cite[Theorem 2.5]{olafsson1987fourier}, there exists a polynomial $f$ on $H_4$  such that 
the complements of the open orbits in the double coset $Q_4\backslash H_4/\GO_{2,2}(E)^\natural$ is the zero set of $f$. Thanks to \cite[Proposition 4.9]{dima2018analytic}, the multiplicity $\dim\Hom_{\GO_{2,2}(E)^\natural}(I_{Q_4}^{H_4}(1/2),\Sigma) $ is at least $\dim\Hom_{\GO_{2,2}(E)^\natural}(I_0(1/2),\Sigma)$ where the submodule $I_0$ corresponds to the open orbits. More precisely, $I_0(1/2)\cong ind_{\GO_{4,0}(F)}^{\GO_{2,2}(E)^\natural}\mathbb{C}\oplus ind_{\GO_{2,2}(F)}^{\GO_{2,2}(E)^\natural}\mathbb{C}\oplus ind_{\GO_{3,1}(F)}^{\GO_{2,2}(E)^\natural}\mathbb{C}$ and 
\begin{equation}\label{openorbitindentity}
\dim\Hom_{\GO_{2,2}(E)^\natural}(I_{Q_4}^{H_4}(1/2),\Sigma)\geq \dim\Hom_{\GO_{2,2}(E)^\natural}(ind_{\GO_{4,0}(F)}^{\GO_{2,2}(E)^\natural}\mathbb{C}\oplus ind_{\GO_{2,2}(F)}^{\GO_{2,2}(E)^\natural}\mathbb{C}\oplus ind_{\GO_{3,1}(F)}^{\GO_{2,2}(E)^\natural}\mathbb{C},\Sigma ) .\end{equation}
Together with \eqref{exact:inequ}, the sum
\begin{equation}
\begin{split}\label{key:inequality}
	&\dim\Hom_{\GSp_{4}(F)}(\Sigma,\mathbb{C})+\dim\Hom_{\GSp_2(F)}(\Theta_2(\Sigma),\mathbb{C})\\
\geq&\dim\Hom_{\GO_{2,2}(E)^\natural}(I_{Q_4}^{H_4}(1/2),\Sigma)\\
\geq&\Hom_{\GO_{2,2}(E)^\natural}(ind_{\GO_{4,0}(F)}^{\GO_{2,2}(E)^\natural}\mathbb{C}\oplus ind_{\GO_{2,2}(F)}^{\GO_{2,2}(E)^\natural}\mathbb{C}\oplus ind_{\GO_{3,1}(F)}^{\GO_{2,2}(E)^\natural}\mathbb{C},\Sigma ) \\
=&\dim\Hom_{\GO_{4,0}(F)}(\Sigma,\mathbb{C})+ 
\dim\Hom _{\GO_{2,2}(F)}(\Sigma,\mathbb{C})+\dim\Hom_{\GO_{3,1}(F)}(\Sigma,\mathbb{C}).	
	\end{split}
	\end{equation}

Let us turn the table around. There is an exact sequence of $\GSp_{8}(F)$-representations
	 \[\xymatrix{0\ar[r]&R^{3,3}(\mathbf{1})\ar[r]&\mathcal{I}(1/2)\ar[r]&R^{4,0}(\mathbf{1})\ar[r]&0 } \]
	 where $\mathcal{I}(s)$ is the degenerate principal series of $\GSp_8(F)$ and $R^{m,n}(\mathbf{1})$
	 is the big theta lift to $\GSp_8(F)$ of the trivial representation $\mathbf{1}$ of $\GO_{m,n}(F)$.
	 There is only one open orbit in the double coset decomposition $P_4\backslash\GSp_8(F)/\GSp_4(E)^\natural$.
	 In a similar way, thanks to \cite[Theorem 2.5]{olafsson1987fourier} and \cite[Proposition 4.9]{dima2018analytic}, 
	 \begin{eqnarray}
	 \begin{split}\label{reverse:ineq}
	 \dim\Hom_{\GSp_4(F)}(\tau,\mathbb{C})=&\dim\Hom_{\GSp_4(E)^\natural}(\mathcal{I}_0(1/2),\tau)\\
	  \leq&\dim\Hom_{\GSp_4(E)^\natural}(\mathcal{I}(1/2),\tau )\\
	  \leq&\dim\Hom_{\GSp_4(E)^\natural }(R^{3,3}(\mathbf{1}),\tau )+\dim\Hom_{\GSp_4(E)^\natural}(R^{4,0}(\mathbf{1}),\tau)\\
	  =&\dim\Hom_{\GO_{3,3}(F)}(\Theta_{6}^+(\tau),\mathbb{C})+\dim\Hom_{\GO_{4,0}(F)}(\Theta^+_4(\tau),\mathbb{C}) .   
	 \end{split}
	 \end{eqnarray}
Now we separate them into two cases:  $\pi_1\ncong\pi_2$ and $\pi_1\cong\pi_2$.
\begin{enumerate}[(A).]
	\item  If $\pi_1\ncong\pi_2,$ then the theta lift to $\GSp_2(E)$ of $\Sigma$ is zero,
	 $$\Hom_{\GO_{2,2}(E)^\natural}(R_1(\mathbf{1})\otimes\nu,\Sigma )=\Hom_{\GSp_2(F)}(\Theta_1(\Sigma),\mathbb{C})=0$$ and $\Sigma=\Ind_{\GSO(2,2)(E)}^{\GO(2,2)(E)}(\pi_1\boxtimes\pi_2).$ There are several subcases:
	\begin{enumerate}[label=(A\arabic*)]
		\item 
	If $\pi_i(i=1,2)$ are both $D^\times(F)$-distinguished, which implies
	that $\phi_{\pi_i}$ are conjugate-orthogonal and so that $\pi_i$ are both $\GL_2(F)$-distinguished, then $\pi_1^\vee\not\cong\pi_2^\sigma.$ Otherwise,  $\pi_1^\sigma\cong\pi_1^\vee\cong\pi_2^\sigma,$ which contradicts the assumption $\pi_1\ncong\pi_2.$  Then the inequality \eqref{key:inequality} can be rewritten as \begin{eqnarray}\label{lowbound2}
	\dim \Hom_{\GSp_4(F)}(\tau,\mathbb{C})\geq \dim \Hom_{\GO_{2,2}(E)^\natural}(I_{Q_4}^{H_4}(1/2),\Sigma)
	\geq 2.
	\end{eqnarray} 
	Flicker \cite{flicker1991ondist} proved that
	$(\GL_n(E),\GL_n(F))$ is a Gelfand pair, which implies that
	$$1\geq \Hom_{\GSO_{3,3}(F)}(\Theta_6^+(\tau),\mathbb{C})=\Hom_{\GO_{3,3}(F)}(\Theta^+_6(\tau),\mathbb{C}).$$ 
	Thus \begin{equation}\label{upperbound2}
	\dim \Hom_{\GSp_4(F)}(\tau,\mathbb{C})\leq1+1
	\end{equation} due to the upper bound \eqref{reverse:ineq}.  Then \eqref{lowbound2} and \eqref{upperbound2} imply  $$\dim\Hom_{\GSp_4(F)}(\tau,\mathbb{C})=2.$$
	\item	If $\pi_1=\pi(\chi_1,\chi_2),\chi_1\neq\chi_2,\chi_1|_{F^\times}=\chi_2|_{F^\times}=\mathbf{1}$ and $\pi_2$ is $\GL_2(F)$-distinguished, then both $\phi_{\pi_1}$ and $\phi_{\pi_2}$ are conjugate-orthogonal, $\pi_1^\vee\ncong\pi_2^\sigma$ and $$\Hom_{\GO_{4,0}(F)}(\Sigma,\mathbb{C})=0=\Hom_{\GO_{3,1}(F)}(\Sigma,\mathbb{C}).$$ 
	Moreover, $\Hom_{\GSO_{3,3}(F)}(\Theta_6^{+}(\tau),\mathbb{C})\neq0.$
	Since 
	$$\dim \Hom_{\GO_{2,2}(E)^\natural}(I_{Q_4}^{H_4}(1/2),\Sigma)\geq\dim \Hom_{\GO(2,2)(F)}(\Sigma,\mathbb{C})+0=1, $$ 
	the desired equality $\dim \Hom_{\GSp_4(F)}(\tau,\mathbb{C})=1$ follows from \eqref{key:inequality} and \eqref{reverse:ineq}.
\item	If $\pi_1^\sigma\cong\pi_2^\vee,$ then $$\dim \Hom_{\GO_{3,1}(F)}(\Sigma,\mathbb{C})=1.$$  
 By the previous arguments, we know that $\Hom_{\GO_{2,2}(F) }(\Sigma,\mathbb{C})=0$ in this case.
Therefore $$\dim \Hom_{\GSp_4(F)}(\tau,\mathbb{C} )=1.$$
\end{enumerate}
In other cases, if $\pi_1^\sigma\ncong \pi_2^\vee$ and either $\phi_{\pi_1}$ or $\phi_{\pi_2}$  is not conjugate-orthogonal, then 
 $$\dim \Hom_{\GSp_4(F)}(\tau,\mathbb{C})=0. $$ 
  If not, then
\[\dim \Hom_{\GSp_4(F)}(\tau,\mathbb{C})=\dim \Hom_{ \GSO_{3,3}(F)}(\Theta_6^+(\tau),\mathbb{C})=1. \]
	Set $\Pi\boxtimes\chi=\Theta_6^+(\tau)|_{\GSO_{3,3}(E)}$ as a representation of $\GSO_{3,3}(E)$, which is irreducible due to Proposition \ref{big:GSp4GO6}. Then $\Pi$ is $\GL_4(F)$-distinguished and so $\phi_\Pi$ is conjugate-orthogonal. 
	\par
	We consider the following cases:
	\begin{itemize}
		\item If $\phi_{\pi_1}$ is conjugate-orthogonal, then $\phi_{\pi_2}$ is conjugate-orthogonal by \cite[Theorem 5.2]{matringe2009distinction}.
			\item If $\phi_{\pi_1}$ is irreducible, then $\phi_{\pi_1}$ is conjugate-orthogonal, which will imply that $\phi_{\pi_2}$ is conjugate-orthogonal as well.
		\item Now suppose that both $\phi_{\pi_1} $ and $\phi_{\pi_2}$ are reducible	and that neither $\phi_{\pi_1}$ nor $\phi_{\pi_2}$ is conjugate-orthogonal.
		Assume that $\phi_{\pi_i}=\chi_{i1}+\chi_{i2}~ (i=1,2)$. Then
			$$\phi_\Pi=\chi_{11}+\chi_{12}+\chi_{21}+\chi_{22},~\chi_{11}\chi_{12}=\chi_{21}\chi_{22}:E^\times/F^\times\rightarrow\mathbb{C}^\times.$$
	 Thanks to  \cite[Theorem 5.2]{matringe2009distinction},
		$\chi_{11}\chi_{21}^\sigma=\mathbf{1}$ and $\chi_{12}\neq\chi_{22} $ but $\chi_{12}|_{F^\times}=\mathbf{1}=\chi_{22}|_{F^\times}.$ Moreover, $\chi_{21}\chi_{22}\cdot(\chi_{21}\chi_{22})^\sigma=\mathbf{1}$ implies 
		\[\chi_{21}^\sigma\chi_{21}=\mathbf{1}. \]
		Similarly $\chi_{11}^\sigma\chi_{11}=\mathbf{1}$.
		 Thus,
		$\chi_{21}^\sigma=\chi_{21}^{-1}$ and $\chi_{11}=\chi_{21}.$
		This implies that $\chi_{12}=\chi_{22}$ which contradicts  the condition $\chi_{12}\neq\chi_{22}.$
	\end{itemize}
	Hence the Langlands parameter $\phi_{\Pi}$ can not be conjugate-orthogonal. Thus $\Hom_{\GSp_4(F)}(\tau,\mathbb{C} )=0$ if $\pi_1^\sigma\ncong\pi_2^\vee$ and either $\phi_{\pi_1}$ or $\phi_{\pi_2}$ is not conjugate-orthogonal.
\item If $\pi_1=\pi_2$, then $\Theta_1(\Sigma)=\pi_1,$ which is  obvious except for $\pi_1=\pi(\chi,\chi).$
 The exact sequence \eqref{injection} implies the following inequality
 \begin{equation}\label{lowerinequality}
 \dim \Hom_{\GSp_4(F)}(\tau,\mathbb{C} )\geq \dim \Hom_{\GO_{2,2}(E)^\natural}(I_{Q_4}^{H_4}(1/2),\Sigma)-\dim \Hom_{\GSp_2(F)}(\pi_1,\mathbb{C}). 
 \end{equation}
 We separate them into  following cases:
\begin{enumerate}[label=(B\arabic*)]
	\item If $\pi_1$ is $D^\times(F)$-distinguished, then $\dim \Hom_{\GO_{2,2}(E)^\natural}(I_0(1/2),\Sigma)=3. $ Again, we consider
	the upper bound \eqref{reverse:ineq} and the lower bound \eqref{lowerinequality}  to obtain the equality $$\dim \Hom_{\GSp_4(\mathbb{C})}(\tau,\mathbb{C})=2.$$
	\item If $\pi_1\cong\pi_0=\pi(\chi_1,\chi_2)$ with $\chi_1\neq\chi_2$ and $\chi_1|_{F^\times}=\chi_2|_{F^\times}=1,$ then $$\dim \Hom_{\GO_{4,0}(F)}(\Sigma,\mathbb{C})=0.$$   
	In a similar way,
	we can get $\dim \Hom_{\GSp_4(F)}(\tau,\mathbb{C})=1.$
	\item If  $\pi_1$ is   not $\GL_2(F)$-distinguished but $(\GL_2(F),\omega_{E/F})$-distinguished, then $$\Hom_{\GSp_2(F)}(\pi_1,\mathbb{C})=0 
\mbox{  and  }	\Hom_{\GO_{3,1}(F)}(\Sigma,\mathbb{C})\neq0, $$
	which implies that $\dim \Hom_{\GO_{2,2}(E)^\natural}(I_{Q_4}^{H_4}(1/2 ),\Sigma)\geq1=\dim \Hom_{\GSO_{3,3}(F)}(\Theta_6^{+}(\tau),\mathbb{C}). $
	Therefore we can deduce that $\dim \Hom_{\GSp_4(F)}(\tau,\mathbb{C})=1.$
\end{enumerate}
If $\pi_1=\pi(\chi,\chi),$ then
there is an exact sequence
 $$\xymatrix{\pi_1\ar[r]&\Theta_2(\pi_1\boxtimes\pi_1)\ar[r]&\pi_1\ar[r]& 0}$$
  of $\GL_2(E)$-representations,
where we can not deduce $\Theta_2(\pi_1\boxtimes\pi_1)$ directly. Let $N=\big\{\begin{pmatrix}
1&n\\0&1
\end{pmatrix}|n\in F \big\}$ be the subgroup of $\GSp_2(F)$. Let $\psi_N$ be a nontrivial  character of $N$.
 Now we consider the Whittaker model of $\Theta_1(\pi_1\boxtimes\pi_1)$,
 \[\dim \Hom_N(\Theta_1(\pi_1\boxtimes\pi_1 ),\psi_N )=\dim\Hom_{\PGL_2(E)}(\pi_1\boxtimes\pi_1,\mathbb{C} )\leq1 \]
 which implies that $\Theta_1(\Sigma)=\pi_1.$
\end{enumerate}
			\item  If $\tau$ is not in case (i) or (ii), then
			  the first occurence index of $\tau$ of $\GSp_4(E)$ in the Witt Tower $\mathbb{H}_E^r$ is $3$. Observe that
			 $\Theta_6^+(\tau)|_{\GSO_{3,3}(E)}$ is irreducible unless $\tau=\Ind_{Q(Z)}^{\GSp_4(E)}(\chi,\pi)$ with $\chi=|-|_E$. 
			 \par
			 Suppose that $\tau\neq\Ind_{Q(Z)}^{\GSp_4(E)}(|-|_E,\pi)$.
			 Consider the double see-saw diagrams 
			 \[\xymatrix{\GO_{2,2}(E)^\natural\ar@{-}[d]\ar@{-}[rd]&\GSp_8(F)\ar@{-}[d]\ar@{-}[rd]& \GO_{3,3}(E)^\natural\ar@{-}[d]\ar@{-}[ld]\\ \GO_{4,0}(F)\ar@{-}[ru] &\GSp_4(E)^\natural& \GO_{3,3}(F). } \]
			 By \cite[Page 211]{Kudla1992} and Proposition \ref{degenerateseries}, there are two exact sequences of $\GSp_8(F)$-modules
			 \[\xymatrix{0\ar[r]&R^{3,3}(\mathbf{1})\ar[r]&\mathcal{I}(1/2)\ar[r]&R^{4,0}(\mathbf{1})\ar[r]&0 } \]
			 and \[\xymatrix{0\ar[r]&R^{4,0}(\mathbf{1})\oplus R^{2,2}(\mathbf{1})\ar[r]&\mathcal{I}(-1/2)\ar[r]&R^{5,1}(\mathbf{1})\cap R^{3,3}(\mathbf{1})\ar[r]&0 }  \]
			 where $\mathcal{I}(s)$ is the degenerate principal series of $\GSp_8(F)$ and $R^{m,n}(\mathbf{1})$
			 is the big theta lift to $\GSp_8(F)$ of $\mathbf{1}$ of $\GO_{m,n}(F)$.  Assume that  $\tau$ is generic and its theta lift to $\GO_{2,2}(E)$ is zero.
			 Then  $$\Hom_{\GSp_4(E)^\natural}(R^{4,0}(\mathbf{1}),\tau)=\Hom_{\GO_{4,0}(F)}(\Theta_4^+(\tau),\mathbb{C})=0,$$
			  so that  
			  $$\dim\Hom_{\GSp_4(E)^\natural}(\mathcal{I}(-1/2),\tau)=\dim\Hom_{\GSp_4(E)^\natural }(R^{5,1}(\mathbf{1})\cap R^{3,3}(\mathbf{1}),\tau ).$$ 
			  Thus
			 \begin{equation}\label{openorbitforGSp4}
			 \begin{split}
			\dim \Hom_{\GSp_4(F)}(\tau^\vee,\mathbb{C})
			 &=\dim \Hom_{\GSp_4(E)^\natural}(\mathcal{I}_0(1/2),\tau )\\
			 &\leq \dim\Hom_{\GSp_4(E)^\natural}(\mathcal{I}(1/2),\tau)\\
			 &\leq \dim\Hom_{\GSp_4(E)^\natural}(R^{3,3}(\mathbf{1}),\tau)\\
			 &=\dim \Hom_{\GO_{3,3}(F)}(\Theta_6^+(\tau),\mathbb{C} )\\
			 &=\dim \Hom_{\GO_{3,3}(F)}((\Pi\boxtimes\chi)^+,\mathbb{C})
			 \end{split}
			 \end{equation}
where 	$(\Pi\boxtimes\chi)^\pm$ are two extensions to $\GO_{3,3}(E)$ of $\Pi\boxtimes\chi$.		On the other hand, one has
			\[\Hom_{\GO_{3,3}(F)}((\Pi\boxtimes\chi)^-,\mathbb{C})=\Hom_{\GO_{3,3}(F)}(\Theta_6^+(\tau)\otimes\nu,\mathbb{C})\cong \Hom_{\GSp_4(E)^\natural}(\Theta(\nu),\tau)=0. \]  Then we have an inequality
			 $$	\dim	 \Hom_{\GSp_4(F)}(\tau,\mathbb{C})	\leq \dim \Hom_{\GSO_{3,3}(F)}(\Pi\boxtimes\chi,\mathbb{C})=\dim \Hom_{\GL_4(F)}(\Pi,\mathbb{C}).$$
Now we want to obtain the reverse inequality.
Note that $$\xymatrix{ 1\ar[r]& R^{5,1}(\mathbf{1})\cap R^{3,3}(\mathbf{1})\ar[r]& R^{3,3}(\mathbf{1})\ar[r]& R^{2,2}(\mathbf{1})\ar[r]& 1 }$$ is exact due to \cite[Proposition 7.2]{gan2014formal}.
There is an injection \begin{equation}\label{reverseinequality}
\Hom_{\GSp_4(E)^\natural}(R^{3,3}(\mathbf{1}),\tau)\hookrightarrow \Hom_{\GSp_4(E)^\natural}(R^{5,1}(\mathbf{1})\cap R^{3,3}(\mathbf{1}),\tau )=\Hom_{\GSp_4(E)^\natural}(\mathcal{I}(-1/2),\tau) \end{equation}
since the theta lifts to $\GO_{2,2}(E)$ and $\GO_{4,0}(E)$ of $\tau$  are both zero by the assumption.

We will show that $\tau$ does not occur on the boundary of $\mathcal{I}(-1/2)$ under the assumptions.
If $\tau$ is non-discrete, then $\tau=J_{Q(Z)}(\chi,\pi),\chi\neq\mathbf{1},$
due to \cite[Table 1]{gan2011theta}.
Note that $$\mathcal{I}_1(s)/\mathcal{I}_0(s)=ind_{(E^\times \times \GSp_2(F))N'}^{\GSp_4(E)^\natural}\chi' $$ where $N'\cong E\oplus Mat_{2,2}(F)$ and
$\chi'(t,g)=|N_{E/F}(t)^{s+\frac{1}{2} }\cdot\lambda(g)^{-2s-3}|_F$.
Set $$P'=(\GL_1(E)\times \GSp_2(E)^\natural )\cdot N'.$$ 
Thanks to the second adjoint theorem due to Bernstein,
we have
\[ \Hom(\mathcal{I}_1(-1/2)/\mathcal{I}_0(-1/2),\tau )=\Hom_{E^\times\times \Sp_2(E)\times F^\times }(\mathbf{1}\otimes ind_{\Sp_2(F)}^{\Sp_2(E)}\mathbb{C}\otimes|-|_F^{-2},R_{\bar{P'}}(J_{Q(Z)}(\chi,\pi))  ) =0, \]
because $R_{\bar{P'}}(J(\chi,\pi))=\chi\otimes\pi+\chi^{-1}\otimes\pi\chi$ and $\chi\neq\mathbf{1}.$
Moreover, the cuspidal supports of $J_{Q(Z)}(\chi,\pi)$ and $\mathcal{I}_2(-1/2)/\mathcal{I}_1(-1/2)$ are disjoint. Therefore $\tau=J_{Q(Z)}(\chi,\pi)$ does not occur on the boundary of $\mathcal{I}(-1/2)$ and so
\[\dim\Hom_{\GSp_4(E)^\natural}(\mathcal{I}(-1/2),\tau)\leq\dim\Hom_{\GSp_4(E)^\natural}(\mathcal{I}_0(-1/2),\tau)=\dim\Hom_{\GSp_4(F)}(\tau,\mathbb{C}).   \]
Note that if $\tau$ is a discrete series representation. Then we have $$\Hom_{\GSp_4(E)^\natural}(\mathcal{I}_{i+1}(-1/2)/\mathcal{I}_i(-1/2),\tau )=0$$ for $i=0,1$. If not, then
we will get a contradiction. Suppose that $$\Hom_{\GSp_4(E)^\natural}(\mathcal{I}_1(-1/2)/\mathcal{I}_0(-1/2),\tau)\neq0. $$ 
Then 
$\Hom_{\GL_1(E)}(\mathbf{1},R_{\bar{P'}}(\tau))\neq0  $
which contradicts Casselman's criterion \cite{casselman82duke} for the discrete series representation  that 
\[\Hom_{\GL_1(E)}(|-|_E^{s},R_{\bar{P'}}(\tau))\neq0  \]
implies $s<0$. Similarly,
\[\Hom_{\GSp_4(E)^\natural}(\mathcal{I}_2(-1/2)/\mathcal{I}_1(-1/2),\tau )=\Hom_{\GL_2(E)\times F^\times}(\delta_{P^\natural}^{\frac{1}{6}},R_{\bar{P}^\natural}(\tau) )=0   \]
  and so
\begin{equation}\label{-halfopen}\dim \Hom_{\GSp_4(E)^\natural }(\mathcal{I}(-1/2),\tau )\leq \dim \Hom_{\GSp_4(E)^\natural }(\mathcal{I}_0(-1/2),\tau ). \end{equation}
Therefore one can combine \eqref{openorbitforGSp4},\eqref{reverseinequality} and \eqref{-halfopen} to obtain that
\[\dim \Hom_{\GSp_4(F)}(\tau,\mathbb{C})=\dim
\Hom_{\GSp_4(E)^\natural}(\mathcal{I}_0(-1/2),\tau)=\dim \Hom_{\GSO_{3,3}(F)}(\Theta_6^+(\tau),\mathbb{C}) . \]
The right hand side is $1$ if and only if $\Pi$ is $\GL_4(F)$-distinguished.
\par
If $\tau=\Ind_{Q(Z)}^{\GSp_4(E)}(|-|_E,\pi)$ is irreducible, then
$\theta_6(\tau)=J_P(\pi|-|_E,\pi)\boxtimes\omega_{\pi}|-|_E$. It suffices to show that 
$I_P(\pi|-|_E,\pi)$ is $\GL_4(F)$-distinguished if and only if $J_P(\pi|-|_E,\pi)$ is $\GL_4(F)$-distinguished. It follows from Lemma \ref{arnab}.
	\end{enumerate}
	Then we have finished the proof.
			 \end{proof}
		
			 	\begin{rem}
			 		In fact, we can show that if  $\tau=\theta(\pi_1\boxtimes\pi_2)$ with $\pi_1^\vee\cong\pi_2^\sigma$ is generic, then $\phi_\Pi=\phi_{\pi_1}\oplus\phi_{\pi_2}$ is not only conjugate-orthogonal but also conjugate-symplectic. Keeping this fact in mind   will be helpful when we verify the Prasad conjecture for $\GSp_4$ in \S\ref{subsect:GSp(4)conj}.
			 		\end{rem}

\begin{coro}
	The pair $(\GSp_4(E)^\natural,\GSp_4(F))$ is not a Gelfand pair.
\end{coro}
For a generic representation $\tau$ of $\GSp_4(E)$ with $\omega_\tau|_{F^\times}=\chi_F^2,$
we may consider the muliplicity $$\dim\Hom_{\GSp_4(F)}(\pi,\chi_F ) $$
which is equal to $\dim\Hom_{\GSp_4(F)}(\pi\otimes\chi_E^{-1},\mathbb{C}) ,$ where $\chi_E|_{F^\times}=\chi_F.$
 We will focus on the case $\chi_F=\omega_{E/F}$ when we verify the Prasad conjecture for $\GSp_4$ in \S\ref{subsect:GSp(4)conj}.

%% file: PGSp.tex
\section{The $\GSp_{1,1}(F)$-distinguished representations}	\label{sect:GSp(1,1)}
\subsection{Notation}
\begin{itemize}
	\item $D$ (resp. $D_E$) is a quaternion divison algebra over $F$ (resp. $E$) with a standard involution $\ast$.
	\item $\pi^{D_E}$ is the Jacquet-Langlands lift to $D_E^\times(E)$ of $\pi$ and $\pi^{D_E}\boxtimes\pi^{D_E}$ is a representation of $\GSO_{4,0}(E)$.
	\item $\mathfrak{W}$ (resp. $\mathfrak{V}$) is a right skew-Hermitian (resp. left Hermitian) $D$-vector space with isometry group $U(\mathfrak{W})$ (resp. $U(\mathfrak{V})$).
	\item $\mathfrak{U}^\ast$ is the dual $D$-vector space of $\mathfrak{U}$ in $\Res_{R/D}V_R$.
	\item $\mathfrak{W}\otimes_D\mathfrak{V}$ is a symplectic $F$-vector space.
	\item $\GO_{3,0}^\ast$ (resp. $\GO_{r,r}^\ast$) is the inner form of $\GO_{3,3}$ (resp. $\GO_{2r,2r}$) defined over $F$.
		\item $\mathfrak{I}(s)$ (resp. $I(s)$) is the degenerate principal series of $\GSp_{2,2}(F)$ (resp. $\GO_{2,2}^\ast(F)$).
		\item $\GSO^\ast_{2,0}$ is the inner form of $\GSO_{3,1}$ defined over $F$.
	\item $\GO_{5,1}$  is the pure inner form of $\GO_{3,3}$ defined over $E$ and $\Pi^D\boxtimes\chi$ is a representation of $\GSO_{5,1}(E)$.
	\item $B_1$ is the minimal parabolic subgroup of $\GL_2(D_E)(E)$.
	\item $\GSp_{1,0}=D^\times$ (resp. $\Sp_{1,0}$) is the inner form of $\GL_2$ (resp. $\SL_2$).
			\item $P(Y_D)$ (resp. $\mathfrak{Q}$) is the Siegel parabolic subgroup of $\GU(\mathfrak{V})$ (resp. $\GO_{2,2}^\ast(F)$).
	\item $\mathfrak{R}^3(\mathbf{1})$ (resp. $\mathfrak{R}^2(\mathbf{1})$) is the big theta lift to $\GSp_{2,2}(F)$ of the trivial representation of $\GO_{3,0}^\ast(F)$ (resp. $\GO_{1,1}^\ast(F)$) and $\mathfrak{R}^{1,j}(\mathbf{1})$ is the big theta lift to $\GO_{2,2}^\ast(F)$ from $\GSp_{1,j}(F)$.
	\item $\theta_2^-(\tau)$ (resp. $\Theta_{2}^-(\tau)$) is the small (resp. big) theta lift to $\GO_{5,1}(E)$ of $\tau$ of $\GSp_4(E)$.
	\item $\Theta_{\mathfrak{W},\mathfrak{V},\psi}(\pi)$ is the big theta lift to $\GU(\mathfrak{V})$ of $\pi$ of $\GU(\mathfrak{W})$.
	\item $\gamma_F$ is the Weil index and $\gamma_F(\psi\circ q)\in\mu_8$ for the character of second degree $x\mapsto \psi(q(x,x))$, where $q$ is a non-degenerate symmetric $F$-bilinear form.
\end{itemize}
\subsection{Theta lifts for quaternionic unitary groups}
In order to study the $\GSp_{1,1}$-distinction problems,
we need to introduce the local theta lift for quaternionic unitary groups, following \cite{gan2014inner,gurevich2015non,yamana2011deg}.
\subsubsection{Morita equivalence}
Let $R=Mat_{2,2}(E)$ be the split quaternion algebra over $E.$ Any left Hermitian (resp. right skew-Hermitian) free $R$-module $(W_R,h_R)$  corresponds to a symplectic (resp. orthogonal) space $(W_E,h_E)$ over $E$ and 
\[\dim_E W_E=2\cdot\dim_R W_R , ~Aut(W_R,h_R)=Aut(W_E,h_E).\] 
(See \cite[\S2.1]{gurevich2015non} for more details.)

\subsubsection{Dual pairs}
Let $D$ be the unique nonsplit quaternion algebra over $F,$ with a standard involution $\ast$. Then
$D\otimes_F E\cong R.$ There is a $D$-linear map
\[tr_{R/D}:R\longrightarrow D \]
such that $tr_{R/D}(d)=2d$ for $d\in D.$ Given a $4$-dimensional symplectic space
$(\mathcal{W}_2,h_E)$ over $E,$ corresponding to a $2$-dimensional left Hermitian space $(W_R,h_R),$ we set
\[h_D(x,y)=\frac{1}{2}tr_{R/D}(h_R(x,y))\in D\]
 for all $x,y\in W_R$.
Then $h_D$ is a nondegenerate Hermitian form on  $\mathfrak{V}=\Res_{R/D}W_R$ and $\dim_D \mathfrak{V}=4.$

Given a left Hermitian space $(\mathfrak{V},h_D)$  and a right skew-Hermitian space $(\mathfrak{W},s_D),$ the tensor product space $\mathfrak{W}\otimes_D \mathfrak{V}$
admits a symplectic form defined by
\[<w\otimes v,w'\otimes v'>:=\frac{1}{2}tr_{D/F}((w,w')\cdot(v,v')^\ast) . \]
This gives an embedding of $F$-groups
\[U(\mathfrak{W})\times U(\mathfrak{V})\longrightarrow \Sp(\mathfrak{W}\otimes_D \mathfrak{V}). \]
Then we can define the Weil representation $\omega_\psi$ on $U(\mathfrak{W})\times U(\mathfrak{V}),$ using the complete polarization  $\mathfrak{V}=Y_D+Y_D^\ast$ of $\mathfrak{V}$. 
\begin{thm}\cite[Theorem 1.2]{gan2015howe}
	The Howe duality conjecture holds for the dual pair $U(\mathfrak{W})\times U(\mathfrak{V})$.
\end{thm}
We extend it to the similitude group
$\GU(\mathfrak{W})\times \GU(\mathfrak{V})$ following Roberts. (See \cite[\S 3]{gan2014inner}.)

\subsubsection{The see-saw diagram}
Let us fix the polarization $W_R=Y_R+Y_R^\ast$.
Then $$\mathfrak{V}=\mbox{Res}_{R/D}W_R=Y_D+Y^\ast_D.$$
Consider the following see-saw diagram
\[\xymatrix{\GU(\mathfrak{V})\ar@{-}[rrd]\ar@{-}[d] && \GO_{2,2}(E)^\natural\ar@{-}[lld]\ar@{-}[d]\\ \GU(W_R)^\natural&& \GO^\ast_{1,1}(F) .} \]
Here  $\GU(W_R)^\natural=\GSp_4(E)^\natural$. 
\begin{prop}\cite[Theorem 8.2]{gurevich2015non}
	Let $\tau$ be an irreducible representation of $\GSp(\mathcal{W}_2)\cong \GU(W_R).$ Assume that  $\pi$ is an irreducible representation of $\GO_{1,1}^\ast(F)$. Then
	\[\Hom_{\GU(W_R)^\natural }(\Theta_{\mathfrak{W},\mathfrak{V},\psi}(\pi),\tau )=\Hom_{\GO_{1,1}^\ast(F)}(\Theta_4^+(\tau),\pi ). \]  
\end{prop}

	Assume that  $V_R$ is a skew-Hermitian free module over $R$ of rank $2$, associated to the anisotropic $4$-dimensional quadratic space over $E$ given by $(D_E, N_{D_E})$ such that
	\[ \GU(V_R)\cong\GO_{4,0}(E). \]
	 Then $\Res_{R/D}V_R$ is a $4$-dimensional skew-Hermitian $D$-vector space with trivial discriminant. There is a natural embedding
\[\mathrm{SU}(V_R)\cong\SO_{4,0}(E)\hookrightarrow \SO^\ast_{2,2}(F)=\mathrm{SU}(\Res_{R/D} V_R).  \]
Given a $1$-dimensional Hermitian vector space $\mathfrak{V}_1$ over $D,$ we consider the theta lift from $\GU(\mathfrak{V}_1)=\GSp_{1,0}(F)$ to $\GO^\ast_{2,2}(F)$ and the theta lift from $\GSO_{4,0}(E)$ to $\GU(R\otimes_D\mathfrak{V}_1)=\GL_2(E)$.
Consider the following see-saw diagram
\[\xymatrix{\GU(\Res_{R/D}V)\ar@{-}[d]\ar@{-}[rd] & \GL_2(E)^\natural\ar@{-}[d]\ar@{-}[ld]\\ \GSO_{4,0}(E)^\natural& \GSp_{1,0}(F) } \]
which is different from the situation in  \cite[Theorem 8.2]{gurevich2015non}, since there does not exist a natural polarization in the symplectic $F$-vector space $\mathbb{V}=(\Res_{R/D}V_R)\otimes_D \mathfrak{V}_1.$ 

Assume that  $\mathbb{V}=\mathbb{X}\oplus\mathbb{Y}$ is a polarization. Set the group
\[\Mp(\mathbb{V})_{\mathbb{Y}}=\Sp(\mathbb{V})\times\mathbb{C}^\times \]
with group law
\[(g_1,z_1)(g_2,z_2)=(g_1g_2,z_1z_2\cdot z_{\mathbb{Y}}(g_1,g_2)) \]
where $z_{\mathbb{Y}}(g_1,g_2)=\gamma_F(\frac{1}{2}\psi\circ q(\mathbb{Y},g_2^{-1}\mathbb{Y},g_1\mathbb{Y} ))$ is a $2$-cocycle (called Rao cocyle) associated to $\mathbb{Y}$ and $q(\mathbb{Y},g_2^{-1}\mathbb{Y},g_1\mathbb{Y})$ is the Leray invariant. (See \cite[\S I.3]{kudla1996notes}.)

Suppose that $\mathbb{V}=\mathbb{X}'\oplus\mathbb{Y}'$ is another polarization of $\mathbb{V}.$ There is an isomorphism 
\[\mathcal{S}(\mathbb{X})\cong\mathcal{S}(\mathbb{X}'). \]
Given $\varphi\in \mathcal{S}(\mathbb{X})$ and $\varphi'\in\mathcal{S}(\mathbb{X}'),$ due to \cite[Lemma 3.3]{ichino2015periods},  we have
\[\varphi(x)=\int_{\mathbb{Y}\cap\mathbb{Y}'\backslash\mathbb{Y} }\psi(\frac{1}{2}<x',y'>-\frac{1}{2}<x,y> )\varphi'(x')dy \]
where $x'\in\mathbb{X}'$ and $y'\in\mathbb{Y}'$ are given by $x'+y'=x+y\in\mathbb{V}.$

	\begin{lem}[Local Siegel-Weil identity]
	Assume that  $\pi$ is an irreducible discrete series representation of $\GL_2(E)$ so that the big theta lift $\Theta(\pi)$  to $\GSO_{4,0}(E)$ is isomorphic to $\pi^{D_E}\boxtimes\pi^{D_E}$, where $\pi^{D_E}
	$ is the Jacquet-Langlands lift to $D_E^\times(E)$ of $\pi.$ Let $\varrho$ be an irreducible representation of $\GSp_{1,0}(F)$. Then
	\[\dim \Hom_{\GSO_{4,0}(E)^\natural }(\Theta(\varrho),\pi^{D_E}\boxtimes\pi^{D_E})=\dim \Hom_{\GSp_{1,0}(F)}(\pi,\varrho ) \]
	where $\Theta(\varrho)$ is the big theta lift to $\GO^\ast_{2,2}(F)$ of $\varrho$. 
\end{lem}
\begin{proof} It suffices to show that two splittings of $\SO_{4,0}(E)\times \Sp_{1,0}(F)$ in $\Mp(\mathbb{V})$ are compactible. 
	Let us fix two polarizations $\mbox{Res}_{R/D}V_R=\mathfrak{U}\oplus \mathfrak{U}^\ast$ and $R\otimes_D\mathfrak{V}_1= X\oplus Y$. Then
	\[\mathbb{V}=\mathbb{X}\oplus\mathbb{Y}= (\mathfrak{U}\otimes_D \mathfrak{V}_1)\oplus (\mathfrak{U}^\ast\otimes_D \mathfrak{V}_1)\mbox{  and  } \mathbb{V}=\mathbb{X}'\oplus\mathbb{Y}'=(D_E\otimes_E X)\oplus (D_E\otimes_E Y). \]
	Choose a fixed element $h_0\in \Sp(\mathbb{V})$ such that $$\mathbb{X}'=h_0\mathbb{X}\mbox{  and  } \mathbb{Y}'=h_0\mathbb{Y}.$$
	By \cite[Appendix B.4]{ichino2015periods}, there is an isomorphism $\alpha_0:\Mp(\mathbb{V})_{\mathbb{Y}'}\rightarrow \Mp(\mathbb{V})_{\mathbb{Y}}$ via
	$$(h,z)\mapsto(\alpha_0(h),z)$$
	  where  $
	\alpha_0(h)=h^{-1}\cdot g\cdot h$ for all $h\in \Sp(\mathbb{V}). $ Moreover,
	\[z_{\mathbb{Y}'}(h_1,h_2 )=z_{\mathbb{Y}}(\alpha_0(h_1),\alpha_0(h_2) ). \]	
	Now we fix the splitting $i_\mathbb{Y}:\Oo^\ast_{2,2}(F)\times \Sp_{1,0}(F)\hookrightarrow \Mp(\mathbb{V})_{\mathbb{Y}}$ and
	\[i_{\mathbb{Y}'}:\SO_{4,0}(E)\times \Sp_2(E)\hookrightarrow \Mp(\mathbb{V})_{\mathbb{Y}'} \]
	where the splitting $i_\mathbb{Y}(y,z)=((y,z),\beta_\mathbb{Y}(z))$ is  defined in \cite[Theorem 3.1]{kudla1994splitting}.
	
	We will show that $ i_{\mathbb{Y}}(h)=\alpha_0\circ i_{\mathbb{Y}'}(h)$ for all $h=(y,z)\in \SO_{4,0}(E)\times \Sp_{1,0}(F)$. Consider
	\[\xymatrix{\SO_{4,0}(E)\times \Sp_{1,0}(F)\ar@{^{(}->}[r]\ar@{=}[d] & \Oo^\ast_{2,2}(F)\times \Sp_{1,0}(F)\ar[r]^-{i_{\mathbb{Y}}} & \Mp(\mathbb{V})_{\mathbb{Y}} \\
		\SO_{4,0}(E)\times \Sp_{1,0}(F)\ar@{^{(}->}[r] & \SO_{4,0}(E)\times \Sp_2(E)\ar[r]^-{i_{\mathbb{Y}'}}& \Mp(\mathbb{V})_{\mathbb{Y}'}.\ar[u]^-{\alpha_0} } \]
	Set $i_\mathbb{Y}(h)=(h,\beta_\mathbb{Y}(h))$. Then $\beta_\mathbb{Y}(z)=1$ for all $z\in \Sp_{1,0}(F)$. Similarly, we have
	$$\beta_{\mathbb{Y}'}(y)=1$$
	for all $y\in \SO_{4,0}(E).$ In order to show that \[\beta_\mathbb{Y}(h)=\beta_{\mathbb{Y}'}(h)\]
	for all $h=(y,z)\in \SO_{4,0}(E)\times \Sp_{1,0}(F)$,
	we will show that $\beta_{\mathbb{Y}}(y)=1=\beta_{\mathbb{Y}'}(z).$
	\begin{itemize}
		\item If $y\in \SO_{4,0}(E)\subset  \Oo^\ast_{2,2}(F)=\bigsqcup_{i=0}^2 \mathfrak{P}\omega_i\mathfrak{P},$ say $y\in \mathfrak{P}\omega_{i}\mathfrak{P},$ where $\mathfrak{P}$ is the Siegel parabolic subgroup of $\Oo^\ast_{2,2}(F),\omega_0=\mathbf{1}_4$ (the identity matrix in $\Oo_{2,2}^\ast(F)$),$$\omega_1=\begin{pmatrix}
		&&1\\&1\\1&\\&&&1
		\end{pmatrix}\mbox{  and  }\omega_2=\begin{pmatrix}
		&&1\\&&&1\\1\\&1
		\end{pmatrix},$$ 
		then $\beta_{\mathbb{Y}}(y)=(-1)^i.$
		Since $\omega_1$ switches a pair of vectors $e_1$ and $e_1'$ in a basis $\{e_1,e_2,e_1',e_2' \}$, which corresponds to an element $h\in \Oo_{4,0}(E)$ with  determinant  $-1,$ where  $\mathfrak{P}$ stabilizes the maximal isotropic subspace $\{e_1,e_2 \},$ it follows that \[\SO_{4,0}(E)\cap \mathfrak{P}\omega_1\mathfrak{P}=\emptyset,\]
		   i.e. , $\beta_\mathbb{Y}(y)=1.$
		\item If $z\in \Sp_{1,0}(F)$ and so $z=g\in \SL_2(E),$ then $\beta_{\mathbb{Y}'}(z)=\gamma_F(x(g),\frac{1}{2}\psi )^4\cdot\gamma_F(\frac{1}{2}\psi\circ N_{D_E} )^4=1$, where
		\[x(g)=\begin{cases}
		N_{E/F}(a_{21})\pmod{{F^\times}^2},&\mbox{ if }g=\begin{pmatrix}
		a_{11}&a_{12}\\a_{21}&a_{22}
		\end{pmatrix}\mbox{ with }a_{21}\neq0;\\
		N_{E/F}(a_{22})\pmod{{F^\times}^2} ,&\mbox{otherwise.}
		\end{cases} \]
	\end{itemize}
	Therefore we have finished the proof.
\end{proof}
\begin{rem}\label{counterforseesaw}
	From the proof above, we can see that the see-saw identity does not hold if one replaces $\SO_{4,0}(E)$ by $\Oo_{4,0}(E)$ in this case.
\end{rem}
Let $V$ be a free $R$-module of rank $2$ corresponding to the quadratic space $\mathbb{H}_E^2$ by the Morita equivalence. Then $\Res_{R/D}V$ is a skew-Hermitian $D$-vector space of dimension $4$.
\begin{lem}
	Let $\Sigma$ be an irreducible representation of $\GO_{2,2}(E)$. Let $\varrho$ be a representation of $\GSp_{1,j}(F)$ for $j=0$ or $1$. Then 
	\[\dim\Hom_{\GO_{2,2}(E)^\natural}(\Theta(\varrho),\Sigma)=\dim\Hom_{\GSp_{1,j}(F)}(\Theta_{1+j}(\Sigma\cdot\nu^{1+j}),\varrho)    \]
	where $\nu$ is the nontrivial character of $\GO_{2,2}(E)/\GSO_{2,2}(E)$ and $\nu|_{\Oo_{2,2}(E)}=\det$.
\end{lem}
\begin{proof}
	Consider the following see-saw diagram
	\[\xymatrix{ \GO_{2,2}^\ast(F)\ar@{-}[rd]\ar@{-}[d] &\GSp_{2+2j}(E)^\natural\ar@{-}[d]\ar@{-}[ld] \\
		\GO_{2,2}(E)^\natural&\GSp_{1,j}(F).
	} \]
Assume that $\mathfrak{W}=\Res_{R/D}V$.
Let us fix the polarization $\mathfrak{W}=\mathfrak{U}+\mathfrak{U}^\ast$ and $\mathbb{H}_E^2=Y+Y^\ast$, where $Y^\ast$ is the dual space of $Y$. Let $\mathfrak{V}$ be a Hermitian $D$-vector space with isometric group $\GSp_{1,j}(F)$. Then
there exists a natural polarization
\[\mathfrak{W}\otimes_D\mathfrak{V}=\mathfrak{U}\otimes_D\mathfrak{V}+\mathfrak{U}^\ast\otimes_D\mathfrak{V}. \]
Similarly, $\mathbb{H}_E^2\otimes_E \mathcal{W}_{1+j}=Y\otimes_E\mathcal{W}_{1+j}+Y^\ast\otimes_E\mathcal{W}_{1+j}, $
where $\mathcal{W}_r$ is the symplectic vector space over $E$ of dimension $2r$. Set $\mathbb{Y}=\mathfrak{U}^\ast\otimes_D \mathfrak{V}$ and $\mathbb{Y}'=Y^\ast\otimes_E\mathcal{W}_{1+j}$.
Then we have the splittting $i_{\mathbb{Y}}$ and $i_{\mathbb{Y}'}$ defined in \cite[Theorem 3.1]{kudla1994splitting}. For instance, $i_{\mathbb{Y}'}(y,z)=((y,z),\beta_{\mathbb{Y}'}(y) )$ for $(y,z)\in\Oo_{2,2}(E)\times\Sp_{2+2j}(E) $ and
\[ i_{\mathbb{Y}}(y,z)=((y,z),\beta_{\mathbb{Y}}(y) )\in \Mp(\mathfrak{W}\otimes_D\mathfrak{V} )_\mathbb{Y}
  \]
  for $y\in \Oo_{2,2}^\ast(F)$ and $z\in\Sp_{1,j}(F)$. Note that $\beta_{\mathbb{Y}'}(y)=1$ for $y\in\Oo_{2,2}(E) $ and
  $$\beta_{\mathbb{Y}}(y)=(-1)^{(1+j)i}$$ if $y\in\mathfrak{P}\omega_i\mathfrak{P} $ where $\Oo^\ast_{2,2}(F)=\bigcup_i \mathfrak{P}\omega_i\mathfrak{P}$ and $\mathfrak{P}$ is the Siegel parabolic subgroup of $\Oo_{2,2}^\ast(F)$.
 Thus
\[\beta_{\mathbb{Y}}(h)=\beta_{\mathbb{Y}'}(h)\cdot (\nu(h))^{1+j} \]
for $h\in\Oo_{2,2}(E)$. Hence
\begin{align*}
\dim\Hom_{\GO_{2,2}(E)^\natural}(\Theta(\varrho),\Sigma)&=\dim\Hom_{\GO_{2,2}(E)^\natural\times\GSp_{1,j}(F) }(\omega_{\psi,\mathbb{Y}},\Sigma\otimes\varrho) \\
&=\dim\Hom_{\GO_{2,2}(E)^\natural\times\GSp_{1,j}(F) }(\omega_{\psi,\mathbb{Y}'},\Sigma\cdot\nu^{1+j}\otimes\varrho  )\\
 &=\dim\Hom_{\GSp_{1,j}(F)}(\Theta_{1+j}(\Sigma\cdot\nu^{1+j}),\varrho ) 
\end{align*}
where $\omega_{\psi,\mathbb{Y}}$ (resp. $\omega_{\psi,\mathbb{Y}'}$) is the Weil representation on $\Mp(\mathfrak{W}\otimes_D\mathfrak{V})$  emphasizing  the splitting $\mathbb{Y}+\mathbb{Y}^\ast$ (resp. $\mathbb{Y}'+{\mathbb{Y}'}^\ast$). This finishes the proof.
\end{proof}

\subsubsection{Degenerate principal series} Let us fix the complete polarization 
\[\mathfrak{V}=Y_D+Y_D^\ast. \]
Suppose $\dim_D\mathfrak{V}=4$.
Assume that  $\mathfrak{I}(s)$ is the degenerated principal series of $\GU(\mathfrak{V})=\GSp_{2,2}(F)$ associated to a Siegel parabolic subgroup $P(Y_D)$, i.e.,
\[\mathfrak{I}(s)=\{f:\GU(\mathfrak{V})\rightarrow \mathbb{C}|~f(pg)=\delta_{P(Y_D)}(p)^{\frac{1}{2}+\frac{s}{5}}f(g)\mbox{ for all } p\in P(Y_D),g\in \GU(\mathfrak{V}) \} \]
where $\delta_{P(Y_D)}$ is the modular character.
Similar to Proposition \ref{degenerateseries}, we have
\begin{lem} \label{GSp:GU}
	Assume that  $\mathfrak{R}^{3}(\mathbf{1})$ is the big theta lift to $\GU(\mathfrak{V})$ of the trivial representation of $\GO_{3,0}^\ast(F)$. Then there is an exact sequence
	\[\xymatrix{0\ar[r]& \mathfrak{R}^{3}(\mathbf{1})\ar[r]& \mathfrak{I}(\frac{1}{2})\ar[r]& \mathfrak{R}^{2}(\mathbf{1})\ar[r]&0 } \]
	where $\mathfrak{R}^{2}(\mathbf{1})$ is the big theta lift to $\GU(\mathfrak{V})$ of the trivial representation
	of $\GO^\ast_{1,1}(F)$. 
\end{lem}
\begin{proof}
	By \cite[Theorem 1.4]{yamana2011deg}, we may give a similar proof as in Proposition \ref{degenerateseries}. So we omit it here.
\end{proof}

	\subsubsection{Double cosets}
Assume that  $P(Y_D)$ is the Siegel parabolic subgroup of $\GU(\mathfrak{V})=\GSp_{2,2}(F)$. Then the homogeneous space
$X_D=P(Y_D)\backslash \GSp_{2,2}(F)$ corresponds to the set of maximal isotropic subspaces in $
\mathfrak{V}$. We consider the double coset $X_D/\GU(W_R)^\natural=X_D/\GSp_4(E)^\natural,$ similar to Lemma \ref{orbitdecomp}.
\begin{prop}
	In the double cosets $X_D/\GSp_4(E)^\natural,$
	there are
	\begin{itemize}
		\item   one closed orbit with  stabilizer
		$P(Y_D)\cap \GSp_4(E)^\natural, $
		\item  one open orbit with  stabilizer $\GU_2(D)(F)=\GSp_{1,1}(F)\subset \GSp_4(E)^\natural$ and
		\item one intermediate orbit with a representative
		\[L=Dr(\sqrt{d}e+f)+D(e-\frac{1}{\sqrt{d}}f )\in X_D,\]
		which is a non-free $R$-module with stabilizer $(\GL_1(E)\times \GSp_{1,0}(F))\cdot N,N\cong E\oplus D$, 	where $r=\begin{pmatrix}
		1&0\\0&0
		\end{pmatrix}=r^2\in R  $ and $W_R=Re+Rf$ with $h_R(e,f)=1$.
	\end{itemize}
\end{prop}

	\begin{lem}\label{degpos1/2}
	Let $\tau$ be an irreducible representation of $\GU(W_R)^\natural=\GSp_4(E)^\natural$ and $\GSp_4(E)^\natural\hookrightarrow \GSp_{2,2}(F) $ be a natural embedding. Then \[\dim \Hom_{\GSp_4(E)^\natural }(\mathfrak{I}(1/2),\tau )\geq \dim \Hom_{\GSp_{1,1}(F)}(\tau^\vee,\mathbb{C} ). \]
\end{lem}
\begin{proof}
	Note that there are three orbits for $P(Y_D) \backslash \GSp_{2,2}(F)/\GSp_4(E)^\natural$.  There is a filtration for
	$\mathfrak{I}(1/2)|_{\GSp_4(E)^\natural}$ as follows
	\[ind_{\GSp_{1,1}(F)}^{\GSp_4(E)^\natural }\mathbb{C}=\mathfrak{I}_0(1/2)\subset\mathfrak{I}_1(1/2)\subset \mathfrak{I}_2(1/2)=\mathfrak{I}(1/2)|_{\GSp_4(E)^\natural } \]
	where $\mathfrak{I}_2(1/2)/\mathfrak{I}_1(1/2)\cong ind_{P^\natural}^{\GSp_4(E)^\natural }\delta_{P^\natural}^{\frac{1}{2}} $ and $\mathfrak{I}_1(1/2)/\mathfrak{I}_0(1/2)\cong ind_{MN}^{\GSp_4(E)^\natural }\delta_{P(Y_D)}^{\frac{3}{5}}\delta_3^{-\frac{1}{2}} ,$ where
	\[M\cong \GL_1(E)\times \GSp_{1,0}(F),N\cong D\oplus E\mbox{  and  }\delta_3(t,x)=|N_{E/F}(t)^4\cdot \lambda(x)^{-4} |_F \]
	for $(t,d)\in M$.
	There exists an involution on $\GSp_{2,2}(F)$ such that the fixed points coincides with $\GSp_4(E)^\natural$. So applying \cite[Theorem 2.5]{olafsson1987fourier} and \cite[Proposition 4.9]{dima2018analytic}, we obtain the following inequality
	\[\dim \Hom_{\GSp_4(E)^\natural}(\mathfrak{I}(1/2),\tau )\geq\dim \Hom_{\GSp_4(E)^\natural}(\mathfrak{I}_0(1/2),\tau )=\dim \Hom_{\GSp_{1,1}(F)}(\tau^\vee,\mathbb{C}). \]
	It finishes the proof.
\end{proof}

%% file: gu2D.tex
\subsection{The distinction problem for $\GSp_{1,1}$}	
Let $\GU_2(D)=\GSp_{1,1}$ be  the inner form of $\GSp_4$ defined over $F,$ whose $E$-points coincide with $\GSp_4(E).$
	Assume that  $\tau$ is an irreducible  representation of $\GSp_4(E)$
	with $\omega_\tau|_{F^\times}=\mathbf{1}$. In this subsection, we will study the multiplicity 
	\[\dim\Hom_{\GSp_{1,1}(F)}(\tau,\mathbb{C} ). \]
	

	\begin{thm}Let $\tau$ be a representation of $\GSp_4(E)$ such that $\Pi_{\phi_{\tau}}$ is generic. Then
		\begin{enumerate}[(i)]
			\item If $\tau=\theta(\pi_1\boxtimes\pi_2)$
			is a nongeneric tempered representation of $\GSp_4(E),$
			where $\pi_1\boxtimes\pi_2$ is an irreducible smooth representation of $\GSO_{4,0}(E),$ then
			$\dim \Hom_{\GSp_{1,1}(F)}(\tau,\mathbb{C})=1$ if and only if one of the following holds:
			\begin{itemize}
				\item $\pi_1\ncong\pi_2$ but $\pi_1^\vee\cong\pi_2^\sigma$;
				\item $\pi_1\cong\pi_2$ are both $(D^\times(F),\omega_{E/F})$-distinguished.
			\end{itemize}
			\item If $\tau=\theta(\pi_1\boxtimes\pi_2)=\theta(\pi_2\boxtimes\pi_1)$ is generic, then
			\[\dim\Hom_{ \GSp_{1,1}(F)}(\tau,\mathbb{C})=\begin{cases}
			2,&\mbox{if }\pi_1=\pi_2=\pi(\chi^{-1},\chi^\sigma);\\
			1,&\mbox{if }\pi_1=\pi_2\mbox{ are square-integrable and }D^\times(F)\mbox{-distinguished};\\
			1,&\mbox{if }\pi_1\mbox{ is }D^\times(F)\mbox{-distinguished and }\pi_2=\pi_0;\\
			2,&\mbox{if }\pi_1\neq\pi_2\mbox{ are both }D^\times(F)\mbox{-distinguished};\\
			0,&\mbox{the other cases}.
			\end{cases} \]
	Here $\pi_0=\pi(\chi_1,\chi_2 )$ with $\chi_1\neq\chi_2,\chi_1|_{F^\times}=\chi_2|_{F^\times}=\mathbf{1}.$	Note that these conditions are mutually exclusive.
			\item Assume that $\tau$ is not as in case (i) or (ii), so that $\tau=\theta(\Pi^{D}\boxtimes\chi )$ is generic, where $\Pi^{D}\boxtimes\chi$ is an irreducible representation of $\GSO_{5,1}(E)$. Then
			$\dim \Hom_{\GSp_{1,1}(F)}(\tau,\mathbb{C})=1$ if and only if one of the following holds:
			\begin{itemize}
				\item $\phi_\Pi$ is irreducible and conjugate-orthogonal or
				\item $\phi_\Pi=\phi_\rho+\phi_\rho\mu$ with $\rho^\sigma\cong \rho^\vee\mu^{-1}$
			\end{itemize}   where $\Pi=JL(\Pi^{D})$ is the Jacquet-Langlands lift to $\GL_4(E)$ of $\Pi^D$. 
		\end{enumerate}\label{innerformperiod}
		\end{thm}
		\begin{proof}
			The proof is very similar to the proof of Theorem \ref{localgspperiod}.
			\begin{enumerate}[(i)]
				\item Assume that  $V_R$ is a  skew-Hermitian free module over $R$ of rank $2$,  corresponding to $D_E$ by the Morita equivalence.
				Then $\Res_{R/D}V_R$ is a $4$-dimensional skew-Hermitian vector space over $D$ with trivial discriminant. Fix a polorization $\Res_{R/D}V=\mathfrak{U}\oplus \mathfrak{U}^\ast$.
				Consider the following diagram
				\[\xymatrix{\GSp_4(E)^\natural\ar@{-}[d]\ar@{-}[rd] & \GO_{2,2}^\ast(F)\ar@{-}[ld]\ar@{-}[d]\ar@{-}[rd] & \GL_2(E)^\natural\ar@{-}[ld]\ar@{-}[d] \\
					\GSp_{1,1}(F)& \GO_{4,0}(E)^\natural& \GSp_{1,0}(F). } \]
				There is an exact sequence of $\GO_{2,2}^\ast(F)$-representations
				\[\xymatrix{0\ar[r]& \mathfrak{R}^{1,1}(\mathbf{1})\ar[r]& I(\frac{1}{2})\ar[r]& \mathfrak{R}^{1,0}(\mathbf{1})\ar[r]&0, } \]
			where $I(s)$ is the degenerate principal series of $\GO^\ast_{2,2}(F)$ and $\mathfrak{R}^{1,j}(\mathbf{1})$ is the theta lift to $\GO_{2,2}^\ast(F)$ the trivial representation  of $\GSp_{1,j}(F)$.  Set $\tau=\Theta_{2}(\Sigma^+)$, where
		\[\Sigma^+=\begin{cases}
		\Ind_{\GSO_{4,0}(E)}^{\GO_{4,0}(E)}(\pi_1\boxtimes\pi_2),&\mbox{ if }\pi_1\ncong\pi_2;\\
		(\pi_1\boxtimes\pi_1)^+,&\mbox{ if }\pi_1\cong\pi_2.
		\end{cases} \]	
		Note that $\GO_{4,0}(E)$ is an anisotropic group.	Using the contravariant exact functor $$\Hom_{\GO_{4,0}(E)^\natural}(-,\Sigma^+),$$ we obtain the short exact sequence
			\[0\rightarrow\Hom_{\GO_{4,0}(E)^\natural}(\mathfrak{R}^{1,0}(\mathbf{1} ),\Sigma^+)\rightarrow \Hom_{\GO_{4,0}(E)^\natural}(I(\frac{1}{2}),\Sigma^+ )\rightarrow \Hom_{\GO_{4,0}(E)^\natural}(\mathfrak{R}^{1,1}(\mathbf{1}),\Sigma^+ )\rightarrow0 .\]
			With the see-saw identities, we have 
			\begin{equation}\label{nongenericquaternion}
		0\rightarrow	\Hom_{\GSp_{1,0}(F)}(\Theta_1(\Sigma^+\otimes\nu),\mathbb{C} )\rightarrow \Hom_{\GO_{4,0}(E)^\natural}(I(\frac{1}{2}),\Sigma^+)\rightarrow \Hom_{\GSp_{1,1}(F)}(\tau,\mathbb{C})\rightarrow 0 ,
			\end{equation}
			where
			$\Theta_{1}(\Sigma^+\otimes\nu)$ is the big theta lift to $\GL_2(E)$ of $\Sigma^+\otimes\nu$.
				There is no $F$-rational points on the non-identity connected component of $\GO_{2,2}^\ast$ \cite[Page 21-22]{moeglin1987correspondances}, so that $$\GO^\ast_{2,2}(F)=\GSO^\ast_{2,2}(F)=\mathfrak{Q}\cdot \GO_{4,0}(E)^\natural,$$
where $\mathfrak{Q}$ is the Siegel parabolic subgroup of $\GO_{2,2}^\ast(F)$. Then
				\begin{equation}\label{quaternionicopenorbit}
				\Hom_{\GO_{4,0}(E)^\natural}(I(\frac{1}{2}),\Sigma^+ )=\Hom_{\GO_{4,0}(E)^\natural}(ind_{\GO^\ast_{2,0}(F)}^{\GO_{4,0}(E)^\natural }\mathbb{C},\Sigma^+ )=\Hom_{\GO^\ast_{2,0}(F)}(\Sigma^+,\mathbb{C}) .\end{equation}
				Here $\GSO^\ast_{2,0}(F)$ sits in the following exact sequence
				\[\xymatrix{1\ar[r]& E^\times\ar@{=}[d]\ar[r]^-i& D_E^\times(E) \times F^\times\ar[r]\ar[d]&\GSO^\ast_{2,0}(F)\ar@{^{(}->}[d]\ar[r]&1\\
				1\ar[r]&E^\times\ar[r]& D_E^\times(E)\times D_E^\times(E)\ar[r]&\GSO_{4,0}(E)\ar[r]&1 } \]
				where $i(e)=(e,N_{E/F}(e)^{-1})$ and the embedding $\GSO^\ast_{2,0}(F)\hookrightarrow \GSO_{4,0}(E)$ is given by $$(x,t)\mapsto(x,t\cdot x^\sigma )$$
				for $x\in D_E^\times(E)$ and $t\in F^\times$.
				The $\sigma$-action on $D_E^\times(E)$ is induced from the isomorphism $D_E(E)\cong D_E(E)\otimes_E(E,\sigma).$ There are two subcases:
			\begin{itemize}
				\item If $\pi_1\ncong\pi_2,$ then $\pi_1\boxtimes\pi_2$ does not participate in theta correspondence with $\GL_2(E)$. The short exact sequence \eqref{nongenericquaternion} implies that
				\begin{equation}\label{nonsplitquaternionperiod}
			\dim	\Hom_{\GSp_{1,1}(F)}(\tau,\mathbb{C})=\dim \Hom_{\GO_{4,0}(E)^\natural}(I(\frac{1}{2}),\Sigma^+)=\dim\Hom_{\GSO_{2,0}^\ast(F)}(\pi_1\boxtimes\pi_2,\mathbb{C}) . 
				\end{equation}
				Hence one can  get 
				 $$\dim \Hom_{\GSp_{1,1}(F)}(\tau,\mathbb{C})=\dim \Hom_{D_E^\times(E)}(\pi_2^\vee,\pi_1^\sigma )$$ where
				$\pi_1^\sigma=JL^{-1}(JL(\pi_1)^\sigma). $
				\item If $\pi_1=\pi_2,$ then the short exact sequence \eqref{nongenericquaternion} implies that 
				\[\dim \Hom_{\GSp_{1,1}(F)}(\tau,\mathbb{C})=\dim \Hom_{\GO_{4,0}(E)^\natural}(I(\frac{1}{2}),\Sigma^+) \]
				because $\Theta_{1}(\Sigma^+\otimes\nu)=0$. Note that
				\[\dim\Hom_{ \GSO^\ast_{2,0}(F)}(\pi_1\boxtimes\pi_1,\mathbb{C} )=\dim\Hom_{\GO_{2,0}^\ast(F)}(\Sigma^+,\mathbb{C})+\dim\Hom_{ \GO^\ast_{2,0}(F)}(\Sigma^+\otimes\nu,\mathbb{C}) . \]
				In a similar way, we have $\dim\Hom_{\GO_{2,0}^\ast(F)}(\Sigma^+\otimes\nu,\mathbb{C} )=\dim\Hom_{\GSp_{1,0}(F)}(\pi_1,\mathbb{C}).  $ Therefore, if $\pi_1$ is $D^\times(F)$-distinguished, then $\pi_1^\sigma\cong\pi_1^\vee $ and so 
				\[\dim\Hom_{\GSO_{2,0}^\ast(F)}(\pi_1\boxtimes\pi_1,\mathbb{C} )=1=\dim\Hom_{\GO_{2,0}^\ast(F)}(\Sigma^+\otimes\nu,\mathbb{C} ) .  \]
				Then $\dim\Hom_{\GSp_{1,1}(F)}(\tau,\mathbb{C})=\dim\Hom_{ \GO^\ast_{2,0}(F)}(\Sigma^+,\mathbb{C})=0 $ if $\pi_1$ is $D^\times(F)$-distinguished.
				Furthermore, $\tau$ is $\GSp_{1,1}(F)$-distinguished if and only if $\pi_1^\vee\cong\pi_1^\sigma$ which is not $D^\times(F)$-distinguished. Thus $\tau$ is $\GSp_{1,1}(F)$-distinguished if and only if $\pi_1$ is $(D^\times(F),\omega_{E/F})$-distinguished, in which case $\phi_{\pi_1}$ is conjugate-symplectic.
			\end{itemize}
			(Similarly, one can show that $$\dim \Hom_{\GSp_{1,1}(F)}(\tau,\omega_{E/F})=\dim \Hom_{D_E^\times(E)}(\pi_2^\vee,\pi_1^\sigma)-\dim \Hom_{D^\times(F)}(\Theta_1(\Sigma^+\otimes\nu),\omega_{E/F}).$$
				Here we use the fact $$\omega_{E/F}\circ\lambda_V |_{\GO^\ast_{2,0}(F)}=\mathbf{1}.$$ Hence $\dim \Hom_{\GSp_{1,1}(F)}(\tau,\omega_{E/F})=1$ if and only if either $\pi_1=\pi_2$ are both $D^\times(F)$-distinguished or $\pi_1\ncong\pi_2$ but
				 $\pi_1^\vee=\pi_2^\sigma$. It will be useful when we verify the Prasad conjecture for $\PGSp_4$ in \S\ref{secpgsp}.)
				\item We will use a similar argument. Assume that $V_R$ corresponds to $\mathbb{H}_E^2$ by the Morita equivalence. Then via the see-saw diagrams
				\[\xymatrix{\GO_{5,1}(E)^\natural\ar@{-}[d]\ar@{-}[rd] & \GSp_{2,2}(F)\ar@{-}[ld]\ar@{-}[d]\ar@{-}[rd] & \GO_{2,2}(E)^\natural\ar@{-}[ld]\ar@{-}[d] \\ \GO^\ast_{3,0}(F)& \GSp_4(E)^\natural&\GO^\ast_{1,1}(F) } \]
				we have $\theta_2^-(\tau)=0$. So Lemma \ref{GSp:GU} implies that
				\[\dim \Hom_{\GSp_4(E)^\natural}(\mathfrak{I}(\frac{1}{2}),\tau)=\dim\Hom_{\GSp_4(E)^\natural}(\mathfrak{R}^2(\mathbf{1}),\tau) = \dim \Hom_{\GO_{1,1}^\ast(F)}(\Theta_{4}^+(\tau),\mathbb{C} ) ,\] where $\mathfrak{I}(s)$ is the degenerate principal series of $\GSp_{2,2}(F)$.  
Due to Lemma \ref{degpos1/2}, 
 $$\dim \Hom_{\GSp_{1,1}(F)}(\tau,\mathbb{C})\leq\dim\Hom_{\GO_{1,1}^\ast(F)}(\Theta_{4}^+(\tau),\mathbb{C}) .$$
We want to get the reverse inequality.
Consider the following diagrams
\[\xymatrix{ \GSp_4(E)^\natural\ar@{-}[d]\ar@{-}[rd] & \GO_{2,2}^\ast(F)\ar@{-}[rd]\ar@{-}[d]\ar@{-}[ld]& \GL_2(E)^\natural\ar@{-}[d]\ar@{-}[ld]
	\\ \GSp_{1,1}(F)& \GO_{2,2}(E)^\natural& \GSp_{1,0}(F).
} \]
	There is an exact sequence of $\GO_{2,2}^\ast(F)$-representations
\[\xymatrix{0\ar[r]& \mathfrak{R}^{1,0}(\mathbf{1})\ar[r]& I(-\frac{1}{2})\ar[r]& \mathfrak{R}^{1,1}(\mathbf{1})\ar[r]&0. } \]
Note that $\dim\Hom_{\GO_{2,2}(E)^\natural}(\mathfrak{R}^{1,0}(\mathbf{1}),\Sigma^+ )=\dim\Hom_{\GSp_{1,0}(F)}(\Theta_{1}(\Sigma^+\otimes\nu),\mathbb{C})=0 .  $
Thanks to \cite[Theorem 2.5]{olafsson1987fourier} and \cite[Proposition 4.9]{dima2018analytic}, we have
\begin{align*}
\dim\Hom_{\GSp_{1,1}(F)}(\tau,\mathbb{C})&=\dim\Hom_{\GO_{2,2}(E)^\natural}(\mathfrak{R}^{1,1}(\mathbf{1}),\Sigma^+)\\  
&=\dim\Hom_{\GO_{2,2}(E)^\natural}(I(-\frac{1}{2}),\Sigma^+)\\ 
&\geq\dim\Hom_{\GO_{2,2}(E)^\natural}(ind_{\GO_{1,1}^\ast(F)}^{\GO_{2,2}(E)^\natural}\mathbb{C}, \Sigma^+ )\\
&=\dim\Hom_{\GO_{1,1}^\ast(F)}(\Sigma^+,\mathbb{C}).   
\end{align*} 
		Therefore $\dim\Hom_{\GSp_4(F)}(\tau,\mathbb{C})=\dim\Hom_{\GO_{1,1}^\ast(F)}(\Theta_4^+(\tau),\mathbb{C}) $ unless $\Theta_{4}^+(\tau)$ is reducible.
		There is no $F$-rational points on the non-identity connected component of $\GO_{1,1}^\ast$ \cite[Page 21-22]{moeglin1987correspondances}, so that
		\[\GO_{1,1}^\ast(F)=\GSO_{1,1}^\ast(F). \]
			Assume that $\pi_1\ncong\pi_2$.	Since $$\GO_{1,1}^\ast(F)=\GSO^\ast_{1,1}(F)\cong\frac{ \GL_2(F)\times D^\times(F)}{\{(t,t^{-1}) \}},$$ one can obtain that for $\pi_1\neq\pi_2$, $\Theta_{4}^+(\tau)=\Ind_{\GSO(2,2)(E)}^{\GO(2,2)(E)}(\pi_1\boxtimes\pi_2)$ and
				\begin{equation}\label{quaternionsum}
				\Hom_{\GO^\ast_{1,1}(F) }(\Sigma^+,\mathbb{C} )=\Hom_{\GO^\ast_{1,1}(F)}(\pi_1\boxtimes\pi_2,\mathbb{C}) \oplus \Hom_{\GO_{1,1}^\ast(F)}(\pi_2\boxtimes\pi_1,\mathbb{C} ).
				\end{equation} 
		There are two subcases:
			\begin{itemize}
				\item 	If $\pi_i~(i=1,2)$ are both $D^\times(F)$-distinguished, then \eqref{quaternionsum} implies that
				\[\dim \Hom_{\GSp_{1,1}(F)}(\tau,\mathbb{C})=\dim \Hom_{\GO^\ast_{1,1}(F)}(\Sigma^+,\mathbb{C} )=2. \]
			\item 	If $\pi_1$ is $D^\times(F)$-distinguished and $\pi_2=\pi(\chi_1,\chi_2)$ with $\chi_1\neq\chi_2,\chi_1|_{F^\times}=\chi_2|_{F^\times}=1,$ then $\pi_2$ is $\GL_2(F)$-distinguished but not $D^\times(F)$-distinguished. So \eqref{quaternionsum} implies that
				\[\dim \Hom_{\GSp_{1,1}(F)}(\tau,\mathbb{C})=1. \]
			\end{itemize}
				If  $\pi_1=\pi_2$ are both square-integrable representations, then \[\Hom_{\GO^\ast_{1,1}(F)}(\Sigma^+,\mathbb{C})=\Hom_{\GSO^\ast_{1,1}(F)}(\pi_1\boxtimes\pi_1,\mathbb{C})=\begin{cases}
				1,&\mbox{ if }\pi_1\mbox{ is }D^\times(F)\mbox{-distinguished};\\
				0,&\mbox{ otherwise.}
				\end{cases} \]

				If $\pi_1=\pi_2=\pi(\chi^{-1},\chi^\sigma),$ then $\Theta_{4}^+(\tau)$ is reducible. We will show that $\tau=I_{Q(Z)}(\mathbf{1},\pi_1)$ does not occure on the boundary of $\mathfrak{I}(\frac{1}{2})$ and so that
				$$\dim\Hom_{\GSp_{1,1}(F)}(\tau,\mathbb{C})=\dim\Hom_{\GO_{1,1}^\ast(F)}(\Theta_{4}^+(\tau),\mathbb{C}).  $$
		There is a filtration 
		\[ind_{\GSp_{1,1}(F)}^{\GSp_4(E)^\natural}\mathbb{C}=\mathfrak{I}_0(s)\subset\mathfrak{I}_1(s)\subset\mathfrak{I}_2(s)=\mathfrak{I}(s)|_{\GSp_4(E)^\natural} \]
		of $\mathfrak{I}(s)|_{\GSp_4(E)^\natural }$		such that $\mathfrak{I}_2(s)/\mathfrak{I}_1(s)=ind_{P^\natural}^{\GSp_4(E)^\natural}\delta_{P^\natural}^{\frac{s+1}{3}}$ and $$\mathfrak{I}_1(s)/\mathfrak{I}_0(s)=ind_{MN}^{\GSp_4(E)^\natural}\delta_{P(Y_D)}^{\frac{1}{2}+\frac{s}{5}}\delta_3^{-\frac{1}{2}}$$
				where $\delta_3(t,x)=|N_{E/F}(t)^4\lambda(d)^{-4} |_F$ for $(t,x)\in M=\GL_1(E)\times\GSp_{1,0}(F)$. If
				\[\Hom_{\GSp_4(E)^\natural}(\mathfrak{I}_1(\frac{1}{2})/\mathfrak{I}_0(\frac{1}{2}),\tau )
				\neq0,\]
				then
				\[\Hom_{\GL_1(E) }(|-|_E,R_{\bar{P}''}(I_{Q(Z)}(\mathbf{1},\pi_1)) )\neq 0,  \]
				which is impossible,
		where $P''=(\GL_1(E)\times\GL_2(E)^\natural)\ltimes N$ is a parabolic subgroup of $\GSp_4(E)^\natural$ and $R_{\bar{P}''}$ denotes the Jacquet funtor associate to the parabolic opposite to $P''$. So 	$$\Hom_{\GSp_4(E)^\natural}(\mathfrak{I}_1(\frac{1}{2})/\mathfrak{I}_0(\frac{1}{2}),\tau )
		=0.$$
			It is quite straightforward to see that $$\Hom_{\GSp_4(E)^\natural}(ind_{P^\natural}^{\GSp_4(E)^\natural}\delta_{P^\natural}^{\frac{1}{2}},I_{Q(Z)}(\mathbf{1},\pi_1))=0.$$ Hence $\tau=I_{Q(Z)}(\mathbf{1},\pi_1)$ does not occur on the boundary of $\mathfrak{I}(\frac{1}{2})$.
			\par 
				 The big theta lift to  $\GSO_{2,2}(E)$ of $\tau$ of $\GSp_4(E)$  is $$\Ext^1_{\GSO(2,2)(E)}(\pi_1\boxtimes\pi_1,\pi_1\boxtimes\pi_1 ).$$
				From the following see-saw pairs diagram
					\[\xymatrix{\GSO_{5,1}(E)^\natural\ar@{-}[d]\ar@{-}[rd] & \GSp_{2,2}(F)\ar@{-}[ld]\ar@{-}[d]\ar@{-}[rd] & \GSO_{2,2}(E)^\natural\ar@{-}[ld]\ar@{-}[d] \\ \GO^\ast_{3,0}(F)& \GSp_4(E)^\natural&\GO^\ast_{1,1}(F) } \]
					one can use the fact $\theta_2^{-}(\tau)=0$ to obtain that
				\[\dim\Hom_{\GSp_{1,1}(F)}(\tau,\mathbb{C})=\dim\Hom_{\GSO_{1,1}^\ast(F)}(\Ext_{\GSO_{2,2}(E)}^1(\pi_1\boxtimes\pi_1,\pi_1\boxtimes\pi_1 ),\mathbb{C} )=2. \]
				\item Assume that $\theta_4^{+}(\tau)=0$.
				Note that $0\rightarrow\mathfrak{R}^2(\mathbf{1})\rightarrow \mathfrak{I}(-\frac{1}{2})\rightarrow \mathfrak{R}^3(\mathbf{1})\rightarrow0$ is exact. Then we can use the same method appearing in (ii) to show that
				\begin{align*}\dim \Hom_{\GO_{3,0}^\ast(F) }(\Theta_2^-(\tau),\mathbb{C} )&=\dim \Hom_{\GSp_4(E)^\natural}(\mathfrak{R}^{3}(\mathbf{1}),\tau)\\
				&=\dim \Hom_{\GSp_{4}(E)^\natural}(\mathfrak{I}(-\frac{1}{2}),\tau)\\
				&\geq\dim\Hom_{\GSp_{1,1}(F)}(\tau,\mathbb{C}) . \end{align*}
				We will show that $\tau$ does not occur on the boundary of $\mathfrak{I}(-\frac{1}{2})$ in this case. Then $$\dim\Hom_{\GSp_4(E)^\natural}(\mathfrak{I}(-\frac{1}{2}),\tau)\leq \dim\Hom_{\GSp_4(E)^\natural}(\mathfrak{I}_0(-\frac{1}{2}),\tau ) =\dim\Hom_{\GSp_{1,1}(F)}(\tau,\mathbb{C})  $$ and so
				\[\dim\Hom_{ \GSp_{1,1}(F)}(\tau,\mathbb{C})=\dim\Hom_{\GSp_4(E)^\natural}(\mathfrak{I}(-\frac{1}{2}),\tau).  \]
			In order to show that $\tau$ does not occur on the boundary of $\mathfrak{I}(-\frac{1}{2})$,	we separate them into two cases.
				\begin{itemize}
					\item If $\tau=I_{Q(Z)}(\chi,\pi)$ with $\chi\neq\mathbf{1}$,
					then $$\Hom_{\GSp_{4}(E)^\natural}(\mathfrak{I}_2(-\frac{1}{2})/\mathfrak{I}_1(-\frac{1}{2}),\tau )=\Hom_{\GSp_4(E)^\natural}(ind_{P^\natural}^{\GSp_4(E)^\natural}\delta_{P^\natural}^{\frac{1}{6}},\tau ) =0. $$
					If $\Hom_{\GSp_4(E)^\natural}(\mathfrak{I}_1(-\frac{1}{2})/\mathfrak{I}_0(-\frac{1}{2}),\tau )\neq 0 $, then
					$\Hom_{\GL_1(E)}(\mathbf{1},R_{\bar{P}''}(\tau) )\neq0 $ which is impossible since
					$R_{\bar{P}''}(\tau)=\chi\otimes\pi\oplus \chi^{-1}\otimes\pi\chi $ and $\chi\neq\mathbf{1}$, where $P''=(\GL_1(E)\times \GL_2(E)^\natural )\rtimes N$.
					\item If $\pi$ is square-integrable, then $\Hom_{\GL_1(E)}(\mathbf{1},R_{\bar{P}''}(\tau) )=0 $ due to the Casselman criterion in \cite{casselman82duke} for a discrete series representation that $\Hom_{\GL_1(E)}(|-|^s_E,R_{\bar{P}''}(\tau))\neq0$ implies that $s<0$. Hence
					$\Hom_{\GSp_4(E)^\natural}(\mathfrak{I}_1(-\frac{1}{2})/\mathfrak{I}_0(-\frac{1}{2}),\tau)=0 $. In a similar way,
					\[\Hom_{\GSp_{4}(E)^\natural}(\mathfrak{I}_2(-\frac{1}{2})/\mathfrak{I}_1(-\frac{1}{2}),\tau )=\Hom_{\GL_2(E)\times F^\times }(\delta_{P^\natural}^{\frac{1}{6}},R_{\bar{P}^\natural}(\tau))=0.  \]
				\end{itemize}
			Hence $\tau$ does not occur on the boundary of $\mathfrak{I}(-\frac{1}{6})$.
			Moreover, $\Theta_{2}^-(\tau)=\Pi^D\boxtimes\chi$ is irreducible.
				Then there exists an identity
				\[\dim \Hom_{\GSp_{1,1}(F) }(\tau,\mathbb{C})=\dim \Hom_{\GO^\ast_{3,0}(F)}(\Pi^{D}\boxtimes\chi,\mathbb{C})=\dim \Hom_{D_4^\times(F)}(\Pi^{D},\mathbb{C}), \]
				where $D_4$ is the division algebra over $F$ with invariance  $\frac{1}{4}\in\mathbb{Q}/\mathbb{Z}$. So
					\begin{itemize}
						\item
				 If $\Pi=JL(\Pi^D)$ is a square-integrable representation of $\GL_4(E)$, then \cite[Theorem 1]{beuzart2017distinguished} and \cite[Theorem 5.2]{matringe2009distinction} imply that
				\[\dim\Hom_{\GL_4(F)}(\Pi,\omega_{E/F})=\dim\Hom_{D_4^\times(F)}(\Pi^D,\omega_{E/F}) =\begin{cases}
				1,&\mbox{\mbox{ if }}\phi_{\Pi}\mbox{ is conjugate-symplectic};\\
				0,&\mbox{ otherwise.}
				\end{cases} \]
So $\dim\Hom_{D_4^\times(F)}(\pi^D,\mathbb{C})=1$ if and only if $\phi_{\Pi}$ is conjugate-orthogonal.			

				\item 	 If $\Pi^{D}$ is an induced representation $\pi(\rho_D,{(\rho_D)}^\vee\otimes\mu)$ with $\mu\neq\omega_{\rho_D}$, then we use the orbit decomposition $B_1\backslash \GL_2(D_E)(E)/\GL_1(D_4)(F)$ and Mackey theory to get that
					\[\Hom_{D_4^\times(F)}(\pi^{D},\mathbb{C} )=\Hom_{D^\times_E(E) }(\rho_D^\sigma\otimes\rho_D^\vee\cdot \mu,\mathbb{C} )=\Hom_{D_E^\times(E)}(\rho_D^\sigma,\rho_D\cdot\mu^{-1})=\begin{cases}
					1,&\mbox{ if }\rho_D^\sigma\cong \rho_D\mu^{-1};\\
					0,&\mbox{ otherwise.}
					\end{cases} \]
			In this case,	 $\rho^\sigma=\rho\mu^{-1}$ where $\rho=JL(\rho_D)$ is the Jacquet-Langlands lift to $\GL_2(E)$
					and $\phi_\Pi=\phi_\rho\oplus\phi_{\rho}^\vee\cdot\mu,$
					which is 
					conjugate-orthogonal due to \cite[Theorem 5.2]{matringe2016distinction}.
				\end{itemize}
			\end{enumerate}
			Then we are done.
		\end{proof}